\documentclass{article}
\usepackage[margin=1in]{geometry}

\usepackage{wrapfig}

\usepackage[utf8]{inputenc} % allow utf-8 input
\usepackage[T1]{fontenc}    % use 8-bit T1 fonts
\usepackage{hyperref}
\hypersetup{
    colorlinks=true,
    citecolor=LightSeaGreen,
    linkcolor=Peach,      
    urlcolor=cyan
    }
\usepackage{url}            % simple URL typesetting
\usepackage{booktabs}       % professional-quality tables
\usepackage{amsfonts}       % blackboard math symbols
\usepackage{nicefrac}       % compact symbols for 1/2, etc.
\usepackage{microtype}      % microtypography
\usepackage[T1]{fontenc}
\usepackage{lmodern} 
\usepackage{comment}
\usepackage{amsmath, amssymb, amsthm}
\usepackage{mathtools}
\usepackage[dvipsnames,svgnames]{xcolor}
\usepackage{dsfont}
\newcommand{\ind}{\mathds{1}}
\usepackage{subcaption}

\usepackage[english]{babel} 
\usepackage[round]{natbib}

\usepackage{enumitem}
\setlist[itemize]{label=-}

\usepackage{pifont}

\newcommand{\E}{\mathbb{E}}
\newcommand{\Tr}{\text{tr}}
\newcommand{\Norm}[1]{\left\lVert#1\right\rVert}
\newcommand{\Inner}[2]{\langle #1, #2 \rangle}

\newcommand{\eqdef}{\stackrel{\text{\rm\tiny def}}{=}}

\newtheorem{lemma}{Lemma}
\newtheorem{proposition}{Proposition}
\newtheorem{corollary}{Corollary}
\theoremstyle{definition}
\newtheorem{definition}{Definition}
\theoremstyle{remark}
\newtheorem{remark}{Remark}

\title{Attention-based PCA}

\author{Rodrigo Maulen-Soto\textsuperscript{\ding{171}} \quad  \& \quad  Claire Boyer\textsuperscript{\ding{169} \ding{168}}}
\date{\today}

\begin{document}

\maketitle

\begin{small}
\noindent\textsuperscript{\ding{171}} Sorbonne Université, CNRS, Laboratoire de Probabilités, Statistique et Modélisation, 75005, Paris, France \\\textsuperscript{\ding{169}} Université Paris-Saclay, CNRS, Inria, Laboratoire de mathématiques d'Orsay, 91405, Orsay, France \\
    \textsuperscript{\ding{168}} Institut Universitaire de France
\end{small}

\begin{abstract}
We study attention mechanisms through the lens of a canonical unsupervised problem: principal component analysis (PCA). We show that, when trained on Gaussian data, both softmax and linear attention layers learn parameters that align with the principal eigenvectors of the covariance matrix, thereby establishing a direct and explicit connection with PCA.

Our analysis covers both finite and infinite prompt regimes. In the infinite-prompt limit, we prove convergence to globally optimal solutions aligned with the leading spectral direction, while in the finite-prompt setting we show that the same behavior emerges up to sampling effects. We further extend the analysis to an in-context setting with spiked Wishart covariances, where attention successfully recovers the underlying signal direction.

These results demonstrate that attention inherently performs PCA-like computations under unsupervised objectives, providing a theoretical foundation for its representation-learning capabilities.
\end{abstract}

\section{Introduction}

Attention-based models \citep{bahdanau2015}, in particular Transformers \citep{vaswani}, have become central in modern machine learning, achieving state-of-the-art results in natural language processing \citep{devlin,bubeck,luong,bahdanau2016} and computer vision \citep{dosovitskiy,liu,stand}. The core attention mechanism computes weighted combinations of token representations based on pairwise interactions, allowing the model to capture long-range dependencies without necessarily relying on fixed positional locality.

A full theoretical understanding of attention-based mechanisms remains incomplete, due to both the architectural complexity and the diversity of tasks they successfully address. A promising research direction toward bridging this gap is to identify key features of real-world problems and study minimal, canonical tasks that retain their core statistical structure, while remaining amenable to rigorous analysis. Notable recent efforts in this direction include \citet{ahn,oswald,yang,zhang2024trained,li2024one,li2023}. However, most existing work focuses on supervised settings, particularly in-context learning \citep{oswald,zhang2024trained,garg,li2023,furuya}, where the goal is to predict the output corresponding to a new query given a prompt consisting of input–output pairs. By contrast, only limited attention has been paid to unsupervised settings, with \citet{clustering} providing one of the few studies exploring the behavior of attention layers in tasks such as clustering.

In this paper, we contribute to this line of research by studying attention layers in an unsupervised setting through the lens of Principal Component Analysis \citep[PCA;][]{pearson1901,hotelling1933,jolliffe2002principal,golub1996}. 
PCA underlies many approaches to dimensionality reduction and feature extraction in statistical learning by relying on the estimation of principal components, i.e., the leading eigenvectors of the data covariance matrix. Understanding how attention layers can extract these spectral directions sheds light on the representation-learning capabilities of Transformers under unsupervised objectives.

\paragraph{Contributions.}
We consider simplified softmax and linear attention layers with rank-one attention parameters. The input tokens are assumed to be i.i.d.\ samples from a multivariate Gaussian distribution 
$\mathcal{N}(0,\Sigma)$. 
Our main contributions can be then summarized as follows.

% \begin{itemize}
%     \item \textbf{Attention performs PCA.}
%     We show that both softmax and linear attention layers can be trained to recover the top eigenvector of the data covariance matrix. This establishes a direct link between attention dynamics and classical spectral methods.

%     \item \textbf{Softmax attention: finite vs.\ infinite prompts.}
%     We analyze a simplified softmax attention model in both finite and infinite prompt regimes. In the infinite prompt limit, we prove convergence to a global minimizer aligned with the principal eigenvector. In the finite prompt setting, we show that the dynamics converge toward the same solution, with deviations explained by sampling effects.

%     \item \textbf{In-context learning with Wishart covariances.}
%     We extend the analysis to a distributional setting where the covariance matrix follows a spiked Wishart model. We show that attention recovers the spike direction both in the population limit and under finite-prompt approximations.

%     \item \textbf{Linear attention exhibits the same phenomenon.}
%     We introduce a linear attention model that admits an explicit analysis for finite prompt lengths and prove that it recovers the same spectral structure. This demonstrates that the emergence of PCA-like behavior is not specific to softmax.%{\color{red}is it a contrib?}
% \end{itemize}
\begin{itemize}
    \item \textbf{A tractable linear attention model.}
    As a warm-up, we provide an explicit analysis for a linear attention model. We show that attention parameters can be trained to recover the top eigenvector of the data covariance, building intuition for the mechanisms at play (full details are provided in the appendix).
     \item \textbf{Softmax attention: from infinite to finite prompts.}
     As for the softmax model, the training dynamics are more delicate to handle due to the nonlinearity of the softmax. We first analyze the tractable infinite-prompt limit, characterizing a global minimizer aligned with the top eigenvector, and then transfer this understanding to the finite-prompt regime. This requires refined concentration arguments, controlling both the optimization landscape (through concentration of critical points) and the training dynamics, ultimately showing concentration around the same solution up to sampling effects.
    % As for the softmax model, the training dynamics are more delicate to handle due to the nonlinearity of the softmax.    We first analyze the tractable infinite-prompt limit, characterizing a global minimizer aligned with the principal eigenvector, and then transfer this understanding to the finite-prompt regime, showing concentration around the same solution up to sampling effects.

\item \textbf{In-context learning under structured covariance models.}
    We extend the analysis to a distributional setting where the covariance follows a spiked Wishart model. In this setting, we show that attention recovers the spike direction both in finite and infinite prompt regimes.
    \item \textbf{Attention performs PCA.}
    Taken together, our results show that rank-one attention layers can be trained in an unsupervised manner to recover the principal component of the input tokens. 
    Under this Gaussian setting, attention-based architectures are shown to detect  the underlying structure of the data.  This provides a clear theoretical connection between attention mechanism and classical spectral methods, positioning attention as an implicit, optimization-driven analogue of PCA.
\end{itemize}

\paragraph{Organization.}
 Section \ref{sec:softmax_version} introduces the problem and studies the convergence of a simplified softmax attention layer in both finite and infinite prompt settings. Section \ref{sec:dist} characterizes the distribution of the resulting encodings and their relation across both regimes. In Section \ref{sec:icl}, we extend the framework to an in-context PCA learning setting.
 The appendices gather results for a simplified linear attention model recovering the principal component, along with proofs of the main results, technical lemmas, and additional numerical experiments.

\section{Training dynamics of a softmax attention layer}
\label{sec:softmax_version}

\subsection{Rank-one softmax attention: model and risk functions}

\paragraph{Setting.}\label{setting}

Let $\mathbb{X}=(X_1,\ldots,X_L)\in\mathbb{R}^{d\times L}$ denote an input prompt made of Gaussian tokens where $X_\ell \overset{\rm i.i.d.}{\sim}\mathcal{N}(0, \Sigma)$, with $\Sigma\in \mathbb{R}^{d\times d}$ a symmetric and definite positive matrix. We consider a simplified softmax attention head acting on such a prompt, defined for $1\leq \ell \leq L$, by
\begin{equation}\label{tlsoft}
    T_L^{\mathrm{soft},\mu}(\mathbb{X})_\ell=\sum_{k=1}^L \mathrm{softmax}(\lambda X_\ell^\top\mu\mu^\top X_k)X_k,
\end{equation}
the $\mathrm{softmax}$ function being applied row-wise.
In this simplified architecture, 
the vector $\mu\in \mathbb{R}^d$ denotes the only attention parameter. This formulation arises from standard architectures where the value matrix is taken to be the identity, and the key and query matrices reduce to row vectors. 
This induces the rank-one structure $K^\top Q=\mu\mu^\top$, which restricts the interaction mechanism of the attention layer and simplifies the analysis.
 We focus in the main text on the softmax attention head, whose theoretical analysis is more delicate due to the softmax nonlinearity, and provide analogous results for a simplified linear counterpart in the appendices.

In order to measure the quality of embedding performed by an attention layer, we consider the following theoretical population risk
\begin{align}\label{rsoftl}
    \mathcal{R}_{\mathrm{soft},L}(\mu)=
    \frac{1}{L}\sum_{\ell=1}^L
    \mathbb{E}\left[\Vert X_\ell-T_L^{\mathrm{soft},\mu}(\mathbb{X})_\ell\Vert^2\right]=\mathbb{E}\left[\Vert X_1-T_L^{\mathrm{soft},\mu}(\mathbb{X})_1\Vert^2\right].
\end{align}
The objective can be viewed as a reconstruction problem, in which the attention mechanism approximates a token $X_1$ as a combination of input tokens leveraging information from the entire prompt. 
In what follows, we assume for simplicity that training is performed through the population risk minimization. Although this idealization departs from the practical procedure, it does not affect the nature of our results, and an empirical counterpart could be handled similarly.

\paragraph{Measure-based formalism.}
To understand the training dynamics of a softmax attention layer, we lean on a measure-based formalism, see e.g.\ \citet{softmaxlinear}. A self-attention rank-one layer with attention parameter $\mu \in \mathbb{R}^d$ can be seen as an operator acting on measures:
\begin{align*}
T^{\lambda, \mu} : \mathcal{P}(\mathbb{R}^d)\times\mathbb{R}^d
&\rightarrow \mathbb{R}^d,\\
(\nu,z)
&\mapsto
T^{\lambda, \mu}[\nu](z)
=
\frac{\int_{\mathbb{R}^d}
\exp(\lambda z^\top \mu \mu^\top z') z'
\, d\nu(z')}
{\int_{\mathbb{R}^d}
\exp(\lambda z^\top \mu \mu^\top z')
\, d\nu(z')},
\end{align*}
so that when the prompt $\mathbb{X}=(X_1,\hdots , X_L)$ is encoded by its associated empirical measure
\[
\hat{\nu}_L
=
\frac{1}{L}\sum_{\ell=1}^L \delta_{X_\ell},
\]
one exactly retrieves the softmax attention formula \eqref{tlsoft}, i.e., $T^{\lambda,\mu}[\hat{\nu}_L](X_\ell)= T_L^{\mathrm{soft},\mu}(\mathbb{X})_\ell$.

A key observation is that when the prompt length grows, the empirical attention operator converges to its infinite-prompt counterpart, i.e.,
\[
T^{\lambda, \mu}
[\hat{\nu}_L](z)
\xrightarrow[L\to\infty]{a.s.}
T^{\lambda, \mu}
[\nu](z),
\]
with $\hat{\nu}_L$ still the empirical measure associated with $L$ i.i.d.\ tokens drawn according to $\nu$.
Moreover, when the token distribution $\nu$ is Gaussian, the infinite-prompt softmax attention becomes a linear operator. Specifically, it was shown in \citet[Lemma 2.1]{softmaxlinear} \citep[see also][Lemma 4.1]{castin}, that if $\nu=\mathcal{N}(0,\Sigma)$ then,
\[
T^{\lambda, \mu}[\nu](z)
=
\lambda\Sigma\mu\mu^\top z,
\] 
and consequently, the finite-prompt estimator converges almost surely to the linear operator
\begin{equation}\label{tisoft}
T_L^{\mathrm{soft},\mu}(\mathbb{X})_1
\xrightarrow[L\to\infty]{a.s.}
T_\infty^{\mathrm{soft},\mu}(X_1)
:=
\lambda\Sigma\mu\mu^\top X_1.
\end{equation}
The same almost sure convergence holds for the corresponding gradient (w.r.t.\ $\mu$) and Hessian of $T^{{\rm soft}, \mu}_L$  (see Lemma~\ref{lemma_as}).
This result shows that in the large-prompt regime, the nonlinear softmax attention layer behaves effectively as a linear operator acting on the token distribution. This convergence enables transferring optimization analyses from the linear operator to softmax attention when the prompt length is sufficiently large. We therefore introduce the theoretical risk of the infinite-prompt layer as $$\mathcal{R}_{\mathrm{soft},\infty}(\mu)=\mathbb{E}[\Vert X_1-T_{\infty}^{\mathrm{soft},\mu}(X_1)\Vert^2],$$ 
which admits the closed form\begin{align}\label{rsoftinfty}
    \mathcal{R}_{\mathrm{soft},\infty}(\mu)&=\Tr(\Sigma)-2\lambda b+\lambda^2ab. 
\end{align}
with 
$a=a(\mu)=\mu^\top\Sigma\mu$ and $b=b(\mu)=\mu^\top\Sigma^2\mu$.
This reduction will be key to the analysis of the optimization landscape for infinite-prompt architectures, and in particular for the characterization of the critical points, as conducted in the next section.

\subsection{Optimization analysis for infinite prompts}
We characterize in what follows the critical points of the infinite-prompt risk $\mathcal{R}_{\mathrm{soft},\infty}$.
\begin{proposition}[Landscape of $\mathcal{R}_{\mathrm{soft},\infty}$]\label{landscapeinf}
Assume that the p.s.d.\ covariance matrix $\Sigma\in\mathbb{R}^{d\times d}$ has a simple spectrum, i.e., distinct eigenvalues $\sigma_1 > \cdots > \sigma_d >0$,  with $u_j$ the associated unit eigenvectors.  
Then all critical points of $\mathcal{R}_{\mathrm{soft},\infty}$ are nondegenerate, and
\[
\mathrm{crit}(\mathcal{R}_{\mathrm{soft},\infty})=\{0\}\cup\left\{\pm \frac{1}{\sqrt{\lambda\sigma_j}}\,u_j :  j=1,\ldots, d\right\}
\]
 Moreover,
\begin{enumerate}
\item the point $0$ is a strict local maximum;
  \item the points $\pm\frac{1}{\sqrt{\lambda\sigma_j}}u_j$, for $j=2,\ldots,d$, are strict saddles;
  \item  the points $\pm\frac{1}{\sqrt{\lambda\sigma_1}}u_1$ are global minimizers of \(\mathcal{R}_{\mathrm{soft},\infty}\). 
  
\end{enumerate}
\end{proposition}

\begin{proposition}[Global minimization of $\mathcal{R}_{\mathrm{soft},\infty}$]\label{globalinf}
For almost every initialization $\mu_0\in \mathbb{R}^d$, the solution of \begin{align}\begin{cases}\label{gfi}\tag{$\mathrm{GF}_\infty$}
    \dot{\mu}_{\infty}(t)&=-\nabla \mathcal R_{\mathrm{soft},\infty}(\mu_{\infty}(t)),\\
    \mu_{\infty}(0)&=\mu_0.
\end{cases}\end{align}  converges to $
 \pm \frac{1}{\sqrt{\lambda \, \sigma_1}} \, u_1,$ 
with $(\sigma_1, u_1)$ the leading eigenpair of $\Sigma$.
\end{proposition}

Proposition \ref{globalinf} shows that training a rank-one softmax attention layer with an infinite-length Gaussian prompt by minimizing the population risk drives the attention parameter toward the leading principal component of the covariance matrix (up to a sign).
Note that all the results and arguments established for the gradient flow in this paper extend to the discrete-time setting, provided that gradient descent is used with a sufficiently small stepsize. We adopt the gradient flow formulation for notational simplicity.
\begin{remark}[PCA with $k$ components] \label{others}
With the knowledge of $u_1$, the second principal component $u_2$ can be obtained via the projected gradient flow
\begin{align}
\begin{cases}\label{pgfi}\tag{$\mathrm{PGF}_\infty$}
\dot{\mu}_{\infty}(t)
=
- \mathrm{P}_{u_1^\perp}\big(\nabla \mathcal R_{\mathrm{soft},\infty}(\mu_{\infty}(t))\big),\\
\mu_{\infty}(0)=\mu_0,
\end{cases}
\end{align}
where $\mathrm{P}_{u_1^\perp}=I_d-u_1u_1^\top$. By Proposition~\ref{landscapeinf}, the only non-degenerate critical points in $u_1^\perp$ that are not strict saddles are $\pm \frac{1}{\sqrt{\lambda\sigma_2}}u_2$, and thus, by the Stable Manifold Theorem~\cite[Theorem III.7]{shub}, the flow converges to these points for a generic initialization. This procedure can be iterated: projecting onto $\mathrm{Span}(u_1,\dots,u_{k-1})^\perp$ yields $u_k$, allowing sequential learning of all eigenvectors through the attention mechanism.
\end{remark}

To quantify the convergence rate of the gradient flow, we leverage Łojasiewicz inequalities, which relate the decay of the objective function to the norm of its gradient near critical points. Since $\mathcal{R}_{\mathrm{soft},\infty}$ is analytic, this inequality holds locally around each critical point, which enables us to derive explicit convergence rates toward the principal component.
\begin{definition}[Łojasiewicz inequality]
Let $f:\mathbb{R}^d \to \mathbb{R}$ be a differentiable function and let $\mu^\star \in \mathbb{R}^d$ be a critical point of $f$. We say that $f$ satisfies the \L ojasiewicz inequality at $\mu^\star$ if there exist constants $C>0$, $\alpha \in [0,1)$, and a neighborhood $U$ of $\mu^\star$ such that
\[
\|\nabla f(\mu)\|
\ge
C\,|f(\mu)-f(\mu^\star)|^{\alpha}
\quad \text{for all } \mu \in U.
\]
\end{definition}

We denote by $\sigma_{\min}(A)$ and $\sigma_{\max}(A)$ the smallest and largest eigenvalues of a matrix $A$.

\begin{proposition}[Local convergence rate on infinite-prompt setting]\label{cr_inf}
The risk $\mathcal R_{\mathrm{soft},\infty}$ satisfies \L ojasiewicz inequality with exponent $1/2$ at each critical point. Set $\mu^\star=\pm u_1 / \sqrt{\lambda \sigma_1}$, then there exist constants $t_0\ge 0$ and $s>0$ such that for all $t\ge t_0$,
\[
\mathcal{R}_{\mathrm{soft},\infty}(\mu_\infty(t))-\mathcal{R}_{\mathrm{soft},\infty}(\mu^\star)
\le 
\left(\mathcal{R}_{\mathrm{soft},\infty}(\mu_\infty(t_0))-\mathcal{R}_{\mathrm{soft},\infty}(\mu^\star)\right)
e^{-s(t-t_0)}.
\]

Besides, for every $\varepsilon>0$ small enough, there exists $t_0>0$ such that
\[
\|\mu_\infty(t)-\mu^\star\|
=
\mathcal{O}\big(e^{-(\tilde{s}-\varepsilon)(t-t_0)}\big),
\]
where
\begin{equation}\label{tildes}
\tilde{s}
=\sigma_{\min}(\nabla^2\mathcal{R}_{\mathrm{soft},\infty}(\mu^\star))=
2\lambda\min\{\sigma_2(\sigma_1-\sigma_2),\sigma_d(\sigma_1-\sigma_d)\}>0.    
\end{equation}

\end{proposition}

\subsection{Analysis transfer to the finite-prompt case}

In this section, we study how properties of the gradient flow \eqref{gfi} with infinite prompts can be transferred to the finite-prompt flow:
 \begin{align}\begin{cases}\label{gfl}\tag{$\mathrm{GF}_L$}
    \dot{\mu}_L(t)&=-\nabla_{\mu} \mathcal R_{\mathrm{soft},L}(\mu_L(t)),\\
    \mu_L(0)&=\mu_0.
\end{cases}
\end{align}

We begin by formalizing the relationship between the finite-prompt and infinite-prompt risks. Intuitively, as the prompt length $L$ increases, the finite-prompt operator $T_L^{\mathrm{soft},\mu}$ should approximate its population counterpart $T_{\infty}^{\mathrm{soft},\mu}$, leading to convergence of the corresponding risk functions. The following result makes this precise by establishing uniform convergence of the risks and their derivatives up to second order on compact sets.

For $k\in \mathbb{N}$, $\nabla^k f$ denotes the $k$-th order derivative of $f$ (in particular $\nabla^0 f=f$), when the codomain is not $\mathbb{R}$, we also write $D^k f$. For $x\in\mathbb{R}^d,\rho>0$ $B(x,\rho):=\{x'\in\mathbb{R}^d:\Vert x'-x\Vert\leq \rho\}$ is the closed ball centered at $x$ of radius $\rho$. We denote by $\Vert\cdot\Vert_F$ the Frobenius norm.
\begin{proposition}\label{uniform_c2}
We have that:
\begin{enumerate}
\item \label{hyp1}
For $k\in\{0,1,2\}$,
\[
\sup_{\mu\in B(0,\rho)}
\E\!\left[
\|D_\mu^k T_L^{\mathrm{soft},\mu}(\mathbb X)_1
-
D_\mu^k T_{\infty}^{\mathrm{soft},\mu}(X_1)\|_F^2
\right]
= \mathcal{O}(\psi_{k}(L)), 
\]
where $\psi_k(L)=L^{-\epsilon_k}(1+ \ln L)^{1-\epsilon_k}$, and $\epsilon_k=\frac{1}{16(k+3)^2\lambda^2(\mu^\top\Sigma\mu)^2+1}\in (0,1)$.
\item \label{hyp2}
For $k\in\{0,1,2\}$, there exists a constant $C_k=C_k(\rho,\Sigma,\lambda)$, such that
\[
\sup_{\mu\in B(0,\rho)}
\E\!\left[
\|D_\mu^k T_{\infty}^{\mathrm{soft},\mu}(X_1)\|_F^2
\right]
\le C_k.
\]
\item \label{hyp3}
Then for $k\in\{0,1,2\}$ there exists a constant
$C=C(\rho,\Sigma,C_0,C_1,C_2)$ independent of $L$ such that
\[
\sup_{\mu\in B(0,\rho)}
\|
\nabla^k \mathcal R_{\mathrm{soft},L}(\mu)
-
\nabla^k \mathcal R_{\mathrm{soft},\infty}(\mu)
\|_F^2
\le
C \sum_{j=0}^{k} \psi_{j}(L).
\]
In particular,
\[
\nabla^k \mathcal R_{\mathrm{soft},L} \xrightarrow[L\to\infty]{} \nabla^k \mathcal R_{\mathrm{soft},\infty}
\quad \text{uniformly on } B(0,\rho).
\]
\end{enumerate}
\end{proposition}
Assertion \ref{hyp1} is closely related to the concentration results of \citet{softmaxlinear}, which establish bounds for the output and its gradient and suggest similar behavior for the Hessian. However, our setting differs in two key ways. First, derivatives are taken with respect to $\mu$ rather than the product $K^\top Q$ of key and query matrices, which
 necessitates a specific analysis. Additionally, unlike prior work that treats the query token independently, we explicitly handle the autocorrelation induced by its inclusion in the prompt (see Appendix \ref{concentration} for more details).

Leveraging Proposition~\ref{uniform_c2}, we now turn to the finite-prompt setting. We begin by establishing a basic stability property of the infinite-prompt dynamics, namely that its trajectories remain bounded. We then show that this boundedness transfers to the finite-prompt dynamics for sufficiently large prompt lengths, ensuring that both flows remain confined to a common compact region over an infinite time horizon.
\begin{lemma}\label{coercive_bounded}
 Since $\mathcal R_{\mathrm{soft},\infty}$ is coercive, the solution trajectory $\mu_\infty$ of \eqref{gfi} is bounded for all $t\geq 0$. That is, that there exists $\rho\geq \Vert\mu_0\Vert$ such that $$\{\mu_\infty(t): t\geq 0\}\subset\{\mu\in\mathbb{R}^d:\mathcal R_{\mathrm{soft},\infty}(\mu)\leq \mathcal R_{\mathrm{soft},\infty}(\mu_0)\}\subset B(0,\rho).$$ 
 \end{lemma}
\begin{proposition}\label{prop:bounded-finite-flow}
Let $\mu_L$ and $\mu_\infty$ be the solutions of \eqref{gfl} and \eqref{gfi} respectively. Let $\rho>0$ be such that $\mu_\infty(t)\in B(0,\rho)$ for all $t\ge 0$. Then, for every $\rho'>\rho$, there exists $L'$ such that,
for all $L\ge L'$, $\mu_L(t)\in B(0,\rho')$ for every  $t\ge 0 .$
In particular, the trajectories $\mu_L$ are uniformly bounded in time for sufficiently large $L$.
\end{proposition}

Having established uniform boundedness, we next analyze how closely the finite-prompt dynamics track their infinite-prompt counterpart. The following results show that, on any finite time horizon, the trajectories and their associated risk values converge uniformly as the prompt length increases.
\begin{proposition}\label{conv_traj}
    Let $ \mu_L,\mu_\infty$ be the solutions of \eqref{gfl} and \eqref{gfi} respectively. Then for every $T>0$, $\mu_L$  converges uniformly to $\mu_{\infty}$ on $[0,T]$ as $L\rightarrow\infty$.
\end{proposition}

As a consequence, we can also control the convergence of the corresponding risk values along the trajectories.
\begin{proposition}\label{conv_values}
    Let $ \mu_L,\mu_\infty$ be the solutions of \eqref{gfl} and \eqref{gfi} respectively. Then $\mathcal R_{\mathrm{soft},L}(\mu_L)$ converges uniformly to $\mathcal R_{\mathrm{soft},\infty}(\mu_{\infty})$ on $[0,T]$ as $L\rightarrow\infty$. 
\end{proposition}

Propositions \ref{conv_traj} and \ref{conv_values} establish uniform convergence over any finite time horizon as $L\to \infty$. Our primary objective, however, is to characterize the long-time behavior of the finite-prompt trajectory. To this end, we can compare it by concentration arguments with the infinite-prompt dynamics, which, by Proposition \ref{globalinf}, converges to a global minimizer aligned with the leading eigenvector.

Beyond this comparison, our setting also allows for a sharper result: a direct characterization of the landscape of the finite-prompt risk. This is particularly noteworthy, as identifying critical points of objectives involving a softmax is typically challenging \citep{dohmatob2025understanding,marion2023,duranthon2025statistical, softmaxlinear,clustering}. The result is formalized in the following proposition.

\begin{proposition}\label{criticalpointsl}
Assume that the p.s.d.\ covariance matrix $\Sigma\in\mathbb{R}^{d\times d}$ has simple spectrum composed of the eigenvalues $\sigma_1>\ldots>\sigma_d>0$ with $(u_j)_{j=1}^d$ the associated unit eigenvectors. Consider a Gaussian finite prompt 
$X_1, \dots, X_L \stackrel{\mathrm{i.i.d.}}{\sim} \mathcal{N}(0, \Sigma)$.

Set $\rho>0$ to be large enough so that $B(0,\rho)$ contains all critical points of $\mathcal{R}_{\mathrm{soft},\infty}$ (see Proposition~\ref{landscapeinf}). Then there exists $L_0 \in \mathbb{N}$ such that for all $L \ge L_0$, the set of critical points of $\mathcal{R}_{\mathrm{soft},L}$ contained in $B(0,\rho)$ is finite, nondegenerate, and
$$
\mathrm{crit}(\mathcal{R}_{\mathrm{soft},L})\cap B(0,\rho)
=
\{\mu_{L,0}^\star\}\cup \{\pm \mu_{L,\sigma_j}^\star: j=1,\ldots,d\},
$$
where
\begin{enumerate}
    \item The point $\mu_{L,0}^\star$ is a strict local maximum such that $\mu_{L,0}^\star \xrightarrow[L\to\infty]{} 0$;

    \item The points $\pm\mu_{L,\sigma_j}^\star$, for $j=2,\ldots,d$, are strict saddles such that $\mu_{L,\sigma_j}^\star \xrightarrow[L\to\infty]{} \frac{1}{\sqrt{\lambda \sigma_j}}\, u_j$;

    \item The points $\pm\mu_{L,\sigma_1}^\star$ are strict local minima such that $\mu_{L,\sigma_1}^\star \xrightarrow[L\to\infty]{} \frac{1}{\sqrt{\lambda \sigma_1}}\, u_1$.
\end{enumerate}

\end{proposition}

We remark that it suffices to characterize the critical points of $\mathcal{R}_{\mathrm{soft},L}$ within a sufficiently large bounded set. Indeed, we will study the dynamics induced by the gradient flow associated with $\mathcal{R}_{\mathrm{soft},L}$, which corresponding trajectories remain bounded for $L$ large enough (Proposition~\ref{prop:bounded-finite-flow}). As a consequence, only critical points contained in this bounded region are relevant for the analysis.

\begin{proposition}[Local convergence rate on finite-prompt setting]\label{cr_l}
Let $L$ large enough and $\mu_L(t)$ be a solution of \eqref{gfl} with a generic initialization $\mu_0$. Then, $$\mu_L(t)\xrightarrow[t\to\infty]{} \mu_L^\star\in\{\pm \mu_{L,\sigma_1}^\star\}.$$ 
 Moreover, there exist $t_0\ge 0$ and $s>0$ such that for all $t\ge t_0$,
\begin{align*}
&\mathcal{R}_{\mathrm{soft},L}(\mu_L(t))-\mathcal{R}_{\mathrm{soft},L}(\mu_L^\star) \leq
\left(\mathcal{R}_{\mathrm{soft},L}(\mu_L(t_0))-\mathcal{R}_{\mathrm{soft},L}(\mu_L^\star)\right)
e^{-s(t-t_0)}.   
\end{align*}
Besides, for every $\varepsilon>0$ small enough, there exists $t_0>0$ such that
\[
\|\mu_L(t)-\mu_L^\star \|
=
\mathcal{O}\big(e^{-(\tilde{s}_L-\varepsilon)(t-t_0)}\big),
\]
with $\tilde{s}_L = \sigma_{\min}\big(\nabla^2 \mathcal{R}_{\mathrm{soft},L}(\mu_L^\star)\big) > 0$,
and $\tilde{s}_L \xrightarrow[L\to \infty]{} \tilde{s}$,
where \(\tilde{s}\) is defined in \eqref{tildes}.
\end{proposition}
 \begin{wrapfigure}[15]{r}{0.48\textwidth}
    \centering
    \vspace{-0.3cm}
    \includegraphics[width=1.\linewidth]{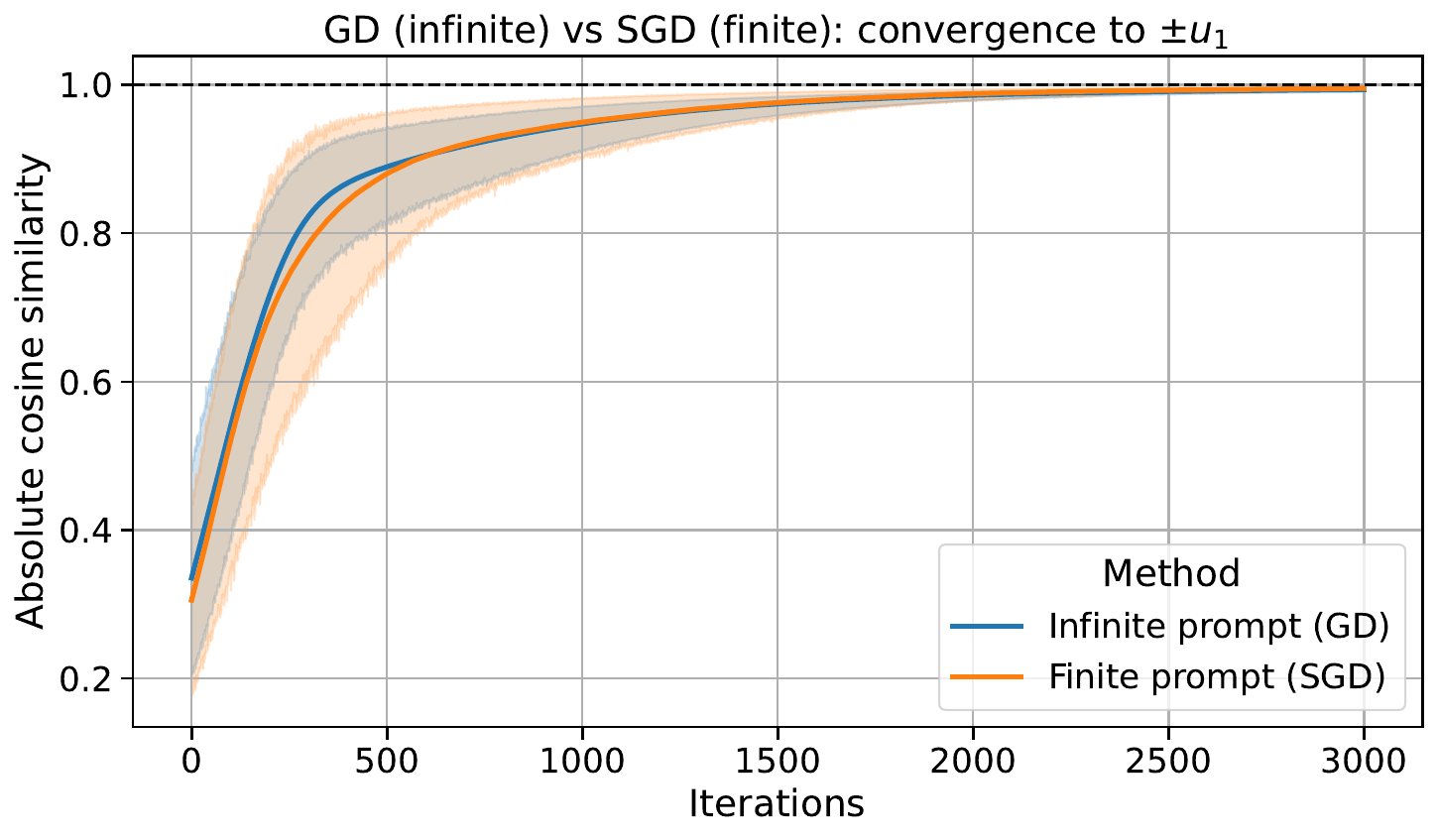}
    \caption{Alignment toward the principal eigenvector over iterations: SGD on $\mathcal{R}_{\mathrm{soft},L}$ ($L=100$) vs GD on $\mathcal{R}_{\mathrm{soft},\infty}$}
    \label{comparison}
\end{wrapfigure}The finite-prompt dynamics converges to the first principal component, sharing the same local convergence behavior as the infinite-prompt limit, with an exponential rate governed by the local Hessian at the minimizer. Moreover, this rate converges to that of the infinite-prompt regime as $L \to \infty$. 
Numerical results in Figure~\ref{comparison} illustrate the alignment toward the leading eigenvector for both gradient descent on $\mathcal{R}_{\mathrm{soft},\infty}$ (assuming direct access to $\Sigma$) and stochastic gradient descent on $\mathcal{R}_{\mathrm{soft},L}$, the latter exhibiting slightly slower and noisier convergence.

\subsection{Connection of attention to Oja's flow}

Interestingly, although our analysis is not motivated by classical results in online PCA, the infinite-prompt risk $\mathcal{R}_{\mathrm{soft},\infty}$ turns out to be closely related to the continuous-time limit of Oja’s rule~\citep{Oja}.

This connection emerges a posteriori from the structure of the gradient flow \eqref{gfi}, which can be written explicitly as
\[
\dot{\mu}_\infty =
4\lambda \Sigma^2 \mu_\infty
- 2\lambda^2\Big[
(\mu_\infty^\top \Sigma^2 \mu_\infty)\,\Sigma \mu_\infty
+ (\mu_\infty^\top \Sigma \mu_\infty)\,\Sigma^2 \mu_\infty
\Big].
\]
To make this link more transparent, we introduce the change of variables $w=\Sigma^{1/2}\mu_\infty$. In these coordinates, the dynamics take the form
\begin{equation}\label{ojalimit}
\dot w= \Sigma\big[A(w)\,\Sigma w - B\,(w^\top \Sigma w)\,  w\big], \quad A(w):=2\lambda(2-\lambda\, w^\top w), \; B:=2\lambda^2.
\end{equation}

This dynamics can be then viewed as a variant of Oja’s flow,
\begin{equation}\label{ojaflow}
 \dot w = \Sigma w - (w^\top \Sigma w)\, w ,
\end{equation}
in which the vector field is premultiplied by the invertible matrix $\Sigma$ and modulated by the coefficients $A(w)$ and $B$. In particular, both flows share the same stationary points, namely the eigenvectors of $\Sigma$, and they will have the same nature (local minima, saddles, or maxima). Equation \eqref{ojaflow} arises in PCA as a continuous-time model for extracting the principal eigenvector of $\Sigma$. Its discrete counterpart, called Oja's rule, is a stochastic online approximation of the ODE~\eqref{ojaflow}. Overall, this shows that, in the infinite-prompt limit, the dynamics of the attention mechanism reduces, up to a linear change of variables, to a generalized version of Oja’s flow, making explicit a connection between the attention mechanism and PCA that, to our knowledge, has not been discussed in the literature.

\section{Distributional properties of the attention-based encoding}\label{sec:dist}

The measure-based perspective allows us to interpret an attention layer with an infinite-length prompt as an operator acting on the input distribution, thereby enabling a characterization of the resulting output distribution in the Gaussian setting. We begin by characterizing the output distribution when the layer has been trained, that is, when the attention parameter $\mu$ has converged.
\begin{proposition}[Distribution of infinite-prompt attention operator]\label{distinf}
Let $\Sigma \in \mathbb{R}^{d\times d}$ be p.s.d.\ with $(\sigma_1,u_1)$ its principal eigenpair. For $\mu \in \mathbb{R}^d$, define 
$\Gamma(\mu) \coloneqq \lambda^2 (\mu^\top \Sigma \mu)\, (\Sigma \mu)(\Sigma \mu)^\top$. 
When $X_1 \sim \mathcal{N}(0, \Sigma)$, then
$$
T_{\infty}^{\mathrm{soft},\mu}(X_1)
= \lambda \Sigma \mu \mu^\top X_1
\sim \mathcal{N}(0, \Gamma(\mu)).
$$

For $\lambda>0$ and $\mu \neq 0$, the matrix $\Gamma(\mu)$ is p.s.d.\ of rank one. Moreover, for $\mu^\star = \frac{1}{\sqrt{\lambda \sigma_1}}\, u_1,$
which is a global minimizer of $\mathcal{R}_{\mathrm{soft},\infty}$ (up to a sign) attained by gradient flow for generic initializations, one has
$$
\Gamma(\mu^\star) = \sigma_1\, u_1 u_1^\top
\qquad\text{and}\qquad
T_{\infty}^{\mathrm{soft},\mu^\star}(X_1)
\sim
\mathcal{N}(0,\sigma_1 u_1 u_1^\top),
$$
that is, the limiting distribution coincides with the law of the projection of $X \sim \mathcal{N}(0,\Sigma)$ onto the principal eigenspace, i.e., $\langle X, u_1\rangle u_1$.
\end{proposition}
This result highlights that, in the infinite-prompt limit, the attention mechanism effectively performs a rank-one projection of the input Gaussian distribution, and at optimality, recovers the principal eigendirection of $\Sigma$, thus aligning with the objective of  PCA.
This naturally raises the question of how this behavior extends beyond the idealized infinite-prompt regime, which we now investigate by analyzing the output distribution for finite-length prompts.

\begin{corollary}[Wasserstein distance]\label{cor:wasserstein}
Consider Gaussian input tokens $X_1,\ldots,X_L \overset{i.i.d.}\sim \mathcal{N}(0,\Sigma)$.
Then, the squared $2$-Wasserstein distance between the distributions $\mathcal{L}(T_L^{\mathrm{soft},\mu}(\mathbb{X})_1)$ and $\mathcal{L}(T_\infty^{\mathrm{soft},\mu}(X_1))$  of the finite- and infinite-prompt architectures, both parameterized by $\mu$, satisfies
\[
W_2^2\!\big(\mathcal{L}(T_L^{\mathrm{soft},\mu}(\mathbb{X})_1),\,\mathcal{L}(T_\infty^{\mathrm{soft},\mu}(X_1))\big)
=
\mathcal{O}\!\left(
L^{-\epsilon}
(1+\ln L)^{1-\epsilon}
\right),
\]
where $\epsilon=\frac{1}{144\lambda^2(\mu^\top\Sigma\mu)^2+1}\in (0,1)$. 

In particular, when $\mu^\star=\frac{u_1}{\sqrt{\lambda \sigma_1}}$ with $(\sigma_1,u_1)$ the principal eigenpair of $\Sigma$, then  
\[
W_2^2\!\big(\mathcal{L}(T_L^{\mathrm{soft},\mu^\star}(\mathbb{X})_1),\,\mathcal{N}(0,\sigma_1u_1u_1^\top)\big)
=
\mathcal{O}\!\left(
L^{-\frac{1}{145}}
(1+\ln L)^{\frac{144}{145}}
\right).
\]

\end{corollary}

Proposition~\ref{globalinf} shows that minimizing the infinite-prompt risk $\mathcal{R}_{\mathrm{soft},\infty}$ recovers the principal components. Combined with the control of the deviation between the law of $T_L^{\mathrm{soft},\mu}(\mathbb{X})_1$ and its Gaussian limit $\mathcal{N}(0,\Gamma(\mu))$, as well as Proposition~\ref{cr_l}, this establishes that finite-prompt attention effectively performs an approximate PCA when training is performed through the minimization of $\mathcal{R}_{\mathrm{soft},L}$. 

We also note that the convergence rate observed in practice is significantly faster than the theoretical bounds (see, e.g., Figure \ref{fig:softmax_vs_L}), indicating that these bounds may be conservative.

\section{Toward spiked covariance models}\label{sec:icl}

In this section, we move beyond the fixed design considered so far and consider a setting in which the data remain Gaussian, but the (random) covariance structure depends on the prompt. This perspective places our analysis to some extent within the framework of in-context learning. 
We now make explicit the dependence on the covariance structure in the risks  
\[
\left\{
\begin{array}{l}
\mathcal{R}_{\mathrm{soft},L}^{(\Sigma)}(\mu)=\mathbb{E}_{X_1,\ldots,X_L\sim \mathcal{N}(0,\Sigma)}\left[\Vert X_1-T_{L}^{\mathrm{soft},\mu}(\mathbb{X})_1]\Vert^2\right] \\
\mathcal{R}_{\mathrm{soft},\infty}^{(\Sigma)}(\mu)
=
\mathbb{E}_{X \sim \mathcal{N}(0,\Sigma)}
\Big[\|X - T_{\infty}^{\mathrm{soft},\mu}(X)\|^2\Big]
\end{array}
\right.
\]
where we recall that the latter can be expressed in terms of the covariance matrix $\Sigma$:
\[
\mathcal{R}_{\mathrm{soft},\infty}^{(\Sigma)}(\mu)
=
\operatorname{tr}(\Sigma)
- 2\lambda\, \mu^{\top}\Sigma^{2}\mu
+ \lambda^{2}\, (\mu^{\top}\Sigma\mu)(\mu^{\top}\Sigma^{2}\mu).
\]
This formulation naturally leads to ``in-context learning'' risks, obtained by averaging over the distribution $\mathcal{D}$ of covariance matrices
\begin{equation}\label{icl_L}
\mathcal{R}_{L}^{\mathrm{ICL}}(\mu)
=
\mathbb{E}_{\Sigma \sim \mathcal{D}}
\big[
\mathcal{R}_{\mathrm{soft},L}^{(\Sigma)}(\mu)
\big]
\quad \text{and} \quad
\mathcal{R}_{\infty}^{\mathrm{ICL}}(\mu)
=
\mathbb{E}_{\Sigma \sim \mathcal{D}}
\big[
\mathcal{R}_{\mathrm{soft},\infty}^{(\Sigma)}(\mu)
\big].
\end{equation}

\paragraph{Choice of \(\mathcal{D}\).} 
We consider covariance matrices drawn from a Wishart distribution $W_d(V,n)$ with scale matrix $V$ and $n$ degrees of freedom. This choice yields random covariance matrices that almost surely have a simple spectrum (i.e., distinct eigenvalues), while their expectation aligns with a spiked covariance model of the form
\[
V = \xi^{2} I_d + \theta\, v v^{\top}, \quad \text{for some } \, v \, \text{ such that } \;
 \|v\| = 1.
\]
This mild random setting interpolates between the isotropic case $(\theta=0)$ and a structured anisotropic regime, providing a natural testbed to assess whether rank-one softmax attention layers can recover the latent direction $v$. Training is now modeled via the gradient flow of \eqref{icl_L}, corresponding to an idealized procedure based on the population risk rather than its empirical counterpart, with Gaussian prompts and prompt-dependent Wishart-distributed covariance matrices.

\paragraph{Infinite-prompt analysis.}
The function \(\mathcal{R}_{\infty}^{\mathrm{ICL}}(\mu)\) can be rewritten depending only (see Lemma~\ref{lemma:icl_expansion}) on the norm of the attention parameter $
r = \|\mu\|$  and on the angle $\alpha = \langle \mu, v \rangle$, between the attention parameter and the covariance latent direction, 
so that
\[
\mathcal{R}_{\infty}^{\mathrm{ICL}}(\mu)
=
\tilde{\mathcal{R}}^{\mathrm{ICL}}(r^2,\alpha^2).
\]

This reformulation enables a precise characterization of the critical points of the objective (see Propositions~\ref{families} and~\ref{criticalglobal}), and in particular shows that the only local (and thus global) minima are given by \begin{align}\label{mustaricl}
    \mu^\star=\pm \alpha^\star v,
\end{align} where the scalar $\alpha^\star$ is defined in \eqref{astar}, and determines the optimal magnitude of the parameter along the signal direction $v$. 
In addition, $\mu=0$ is a strict local maximum, and there are $2(d-1)$ saddle points corresponding to directions orthogonal to $v$. 

\begin{comment}
{\color{gray} {\color{red}I would put this into the appendices}Differentiating, if we define $A(r,\alpha)=a_1 r^2 + a_2 \alpha^2 - a_3$, and $
B(r,\alpha)=b_1 r^2 + b_2 \alpha^2 - b_3,$ for $a_i,b_i>0,$ $i=\{1,2,3\}$ explicited in \eqref{ab}, we get
\[
\nabla \mathcal{R}_{\infty}^{\mathrm{ICL}}(\mu)
=
2\big(A(r,\alpha)\mu + B(r,\alpha)\,\alpha v\big),
\]
 Writing \(\mu=\alpha v+w\) with \(w\perp v\), the stationarity condition becomes
\begin{equation}\label{eq:system}
\begin{cases}
A(r,\alpha)w=0,\\
(A(r,\alpha)+B(r,\alpha))\alpha v=0.
\end{cases}
\end{equation}
\subsection{Optimization landscape analysis}

}
\end{comment}

\begin{proposition}[Local convergence rate on ICL infinite-prompt setting]\label{localcrgf}
Consider the gradient flow
\[
\dot{\mu}(t) = - \nabla \mathcal{R}_\infty^{\mathrm{ICL}}(\mu(t)).
\]
%
%and let $\mu^\star = \pm \alpha^\star v $ (for $\alpha^\star$ defined in \eqref{astar})
%be the only global minimizers of $\mathcal{R}_\infty^{\mathrm{ICL}}$. 
Then, for a generic initialization $\mu_0$, the gradient flow satisfies $\mu(t)\xrightarrow[t\rightarrow\infty]{}\mu^\star$, where $\mu^\star$ is defined in \eqref{mustaricl}. Besides for every $\varepsilon>0$ small enough, there exists $t_0>0$ such that
\[
\|\mu(t)-\mu^\star\| =\mathcal{O}( e^{-(\hat{s}-\varepsilon) (t-t_0)}),
\]
with $\hat{s} = \sigma_{\min}\big(\nabla^2 \mathcal{R}_\infty^{\mathrm{ICL}}(\mu^\star)\big)$. When $\xi^2$ is small enough, we have $\hat{s}\sim  2 \frac{\lambda n(n^2+5n+2) \theta \, \xi^2}{n+4},$
so the convergence rate scales proportionally to $\xi^2$ and $\theta$, and quadratically with $n$.
%{\color{red}and so what? }
\end{proposition}

We observe that a moderate level of isotropic noise $\xi^2$ can facilitate convergence. Moreover, the convergence rate increases linearly with the signal strength $\theta$ and quadratically with the degrees of freedom $n$.

\paragraph{Finite-prompt analysis.}
The transfer of the infinite-prompt analysis can be done to the finite-prompt one as we establish the uniform convergence on bounded sets of $\nabla^k\mathcal{R}_L^{\mathrm{ICL}}(\mu)$ to $\nabla^k\mathcal{R}_\infty^{\mathrm{ICL}}(\mu)$ for $k\in \{0,1,2\}$ as $L\rightarrow\infty$.
\begin{wrapfigure}[15]{r}{0.48\textwidth}
    \centering
    \vspace{-0.1cm}
    \includegraphics[width=1.\linewidth]{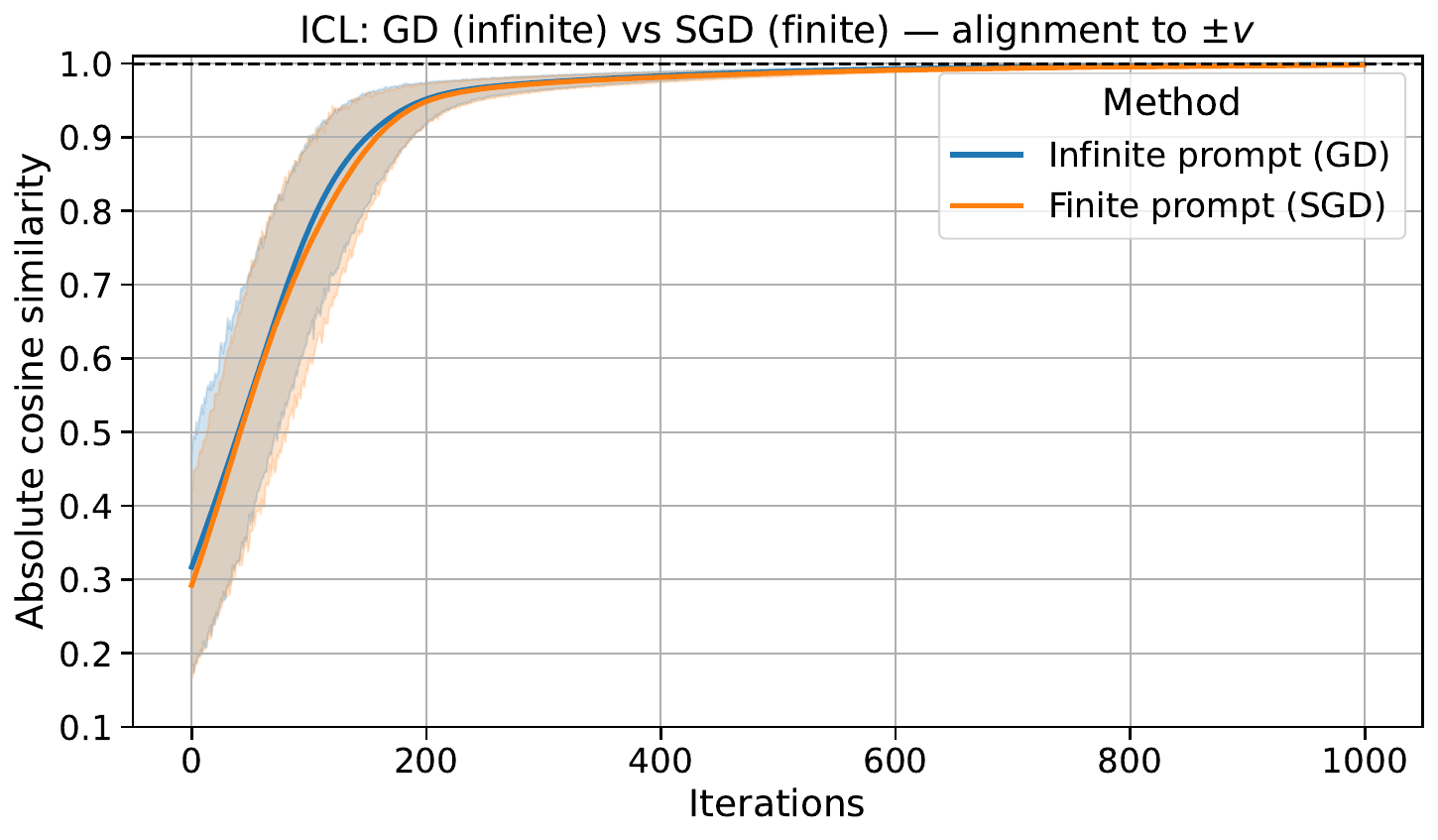}
    \caption{Alignment toward the signal direction $v$ over iterations: SGD on $\mathcal{R}_{L}^{\mathrm{ICL}}$ ($L=100$) vs GD on $\mathcal{R}_{\infty}^{\mathrm{ICL}}$}
    \label{comparison_ICL}
\end{wrapfigure} 
\begin{proposition}\label{uniform_c2_icl}
    For $\rho>0$, $k\in\{0,1,2\}$, \begin{align*}
    \lim_{L\rightarrow\infty}\sup_{\mu\in B(0,\rho)}\Vert \nabla^k \mathcal{R}_L^{\mathrm{ICL}}(\mu)- \nabla^k \mathcal{R}_\infty^{\mathrm{ICL}}(\mu)\Vert_F^2=0.\end{align*}
\end{proposition}

%{\color{gray}{\color{red}to put in the appendices and to summarize with words as it is very similar to what has been seen before}
%}
Following the argument of Proposition~\ref{criticalpointsl}, one proves that for $L$ large enough, the critical points of $\mathcal{R}_{L}^{\mathrm{ICL}}$ in a compact set correspond to perturbations of those of $\mathcal{R}_{\infty}^{\mathrm{ICL}}$. In particular, in this compact set there is one local maximum near $0$, two local minima near $\pm \alpha^\star v$, denoted $\pm\mu_{L,\parallel}^\star$, and possible saddle points located near the $2(d-1)$ orthogonal critical points of $\mathcal{R}_{\infty}^{\mathrm{ICL}}$ (see Proposition~\ref{criticalpointsicl} for more details). 
As a consequence, by proceeding as before, we obtain the following convergence and rate result for the ICL objective.

\begin{proposition}[ICL finite-prompt convergence and rate]
\label{prop:gd_convergence_icl}
Let $\rho>0$ be sufficiently large and $L\ge L_0$ as in Proposition~\ref{criticalpointsicl}. 
Consider the gradient flow
\[
\dot{\mu}(t) = -\nabla \mathcal{R}_L^{\mathrm{ICL}}(\mu(t)).
\]
Then, for almost every initialization $\mu_0\in B(0,\rho)$,
\[
\mu(t)\xrightarrow[t\to\infty]{} \mu_L^\star \in \{\pm \mu_{L,\parallel}^\star\},
\]

where $\mu_{L,\parallel}^\star \to \alpha^\star v$ as $L\to\infty$, with $\alpha^\star$ defined in \eqref{astar}. Moreover, there exist $t_0\ge 0$ and $\hat{s}_L>0$ such that for all $t\ge t_0$,
\[
\mathcal{R}_L^{\mathrm{ICL}}(\mu(t))-\mathcal{R}_L^{\mathrm{ICL}}(\mu_L^\star)
\le 
\big(\mathcal{R}_L^{\mathrm{ICL}}(\mu(t_0))-\mathcal{R}_L^{\mathrm{ICL}}(\mu_L^\star)\big)
e^{-\hat{s}_L (t-t_0)}.
\]

Besides, for every $\varepsilon>0$ small enough, there exists $t_0>0$ such that
\[
\|\mu(t)-\mu_L^\star\|
=
\mathcal{O}\big(e^{-(\hat{s}_L-\varepsilon)(t-t_0)}\big),
\]
with $
\hat{s}_L = \sigma_{\min}\big(\nabla^2 \mathcal{R}_L^{\mathrm{ICL}}(\mu_L^\star)\big) $, and $\hat{s}_L \xrightarrow[L\to\infty]{} \hat{s},$
where $\hat{s}$ is defined in \eqref{hats}.
\end{proposition}
Consequently, minimizing the finite-prompt ICL risk recovers an estimator that asymptotically aligns with the spike direction $v$. This provides an attention-based analogue of spiked PCA in which the latent direction is recovered through training dynamics rather than spectral decomposition. Numerical results in Figure~\ref{comparison_ICL} illustrate the alignment toward the signal direction for both gradient descent on $\mathcal{R}_{\infty}^{\mathrm{ICL}}$ (assuming direct access to $V=\xi^2I_d+\theta vv^\top$) and stochastic gradient descent on $\mathcal{R}_{L}^{\mathrm{ICL}}$.

% This setting connects our analysis to the classical spiked covariance model, a central framework in high-dimensional PCA where one studies the recovery of a low-rank signal direction embedded in isotropic noise. In the usual spectral perspective, the main question is whether the leading empirical eigenvector of a sample covariance matrix aligns with the latent spike, with sharp phase-transition phenomena arising in high-dimensional regimes. Here, the question is different: the spike is not recovered by explicitly diagonalizing a covariance matrix, but through the training dynamics of a rank-one softmax attention layer. Our result can therefore be viewed as an attention-based counterpart of spiked PCA, showing that the population ICL objective drives the attention parameter toward the latent covariance direction.

\section*{Code availability}
Our code is available at \\ \url{https://github.com/rodrigomaulen/Attention_based_PCA}.

\bibliographystyle{plainnat}
\bibliography{cites}
\appendix
\section*{Appendix}
This appendix gathers additional results and technical details supporting the main text. Section \ref{sec:app:linear_version} presents a linear attention model that similarly recovers the principal component. Sections \ref{sec:proof} and \ref{sec:lemmas} gather the proofs of the main results as well as the technical lemmas required throughout the paper. Finally, in Section \ref{sec:numerical} we present numerical experiments that illustrate and support our theoretical findings. 
\section{Linear attention layer}\label{sec:app:linear_version}
In line with \citep{clustering}, for $\mu\in \mathbb{R}^d$ and $X_1,\ldots,X_L$ i.i.d. according to $\mathcal{N}(0,\Sigma)$, we introduce the linear attention operator given by
\begin{equation}\label{simplified}
    T_L^{\mathrm{lin},\mu}(\mathbb{X})_\ell=\frac{\lambda}{L}\sum_{k=1}^L ( X_\ell^\top\mu\mu^\top X_k)X_k.
\end{equation}
We first note that by the strong law of large numbers, for $X\sim \mathcal{N}(0,\Sigma)$,
\[\frac{1}{L}\sum_{k=1}^L \mu^\top X_k X_k\;\xrightarrow[L\to\infty]{\mathrm{a.s.}}\;\mathbb{E}[\mu^\top X X]=\Sigma\mu.\]
Then,
\[
T_L^{\mathrm{lin},\mu}(\mathbb{X})_1
=\frac{\lambda}{L}\sum_{k=1}^L (X_1^\top \mu \,\mu^\top X_k)\,X_k
\;\xrightarrow[L\to\infty]{\mathrm{a.s.}}\;
\lambda\,\Sigma \mu\mu^\top X_1.
\]
This shows that $T_L^{\mathrm{lin},\mu}(\mathbb{X})_1$ and $T_L^{\mathrm{soft},\mu}(\mathbb{X})_1$ share the same almost sure limit when $L\rightarrow\infty$, i.e.,
\[\lim_{L\rightarrow\infty}T_L^{\mathrm{lin},\mu}(\mathbb{X})_1=\lim_{L\rightarrow\infty}T_L^{\mathrm{soft},\mu}(\mathbb{X})_1=T_{\infty}^{\mathrm{soft},\mu}(X_1)=\lambda\Sigma\mu\mu^\top X_1,\quad a.s..\]

\subsection{Risk}

We define the associated population risk as
\begin{equation}\label{rlin}
    \mathcal{R}_{\mathrm{lin},L}(\mu)= \mathbb{E}\left[\Vert X_1-T_L^{\mathrm{lin},\mu}(\mathbb{X})_1\Vert^2\right].
\end{equation}

We start by giving an expression of the risk as a function of the covariance matrix $\Sigma$ of the input sequence. 
\begin{proposition}
    Let $a(\mu) = \mu^T \Sigma \mu$ and $b(\mu) = \mu^T \Sigma^2 \mu$. We have that \begin{equation}\label{rlinl}
    \mathcal{R}_{\mathrm{lin},L}(\mu) = \Tr(\Sigma)-\frac{2\lambda}{L}\Tr(\Sigma) a-\frac{2\lambda(L+1)}{L}b+\frac{\lambda^2(L+2)\Tr(\Sigma)}{L^2}a^2+\frac{\lambda^2(L+2)(L+3)}{L^2}ab.
\end{equation}
\end{proposition}
\begin{proof} The computation starts as follows
    \begin{align*}
        \mathcal{R}_{\mathrm{lin},L}(\mu)&=\mathbb{E}\left[\Vert X_1-\frac{\lambda}{L}\sum_{k=1}^LX_1^\top\mu\mu^\top X_k X_k\Vert^2\right]\\
        &=\mathbb{E}\left[\Vert X_1-\frac{\lambda}{L}\sum_{k=2}^LX_1^\top\mu\mu^\top X_k X_k-\frac{\lambda}{L}(X_1^\top\mu)^2 X_1\Vert^2\right]\\
        &=\mathbb{E}\left[\Vert X_1-\frac{\lambda}{L}\sum_{k=2}^LX_1^\top\mu\mu^\top X_k X_k\Vert^2\right]+\frac{\lambda^2}{L^2}\mathbb{E}\left[(X_1^\top\mu)^4\Vert X_1\Vert^2\right]-2\frac{\lambda}{L}\mathbb{E}\left[(X_1^\top\mu)^2\Vert  X_1\Vert^2\right]\\
        &\quad+2\frac{\lambda^2}{L^2}\sum_{k=2}^L\mathbb{E}\left[(X_1^\top\mu)^3\mu^\top X_k X_1^\top X_k\right]\\
        \end{align*}
        And using Proposition \ref{computations}, we obtain that
        \begin{align*}
          &\mathbb{E}\left[\Vert X_1-\frac{\lambda}{L}\sum_{k=2}^LX_1^\top\mu\mu^\top X_k X_k\Vert^2\right]\\  
        &=\mathbb{E}\left[\Vert X_1\Vert^2\right]-2\frac{\lambda}{L}\sum_{k=2}^L \mathbb{E}\left[X_1^\top\mu\mu^\top X_k X_1^\top X_k\right]+\frac{\lambda^2}{L^2}\sum_{k=2}^L \mathbb{E}[(X_1^\top\mu)^2]\mathbb{E}[(X_k^\top\mu)^2\Vert X_k\Vert^2]\\
        &\quad+2\frac{\lambda^2}{L^2}\sum_{2\leq k<j\leq L}\mathbb{E}[(X_1^\top\mu)^2]\mathbb{E}[X_k^\top\mu X_j^\top\mu X_k^\top X_j]\\
        &=\Tr[\Sigma]-2\lambda(L-1)b+\frac{\lambda^2}{L^2}(L-1)[\Tr(\Sigma)a^2 + 2ab]+\frac{\lambda^2}{L^2}(L-1)(L-2)ab,
    \end{align*}
    where we used
    \begin{align*}
        \mathbb{E}\left[(X_1^\top\mu)^4\Vert X_1\Vert^2\right]=3\Tr(\Sigma)a^2+12ab,
    \end{align*}
    \begin{align*}
        \mathbb{E}\left[(X_1^\top\mu)^2\Vert  X_1\Vert^2\right]=\Tr(\Sigma)a+2b,
    \end{align*}
    \begin{align*}
        \mathbb{E}\left[(X_1^\top\mu)^3\mu^\top X_2 X_1^\top X_2\right]=3ab.
    \end{align*}
    Finally, 
    \begin{align*}
    \mathcal{R}_{\mathrm{lin},L}(\mu)&=\Tr[\Sigma]-2\frac{\lambda}{L}(L-1)b+\frac{\lambda^2}{L^2}(L-1)[\Tr(\Sigma)a^2 + 2ab]+\frac{\lambda^2}{L^2}(L-1)(L-2)ab\\
    &\quad+\frac{\lambda^2}{L^2}(3\Tr(\Sigma)a^2+12ab)-2\frac{\lambda}{L}(\Tr(\Sigma)a+2b)+6\frac{\lambda^2}{L^2}(L-1) ab\\
    &=\Tr(\Sigma)-2\frac{\lambda}{L}\Tr(\Sigma) a-2\frac{\lambda}{L}(L+1)b+\frac{\lambda^2}{L^2}(L+2)\Tr(\Sigma)a^2+\frac{\lambda^2}{L^2}(L+2)(L+3)ab.
\end{align*}
\end{proof}

\begin{comment}
\subsection{Simplification}
Let $\lambda > 0$ and $L \geq 1$. We consider the minimization of the following objective function $\mathcal{R}_{\mathrm{lin},L}(\mu)$:
\begin{equation}
    \mathcal{R}_{\mathrm{lin},L}(\mu) = \Tr(\Sigma)-2\lambda\Tr(\Sigma) a-2\frac{\lambda}{L}(L+1)b+\lambda^2(L+2)\Tr(\Sigma)a^2+\lambda^2(L+2)(L+3)ab.
\end{equation}
\end{comment}

\subsection{Optimization landscape analysis}

We first express the gradient of $\nabla \mathcal{R}_{\mathrm{lin},L}$ with respect to $\mu$, (using $\nabla a = 2\Sigma \mu$ and $\nabla b = 2\Sigma^2 \mu$) as
\begin{align*}
    \nabla \mathcal{R}_{\mathrm{lin},L}(\mu) &=  -4 \frac{\lambda}{L}  \Tr(\Sigma)  \Sigma \mu
   - 4 \frac{\lambda}{L}(L+1)   \Sigma^2 \mu + 4 \frac{\lambda^2}{L^2}(L+2)   \Tr(\Sigma)   a   \Sigma \mu \\
&\quad + 2 \frac{\lambda^2}{L^2}(L+2)(L+3) \Big( b   \Sigma \mu + a   \Sigma^2 \mu \Big),\\
    &=\alpha(\mu)\Sigma \mu+\beta(\mu)\Sigma^2 \mu,
\end{align*}
where $\alpha(\mu)= -4 \frac{\lambda}{L}  \Tr(\Sigma) + 4 \frac{\lambda^2}{L^2}(L+2)   \Tr(\Sigma)   a(\mu) + 2\frac{\lambda^2}{L^2}(L+2)(L+3)   b(\mu)$, and $\beta(\mu)= -4 \frac{\lambda}{L}(L+1) + 2\frac{\lambda^2}{L^2}(L+2)(L+3)   a(\mu)$. Setting $\nabla \mathcal{R}_{\mathrm{lin},L}(\mu) = 0$ leads to 
\[
    \alpha(\mu)\Sigma \mu+\beta(\mu)\Sigma^2 \mu = 0.
\]
We first note that $\mu=0$ is a critical point. Now consider the case $\mu\neq 0$, multiplying by $\Sigma^{-1}$ (since $\Sigma$ is invertible) gives
\begin{equation}\label{alphabeta}
    \alpha(\mu) \mu+\beta(\mu)\Sigma \mu = 0.
\end{equation}

This equation implies that $\Sigma \mu$ is aligned with the direction of $\mu$, i.e., $\Sigma \mu = -\frac{\alpha(\mu)}{\beta(\mu)} \mu,$ whenever $\beta(\mu)\neq 0$. Assume by contradiction that $\beta(\mu)=0$, then $\alpha(\mu)$ would be fixed accordingly, and the critical point condition would further impose $\alpha(\mu)=0$, which would in turn fix $b(\mu)$ to a negative value, which is contradictory with its definition.  Hence, for any non-zero critical point, one must have $\beta(\mu)\neq 0$. It follows that $\mu$ must be an eigenvector of $\Sigma$ with eigenvalue $-\frac{\alpha(\mu)}{\beta(\mu)}$. 

It follows that the critical points are given by $\mu=0$ and points of the form $\gamma_i u_i$ for $i=1,\ldots,d$ with $u_i$ are unit eigenvectors of $\Sigma$ with associated eigenvalue $\sigma_i>0$, and $\gamma_i\neq 0$. Plugging $\mu=\gamma_i u_i$ into the equation \eqref{alphabeta} and simplifying we get the condition $$-\Tr(\Sigma)-(L+1)\sigma_i+\frac{\lambda}{L}(L+2)\Tr(\Sigma)\gamma_i^2\sigma_i+\frac{\lambda}{L}(L+2)(L+3)\gamma_i^2\sigma_i^2=0.$$
Solving for $\gamma_i$, we obtain $$\gamma_i=\pm \sqrt{\frac{ \Tr(\Sigma) + (L+1)\sigma_{i}}{\frac{\lambda}{L}  \sigma_{i} (L+2) (\Tr(\Sigma) + (L+3)\sigma_{i})}}.$$

\begin{proposition}[Characterization of critical points]\label{characterization}
Let $\Sigma\in\mathbb{R}^{d\times d}$ be symmetric positive definite with simple spectrum, and let
$(\sigma_i,u_i)_{i=1}^d$ be its eigenpairs, where $u_i$ are unit eigenvectors and $\sigma_1>\ldots>\sigma_d$.
For $i\in\{1,\ldots,d\}$, define
\[
\mu_i^{\pm} = \pm \gamma_i^\star u_i,
\qquad
\gamma_i^\star
= \sqrt{\frac{ \Tr(\Sigma) + (L+1)\sigma_{i}}{\frac{\lambda}{L}  \sigma_{i} (L+2) (\Tr(\Sigma) + (L+3)\sigma_{i})}}.
\]

Then $\mathrm{crit}(\mathcal{R}_{\mathrm{lin},L})=\{0\}\cup\{\mu_i^{\pm}: i= 1,\ldots,d\}$. We have that $\nabla^2 \mathcal{R}_{\mathrm{lin},L}(0)$ is negative definite, thus $0$ is a local maxima. Moreover, the Hessian $\nabla^2 \mathcal{R}_{\mathrm{lin},L}(\mu_i^{\pm})$ is diagonal in the eigenbasis of $\Sigma$. The eigenvalue of
$\nabla^2 \mathcal{R}_{\mathrm{lin},L}(\mu_i^{\pm})$ associated with the eigenvector $u_j$ is given by
\[
\begin{cases} 
& 8\frac{\lambda}{L}\sigma_i(\Tr(\Sigma)+(L+1)\sigma_i) \text{, if } j=i,\\
& 2\frac{\lambda}{L}\sigma_j(\sigma_i-\sigma_j)\frac{(L-1)\Tr(\Sigma)+(L+1)(L+3)\sigma_i}{\Tr(\Sigma)+(L+3)\sigma_i}\text{, if } j\neq i.
\end{cases}
\]

In particular, if $i=1$, the Hessian $\nabla^2 \mathcal{R}_{\mathrm{lin},L}(\mu_1^{\pm})$ is positive definite.
If $i>1$, the Hessian $\nabla^2 \mathcal{R}_{\mathrm{lin},L}(\mu_i)$ has both positive and negative
eigenvalues. Consequently, $\mu_1^{\pm}=\pm\alpha_1^\star u_1$ are global minimizers of $\mathcal{R}_{\mathrm{lin},L}$,
whereas $\mu_i^{\pm}=\pm\gamma_i^\star u_i$ is a strict saddle point for every $i>1$.
\end{proposition}
\begin{proof} 

We compute the Hessian of $\mathcal{R}_{\mathrm{lin},L}$ at a generic $\mu$,
    \begin{align*}
        \nabla^2 \mathcal{R}_{\mathrm{lin},L}(\mu)&=\Big(-4\frac{\lambda}{L}\Tr(\Sigma)+4\frac{\lambda^2}{L^2}(L+2)\Tr(\Sigma) a(\mu)+2\frac{\lambda^2}{L^2}(L+2)(L+3) b(\mu)\Big)\Sigma
\\
&\quad+\Big(-4\frac{\lambda}{L}(L+1)
+2\frac{\lambda^2}{L^2}(L+2)(L+3) a(\mu)\Big)\Sigma^{2}+8\frac{\lambda^2}{L^2}(L+2)\Tr(\Sigma)\Sigma\mu(\Sigma\mu)^{\top}
\\
&\quad +4\frac{\lambda^2}{L^2}(L+2)(L+3)\Big(\Sigma\mu(\Sigma^{2}\mu)^{\top}+\Sigma^{2}\mu(\Sigma\mu)^{\top}\Big).
    \end{align*}
    Then $$\nabla^2 \mathcal{R}_{\mathrm{lin},L}(0)=-4\frac{\lambda}{L}\Sigma(\Tr(\Sigma)+(L+1)\Sigma).$$ And evaluating at $\mu_i^{\pm}=\pm\gamma_i^\star u_i$ we obtain 

    \[
\begin{aligned}
\nabla^{2}\mathcal{R}_{\mathrm{lin},L}(\mu_i^{\pm})=&\Big(-4\frac{\lambda}{L}\Tr(\Sigma)+4\frac{\lambda^2}{L^2}(L+2)\Tr(\Sigma) (\gamma_i^\star)^2\sigma_i+2\frac{\lambda^2}{L^2}(L+2)(L+3) (\gamma_i^\star)^2\sigma_i^{2}\Big)\Sigma
\\
&\quad+\Big(-4\frac{\lambda}{L}(L+1)+2\frac{\lambda^2}{L^2}(L+2)(L+3) (\gamma_i^\star)^2\sigma_i\Big)\Sigma^{2}\\
&\quad+8\frac{\lambda^2}{L^2}(L+2)(\gamma_i^\star)^2\sigma_i^{2}\Big(\Tr(\Sigma)+(L+3)\sigma_i\Big)u_i u_i^{\top}.
\end{aligned}
\]
Consequently, \[\begin{aligned}
    \nabla^{2}\mathcal{R}_{\mathrm{lin},L}(\mu_i^{\pm})u_i=&\Big(-4\frac{\lambda}{L}\Tr(\Sigma)+4\frac{\lambda^2}{L^2}(L+2)\Tr(\Sigma)(\gamma_i^\star)^2\sigma_i+2\frac{\lambda^2}{L^2}(L+2)(L+3)(\gamma_i^\star)^2\sigma_i^{2}\Big)\sigma_i u_i
\\
&\quad+\Big(-4\frac{\lambda}{L}(L+1)+2\frac{\lambda^2}{L^2}(L+2)(L+3)(\gamma_i^\star)^2\sigma_i\Big)\sigma_i^{2}u_i\\
&\quad+8\frac{\lambda^2}{L^2}(L+2)(\gamma_i^\star)^2\sigma_i^{2}\Big(\Tr(\Sigma)+(L+3)\sigma_i\Big)u_i.\\
=&\quad 8\frac{\lambda}{L}\sigma_i(\Tr(\Sigma)+(L+1)\sigma_i)u_i,
\end{aligned}\]
where the last equality comes from replacing $(\gamma_i^\star)^2$ into the expression. Besides, for $j\neq i$
\[\begin{aligned}
    \nabla^{2}\mathcal{R}_{\mathrm{lin},L}(\mu_i^{\pm})u_j=&\Big(-4\frac{\lambda}{L}\Tr(\Sigma)+4\frac{\lambda^2}{L^2}(L+2)\Tr(\Sigma)(\gamma_i^\star)^2\sigma_i+2\frac{\lambda^2}{L^2}(L+2)(L+3)(\gamma_i^\star)^2\sigma_i^{2}\Big)\sigma_j u_j
\\
&\quad+\Big(-4\frac{\lambda}{L}(L+1)+2\frac{\lambda^2}{L^2}(L+2)(L+3)(\gamma_i^\star)^2\sigma_i\Big)\sigma_j^{2}u_j\\
=&\quad 2\frac{\lambda}{L}\sigma_j(\sigma_i-\sigma_j)\frac{(L-1)\Tr(\Sigma)+(L+1)(L+3)\sigma_i}{\Tr(\Sigma)+(L+3)\sigma_i}u_j,
\end{aligned}\]
\end{proof}

\begin{proposition}
    The function $\mathcal{R}_{\mathrm{lin},L}$ is coercive and locally Lipschitz. Hence, the gradient flow dynamic
\[
\dot{\mu}(t)=  -\nabla \mathcal{R}_{\mathrm{lin},L}(\mu(t)),
\]
converges to a critical point of $\mathcal{R}_{\mathrm{lin},L}$. Furthermore, for almost every initialization, the sequence converges to one of the two global minimizers, $\mu^{\star} = \pm\alpha_1^\star u_1$, where $u_1$ is the normalized principal eigenvector of $\Sigma$, and 
\[
\alpha_1^\star = \sqrt{ \frac{ \Tr(\Sigma) + (L+1)\sigma_{1}}{\frac{\lambda}{L}  \sigma_{1} (L+2) (\Tr(\Sigma) + (L+3)\sigma_{1})}}.
\]
Normalizing the limit point $\mu^\star$ recovers the principal eigenvector $u_1$ up to a sign. Besides, \[\begin{aligned}
    \sigma_1=&\quad-\frac{\alpha(\mu^\star)}{\beta(\mu^\star)}=-\frac{-2  \Tr(\Sigma) +  2\frac{\lambda}{L}(L+2) \Tr(\Sigma) \mu^\star\Sigma\mu^\star + \frac{\lambda}{L}(L+2)(L+3)  \mu^\star\Sigma^2\mu^\star}{-2 (L+1) + \frac{\lambda}{L}(L+2)(L+3) \mu^\star\Sigma\mu^\star}.\\
    %=& \frac{(L+1)-\lambda\Vert\mu^\star\Vert^2\Tr(\Sigma)(L+2)+\sqrt{((L+1)-\lambda\Vert\mu^\star\Vert^2\Tr(\Sigma)(L+2))^2+4\lambda\Vert\mu^\star\Vert^2\Tr(\Sigma)(L+2)(L+3)}}{2\lambda\Vert\mu^\star\Vert^2(L+2)(L+3) }.
\end{aligned}\]
\end{proposition}
\begin{proof}
Since $\mathcal{R}_{\mathrm{lin},L}$ is a polynomial in $\mu$, it is $C^\infty$ and locally Lipschitz. Moreover, $\mathcal{R}_{\mathrm{lin},L}$ is coercive, this implies that the sublevel set $\mathcal{S} = \{ \mu : \mathcal{R}_{\mathrm{lin},L}(\mu) \leq \mathcal{R}_{\mathrm{lin},L}(\mu^0) \}$ is compact for any initialization $\mu^0$.
%Restricted to this compact set, the gradient $\nabla \mathcal{R}_{\mathrm{lin},L}$ is Lipschitz continuous with constant $L_{0}$. 
%Then for a stepsize $\gamma < 2/L_{0}$, the gradient descent iteration $\mu^{k+1} = \mu^k - \gamma \nabla \mathcal{R}_{\mathrm{lin},L}(\mu^k)$ produces a sequence where 
The gradient flow trajectory $\mathcal{R}_{\mathrm{lin},L}(\mu(t))$ is non-increasing and $\|\nabla \mathcal{R}_{\mathrm{lin},L}(\mu(t))\| \to 0$. Furthermore, since $\mathcal{R}_{\mathrm{lin},L}$ is polynomial, it satisfies the \L{}ojasiewicz inequality at every critical point, thus $\mu(t)$ converges to the set of critical points given by Proposition \ref{characterization}, namely $\mathrm{crit}(\mathcal{R}_{\mathrm{lin},L}) = \{0\} \cup \{\mu_i^{\pm}:i=1,\ldots,d\}$.\\
%We view gradient descent as the iteration of a map $G(\mu) = \mu - \gamma \nabla \mathcal{R}_{\mathrm{lin},L}(\mu)$. A critical point $\mu^\star$ of $\mathcal{R}$ is a fixed point of $G$. The Jacobian of this map is $J(\mu^\star) = I - \gamma \nabla^2 \mathcal{R}(\mu^\star)$.
%A fixed point $\mu^\star$ is \textit{unstable} if $J(\mu^\star)$ has at least one eigenvalue with magnitude strictly greater than 1. This occurs if $\nabla^2 \mathcal{R}(\mu^\star)$ has a strictly negative eigenvalue $\nu < 0$ (provided $\gamma$ is sufficiently small), since the corresponding eigenvalue of $J$ would be $1 - \gamma\nu > 1$.

Using the classification from Proposition~\ref{characterization}, the Stable Manifold Theorem \cite[Theorem III.7]{shub} states that the set of initial conditions converging to an unstable fixed point of a $C^2$ diffeomorphism has Lebesgue measure zero. 
Consequently, for almost all initializations, the sequence cannot converge to $0$ or any $\mu_i$ with $i > 1$.
Since the sequence is bounded and must converge to a critical point, it converges to the stable minimizer $\mu_1^{\pm} = \pm\alpha_1^\star u_1$.
\end{proof}

\begin{remark}
As in Remark~\ref{others}, once the leading eigenvector $u_1$ has been recovered, the second principal component $u_2$ can be obtained by constraining the dynamics to the orthogonal subspace $u_1^\perp$. To this end, we consider the projected gradient flow
\begin{align}
\begin{cases}\label{pgfilin}
\dot{\mu}(t)
=
- \mathrm{P}_{u_1^\perp}\big(\nabla \mathcal R_{\mathrm{lin},L}(\mu(t))\big),\\
\mu(0)=\mu_0,
\end{cases}
\end{align}
where $\mathrm{P}_{u_1^\perp}=I_d-u_1u_1^\top$ denotes the orthogonal projection onto $u_1^\perp$. This projection removes the component along $u_1$, allowing the dynamics to generically recover $u_2$.
\end{remark}

\section{Proofs}\label{sec:proof}
In the following section, we present the proofs of the propositions and lemmas stated in the main text.

\subsection{Proof of Lemma \ref{coercive_bounded}}
\begin{proof}
Along the gradient flow \eqref{gfi}, the map
$t\mapsto \mathcal R_{\mathrm{soft},\infty}(\mu_\infty(t))$ is non-increasing. Hence, for all $t\ge 0$,
\[
\mathcal R_{\mathrm{soft},\infty}(\mu_\infty(t))
\le
\mathcal R_{\mathrm{soft},\infty}(\mu_0).
\]
By coercivity of $\mathcal R_{\mathrm{soft},\infty}$, the sublevel set $[\mathcal R_{\mathrm{soft},\infty}\leq R_{\mathrm{soft},\infty}(\mu_0)]$ is bounded. Therefore, there exists $\rho\geq \Vert\mu_0\Vert$ such that
$\mu_\infty(t)\in B(0,\rho)$ for all $t\ge 0$, which proves the claim.
\end{proof}
\subsection{Proof of Proposition \ref{landscapeinf}}
\begin{proof}
Compute gradients using $\nabla_\mu a = 2\Sigma\mu$ and $\nabla_\mu b = 2\Sigma^2\mu$:
\[
\nabla_\mu \mathcal{R}_{\mathrm{soft},\infty}(\mu) = -4\lambda\Sigma^2\mu + 2\lambda^2\big(b \Sigma\mu + a \Sigma^2\mu\big).
\]
At a nonzero critical point, divide by $2$ and rearrange to get
\[
(\lambda a-2)\Sigma^2\mu + \lambda b \Sigma\mu = 0.
\]
Multiplying by $\Sigma^{-1}$ yields
\[
(\lambda a-2)\Sigma\mu + \lambda b\mu = 0\]
Assuming $\lambda a-2\neq 0$, we can rearrange the previous equation into the following one
\[
\Sigma\mu = c(\mu)\mu,
\]
for $c(\mu)=-\dfrac{\lambda b}{\lambda a-2}$. If $\lambda a-2=0$, we would have 
$$\lambda b\mu=0\Longleftrightarrow \frac{b}{a}\mu=0,
$$
which has no solution for $\mu\neq 0$ (since $\Sigma$ is positive definite). Thus any nonzero critical $\mu$ is an eigenvector of $\Sigma$. Writing $\Sigma v=\sigma v$ and $\mu=\alpha v$ we obtain
\[
a=\alpha^2\sigma,\qquad b=\alpha^2\sigma^2,
\]
and stationarity reduces to
\[
\lambda\sigma\alpha^2 = 1 \quad\Longrightarrow\quad \alpha = \pm\frac{1}{\sqrt{\lambda\sigma}}.
\]
Hence the nonzero critical points are precisely $\pm\frac{1}{\sqrt{\lambda\sigma}}v$ for eigenpairs $(\sigma,v)$, together with the trivial critical point $\mu=0$.

\begin{comment}
Evaluate the objective at a nonzero critical point. For $\mu=\pm\frac{1}{\sqrt{\lambda\sigma}}v$ we have $a=1/\lambda$ and $b=\sigma/\lambda$, so
\[
F(\mu)=\operatorname{tr}(\Sigma)-2\lambda b + \lambda^2 a b
= \operatorname{tr}(\Sigma)-\sigma.
\]
Therefore the value of \(F\) at a nonzero critical point is $\operatorname{tr}(\Sigma)-\sigma$. Minimizing this over eigenvalues \(\sigma\) is equivalent to \emph{maximizing} \(\sigma\). Consequently any critical point associated with the top eigenvalue \(\sigma_1\) attains the smallest possible value
\[
F_{\min}=\operatorname{tr}(\Sigma)-\sigma_1,
\]
so such critical points are global minimizers of \(F\).
\end{comment}
Now let us compute the Hessian,
\[
\nabla^2_\mu \mathcal{R}_{\mathrm{soft},\infty}(\mu)
= -4\lambda\Sigma^2
+4\lambda^2\big(\Sigma^2\mu \mu^\top\Sigma+\Sigma\mu \mu^\top\Sigma^2\big)
+2\lambda^2\big(b\Sigma+a\Sigma^2\big).
\]
Evaluate $\nabla^2_\mu \mathcal{R}_{\mathrm{soft},\infty}$ at a critical $\mu_{\star}^{\pm}=\pm\alpha v$ with $\Sigma v=\sigma v$ and $\alpha^2=\frac{1}{\lambda\sigma}$. For any eigenvector $w$ of $\Sigma$ with $\Sigma w=\tau w$ one obtains
\[
\nabla^2_\mu \mathcal{R}_{\mathrm{soft},\infty}(\mu_{\star}^{\pm})w =
\begin{cases}
2\lambda\tau(\sigma-\tau) w, & w\perp v,\\
8\lambda\sigma^2 w, & w=v,
\end{cases}
\]
which follows from substituting $a=1/\lambda$, $b=\sigma/\lambda$ and simplifying. From this:
\begin{itemize}
 \item At $\mu=0$ we have $a=b=0$ and $\nabla^2_\mu \mathcal{R}_{\mathrm{soft},\infty}(0)=-4\lambda\Sigma^2$, which is negative definite; hence $\mu=0$ is a strict local maximum.
  \item If there exists $\tau>\sigma$ (in particular if $\sigma<\sigma_1$) then for that $\tau$ the curvature $2\lambda\tau(\sigma-\tau)$ is negative, so the critical point is a strict saddle.
  \item If $\sigma=\sigma_1$, then for every eigenvalue $\tau<\sigma_1$ we have $\tau(\sigma_1-\tau)>0$, so the Hessian at $\mu_{\star}^{\pm}=\pm\frac{1}{\sqrt{\lambda\sigma_1}}u_1$
is positive definite. Moreover, evaluating the objective yields
\[
\mathcal{R}_{\mathrm{soft},\infty}(\mu_{\star}^{\pm})=\operatorname{tr}(\Sigma)-\sigma_1,
\]
and by Lemma \ref{lem:localtoglobal}, these local minimizers are, in fact, global minimizers.
 
\end{itemize}

%This completes the proof: the only global minimizers are the points proportional to vectors in the top eigenspace of \(\Sigma\), and any critical point coming from a smaller eigenvalue is a strict saddle.
\end{proof}
\subsection{Proof of Proposition \ref{globalinf}}
\begin{proof}
The claim follows from the fact that this function is algebraic, and for such a function, gradient flow (and gradient descent with a proper stepsize) avoids strict saddle points for almost every initialization. By the Stable Manifold Theorem \cite[Theorem III.7]{shub}, the set of initial conditions whose trajectories converge to such a point is contained in its stable manifold, which has Lebesgue measure zero. Hence almost every initialization converges to a local minimum (see, e.g., \citep{lee2016gradient, panageas2017gradient}).
\end{proof}
\subsection{Proof of Proposition \ref{cr_inf}}
\begin{proof}
Since $\mathcal{R}_{\mathrm{soft},\infty}$ is a polynomial, it is real analytic. Moreover, by Proposition \ref{landscapeinf} its critical points are finite and nondegenerate. Therefore, by Corollary \ref{cor1} it satisfies the \L ojasiewicz inequality with exponent $\theta=\frac{1}{2}$ in a neighborhood of each critical point. Let
\[
E(t)=\mathcal{R}_{\mathrm{soft},\infty}(\mu_\infty(t))-\mathcal{R}_{\mathrm{soft},\infty}(\mu^\star).
\]
Along the flow,
\[
\dot{E}(t)
=
-\|\nabla \mathcal{R}_{\mathrm{soft},\infty}(\mu_\infty(t))\|^2.
\]
By the \L ojasiewicz inequality, there exists $t_0>0$ and $c>0$ such that for every $t\geq t_0$
\[
\|\nabla \mathcal{R}_{\mathrm{soft},\infty}(\mu_\infty(t))\|
\ge
c\,E(t)^{1/2}.
\]
Thus
\[
\dot{E}(t)\le -c^2 E(t),
\]
and Grönwall's inequality yields
\[
E(t)\le E(0)e^{-c^2 t}.
\]

Next, since $\mu^\star$ is a nondegenerate minimizer, $\nabla^2 \mathcal{R}_{\mathrm{soft},\infty}(\mu^\star)$ is positive definite. By continuity of the Hessian, for any $\varepsilon>0$ there exists a neighborhood of $\mu^\star$ such that for all $\mu$ in this neighborhood,
\[
\|\nabla^2 \mathcal{R}_{\mathrm{soft},\infty}(\mu)-\nabla^2 \mathcal{R}_{\mathrm{soft},\infty}(\mu^\star)\|
\le \varepsilon.
\]
Using the integral form of Taylor's theorem,
\[
\nabla \mathcal{R}_{\mathrm{soft},\infty}(\mu)
=
\nabla^2 \mathcal{R}_{\mathrm{soft},\infty}(\mu^\star)(\mu-\mu^\star)
+
r(\mu),
\]
where $\|r(\mu)\|\le \varepsilon \|\mu-\mu^\star\|$. Therefore,
\[
\langle \nabla \mathcal{R}_{\mathrm{soft},\infty}(\mu), \mu-\mu^\star\rangle
\ge
\langle \nabla^2 \mathcal{R}_{\mathrm{soft},\infty}(\mu^\star)(\mu-\mu^\star),\mu-\mu^\star\rangle
-
\varepsilon \|\mu-\mu^\star\|^2.
\]
Since the smallest eigenvalue of $\nabla^2 \mathcal{R}_{\mathrm{soft},\infty}(\mu^\star)$ is
\[
\tilde{s}
=
2\lambda\min\{\sigma_2(\sigma_1-\sigma_2),\sigma_d(\sigma_1-\sigma_d)\},
\]
it follows that
\[
\langle \nabla \mathcal{R}_{\mathrm{soft},\infty}(\mu), \mu-\mu^\star\rangle
\ge
(\tilde{s}-\varepsilon)\|\mu-\mu^\star\|^2.
\]
By Proposition \ref{globalinf}, we have $\mu_\infty(t)\to\mu^\star$, then there exists $t_0>0$ such that this inequality holds for all $t\ge t_0$. Then
\[
\frac{d}{dt}\|\mu_\infty(t)-\mu^\star\|^2
=
-2\langle \nabla \mathcal{R}_{\mathrm{soft},\infty}(\mu_\infty(t)), \mu_\infty(t)-\mu^\star\rangle
\le
-2(\tilde{s}-\varepsilon)\|\mu_\infty(t)-\mu^\star\|^2.
\]
Applying Grönwall's inequality,
\[
\|\mu_\infty(t)-\mu^\star\|^2
\le
\|\mu_\infty(t_0)-\mu^\star\|^2 e^{-2(\tilde{s}-\varepsilon)(t-t_0)},
\]
we conclude that
\[
\|\mu_\infty(t)-\mu^\star\|
\le
C e^{-(\tilde{s}-\varepsilon)(t-t_0)}.
\]
\end{proof}
\subsection{Proof of Proposition \ref{uniform_c2}}
\begin{proof}
For simplicity, we will use the following shortcut notation
\[
T_L := T_L^{\mathrm{soft},\mu}(\mathbb X)_1,
\qquad
T_\infty := T_{\infty}^{\mathrm{soft},\mu}(X_1).
\]

\medskip
\noindent

    We first note that Assertion \ref{hyp1} is a direct consequence of Lemma \ref{lemma_concentration}. Besides, Assertion \ref{hyp2} directly follows from \[\sup_{\mu\in B(0,\rho)}\Vert D_{\mu}^k T_\infty\Vert_F\leq C_k'\Vert X_1\Vert,\]
for $C_0'=\lambda\Vert\Sigma\Vert\rho^2, C_1'=2\lambda\Vert\Sigma\Vert\rho, C_2'=2\lambda\Vert\Sigma\Vert$. Defining 
\[C_k=(C_k')^2\Tr(\Sigma), \quad k\in \{0,1,2\},\] we get the required. 

Now we prove Assertion \ref{hyp3}, from Assertion \ref{hyp1} with $k=0$ and Assertion \ref{hyp2},
\[
\E[\|T_L\|^2]
\le
2\E[\|T_L-T_\infty\|^2]
+2\E[\|T_\infty\|^2]
\le
2\psi_0(L)+2C_0.
\]
Hence $\sup_{L,\mu}\E[\|T_L\|^2]<\infty$.
The same argument using $k=1,2$ gives uniform $L^2$-bounds for
$D_\mu T_L$ and $D_\mu^2 T_L$.
\medskip

\noindent
\paragraph{Case $k=0$.}

Using
\[
\|a\|^2-\|b\|^2
=2\langle b,a-b\rangle+\|a-b\|^2,
\]
we obtain
\[
\mathcal R_{\mathrm{soft},L}-\mathcal R_{\mathrm{soft},\infty}
=
2\E[\langle X_1-T_\infty, T_\infty-T_L\rangle]
+\E[\|T_L-T_\infty\|^2].
\]
By Cauchy--Schwarz and the uniform $L^2$-bounds,
\[
|\mathcal R_{\mathrm{soft},L}-\mathcal R_{\mathrm{soft},\infty}|
\le
C\E[\|T_L-T_\infty\|^2]^{1/2}.
\]
Taking the supremum over $\mu$ and squaring gives
\[
\sup_{\mu\in B(0,\rho)}|\mathcal R_{\mathrm{soft},L}-\mathcal R_{\mathrm{soft},\infty}|^2
\le C\psi_0(L).
\]

\medskip
\noindent
\paragraph{Case $k=1$.}

Differentiating under the expectation gives
\[
\nabla \mathcal R_{\mathrm{soft},L}
=
-2 \E[(X_1-T_L)^\top D_\mu T_L].
\]
Hence
\[
\nabla \mathcal R_{\mathrm{soft},L}-\nabla \mathcal R_{\mathrm{soft},\infty}
=
-2\E[(X_1-T_L)^\top(D_\mu T_L-D_\mu T_\infty)]
-2\E[(T_\infty-T_L)^\top D_\mu T_\infty].
\]
Each term is bounded using Cauchy--Schwarz and the uniform $L^2$ bounds:
\[
\|\nabla \mathcal R_{\mathrm{soft},L}-\nabla \mathcal R_{\mathrm{soft},\infty}\|_F^2
\le
C\big(
\E[\|D_\mu T_L-D_\mu T_\infty\|_F^2]
+
\E[\|T_L-T_\infty\|^2]
\big).
\]
Taking the supremum over $\mu$ yields
\[
\sup_{\mu\in B(0,\rho)}
\|\nabla \mathcal R_{\mathrm{soft},L}-\nabla \mathcal R_{\mathrm{soft},\infty}\|_F^2
\le
C\big(\psi_1(L)+\psi_0(L)\big).
\]

\medskip
\noindent
\paragraph{Case $k=2$.}

We compute
\[
\nabla^2 \mathcal R_{\mathrm{soft},L}
=
2\E[(D_\mu T_L)^\top(D_\mu T_L)]
-
2\E[(X_1-T_L)^\top D_\mu^2 T_L].
\]

and the analogous formula for $\mathcal R_{\mathrm{soft},\infty}$.
Set
\[
\Delta T := T_L-T_\infty, 
\qquad
\Delta DT := D_\mu T_L-D_\mu T_\infty,
\qquad
\Delta D^2T := D_\mu^2 T_L-D_\mu^2 T_\infty.
\]

\medskip
\noindent
Regarding the first term of the Hessian, we rewrite it as

\[
(D_\mu T_L)^\top(D_\mu T_L)
-
(D_\mu T_\infty)^\top(D_\mu T_\infty)
=
(\Delta DT)^\top(\Delta DT)
+
(D_\mu T_\infty)^\top(\Delta DT)
+
(\Delta DT)^\top(D_\mu T_\infty).
\]

Therefore,
\begin{align*}
&\E\!\big[(D_\mu T_L)^\top(D_\mu T_L)\big]
-
\E\!\big[(D_\mu T_\infty)^\top(D_\mu T_\infty)\big] \\
&=
\E\!\big[(\Delta DT)^\top(\Delta DT)\big]
+
\E\!\big[(D_\mu T_\infty)^\top(\Delta DT)\big]
+
\E\!\big[(\Delta DT)^\top(D_\mu T_\infty)\big].
\end{align*}

\medskip
\noindent
With respect to the second term, we have
\[
(X_1-T_L)^\top D_\mu^2 T_L
-
(X_1-T_\infty)^\top D_\mu^2 T_\infty
=
A_1+A_2,
\]
where
\[
A_1
:=
(X_1-T_L)^\top(D_\mu^2 T_L-D_\mu^2 T_\infty),
\]
\[
A_2
:=
\big[(X_1-T_L)-(X_1-T_\infty)\big]^\top D_\mu^2 T_\infty.
\]

Since $(X_1-T_L)-(X_1-T_\infty)=T_\infty-T_L=-\Delta T$, we obtain
\[
A_2
=
-(\Delta T)^\top D_\mu^2 T_\infty.
\]

Hence, 
\begin{align*}
&\E\!\big[(X_1-T_L)^\top D_\mu^2 T_L\big]
-
\E\!\big[(X_1-T_\infty)^\top D_\mu^2 T_\infty\big] \\
&=
\E\!\big[(X_1-T_L)^\top \Delta D^2T\big]
-
\E\!\big[(\Delta T)^\top D_\mu^2 T_\infty\big].
\end{align*}

\medskip
\noindent

Combining the previous expansions,
\begin{align*}
\nabla^2 \mathcal R_{\mathrm{soft},L}
-
\nabla^2 \mathcal R_{\mathrm{soft},\infty}
=
&\;2\E[(\Delta DT)^\top(\Delta DT)] \\
&+2\E[(D_\mu T_\infty)^\top(\Delta DT)]
+2\E[(\Delta DT)^\top(D_\mu T_\infty)] \\
&-2\E[(X_1-T_L)^\top \Delta D^2T] \\
&+2\E[(\Delta T)^\top D_\mu^2 T_\infty].
\end{align*}

\medskip
\noindent

Using $\| \E [Z] \|_F^2 \le \E[\|Z\|_F^2]$ and Cauchy--Schwarz:

\[
\|\E[(\Delta DT)^\top(\Delta DT)]\|_F
\le
\E[\|\Delta DT\|_F^2],
\]

\[
\|\E[(D_\mu T_\infty)^\top(\Delta DT)]\|_F
\le
\E[\|D_\mu T_\infty\|_F^2]^{1/2} \E[\|\Delta DT\|_F^2]^{1/2},
\]

\[
\|\E[(X_1-T_L)^\top \Delta D^2T]\|_F
\le
\E[\|X_1-T_L\|^2]^{1/2}
\E[\|\Delta D^2T\|_F^2]^{1/2},
\]

\[
\|\E[(\Delta T)^\top D_\mu^2 T_\infty]\|_F
\le
\E[\|\Delta T\|^2]^{1/2}\E[\|D_\mu^2 T_\infty\|_F^2]^{1/2}.
\]

By Assertion \ref{hyp2}, for $\mu\in B(0,\rho)$ there exists $C>0$ such that
\[
\|\nabla^2 \mathcal R_{\mathrm{soft},L}-\nabla^2 \mathcal R_{\mathrm{soft},\infty}\|_F^2
\le
C\big(
\E[\|T_L-T_\infty\|^2]
+
\E[\|D_\mu T_L-D_\mu T_\infty\|_F^2]
+
\E[\|D_\mu^2 T_L-D_\mu^2 T_\infty\|_F^2]
\big).
\]
Taking the supremum over $\mu$ concludes the proof.
\end{proof}
\subsection{Proof of Proposition \ref{prop:bounded-finite-flow}}
\begin{proof}
By Lemma \ref{coercive_bounded}, if $\Vert\mu\Vert\geq \rho$, then $\mathcal R_{\mathrm{soft},\infty}(\mu)>\mathcal R_{\mathrm{soft},\infty}(\mu_0)$. Thus, for every $\rho'>\rho$ we have that if $\Vert\mu\Vert=\rho'>\rho$ then \[\mathcal R_{\mathrm{soft},\infty}(\mu) > \mathcal R_{\mathrm{soft},\infty}(\mu_0),\] by compactness of $\{\mu\in\mathbb{R}^d:\Vert\mu\Vert=\rho'\}$ and continuity of $\mathcal R_{\mathrm{soft},\infty}$, thus \[\min_{\Vert\mu\Vert=\rho'}\mathcal R_{\mathrm{soft},\infty}(\mu) > \mathcal R_{\mathrm{soft},\infty}(\mu_0),\] so there exists $\varepsilon>0$ such that \[\min_{\Vert\mu\Vert=\rho'}\mathcal R_{\mathrm{soft},\infty}(\mu) > \mathcal R_{\mathrm{soft},\infty}(\mu_0)+2\varepsilon,\] and then \begin{equation}\label{coercive}
\|\mu\| = \rho' \quad \Longrightarrow \quad
\mathcal R_{\mathrm{soft},\infty}(\mu) > \mathcal R_{\mathrm{soft},\infty}(\mu_0) + 2\varepsilon.    
\end{equation}
Fix any $\rho'>\rho$, then there exists $\varepsilon>0$ that satisfies \eqref{coercive}. By Proposition \ref{uniform_c2}, we have uniform convergence of $\mathcal R_{\mathrm{soft},L}$
to $\mathcal R_{\mathrm{soft},\infty}$ on $B(0,\rho')$, there exists
$L'\in\mathbb{N}$ such that for all $L\ge L'$,
\[
\sup_{\|\mu\|\le \rho'} 
\bigl|
\mathcal R_{\mathrm{soft},L}(\mu)
-
\mathcal R_{\mathrm{soft},\infty}(\mu)
\bigr|
\le \varepsilon.
\]
We argue by contradiction. Assume that for some $L\ge L'$, the trajectory
$\mu_L$ does not remain in $B(0,\rho')$ for every $t\geq 0$. Then, there exists a first exit time
$t^\ast>0$ such that
\[
\|\mu_L(t^\ast)\|=\rho',
\qquad
\|\mu_L(t)\|<\rho' \ \text{for all } t<t^\ast.
\]
Since $\mathcal R_{\mathrm{soft},L}$ decreases along the flow,
\[
\mathcal R_{\mathrm{soft},L}(\mu_L(t^\ast))
\le
\mathcal R_{\mathrm{soft},L}(\mu_0).
\]
Using uniform convergence of the risk on $B(0,\rho')$ at $\mu_L(t^\ast)$ and at $\mu_0$ gives
\[
\mathcal R_{\mathrm{soft},\infty}(\mu_L(t^\ast))
\le \mathcal R_{\mathrm{soft},L}(\mu_L(t^\ast)) + \varepsilon
\le \mathcal R_{\mathrm{soft},L}(\mu_0) + \varepsilon
\le \mathcal R_{\mathrm{soft},\infty}(\mu_0) + 2\varepsilon.
\]

This is a contradiction with \eqref{coercive}. Hence, for all $L\ge L'$, the trajectory $\mu_L$ remains in $B(0,\rho')$ for all $t\ge 0$, proving the lemma.
\end{proof}
\subsection{Proof of Proposition \ref{conv_traj}}
\begin{proof}
   Fix $T>0$, by Proposition Lemma \ref{prop:bounded-finite-flow}, there exists $\rho>0$ and $L'\in\mathbb{N}$, such that for every $L\geq L'$ the trajectories
  $\mu_L(t)$ and $\mu_\infty(t)$ of \eqref{gfl} and \eqref{gfi} satisfy
  $\mu_L(t),\mu_\infty(t)\in B(0,\rho)$ for all $t\in[0,T]$.
  Besides, in Proposition \ref{uniform_c2} we have shown that for $k\in \{0,1,2\}$ $$\nabla^k \mathcal R_{\mathrm{soft},L}\underset{L\to\infty}{\to}\nabla^k \mathcal R_{\mathrm{soft},\infty}$$
  uniformly on $B(0,\rho)$, and $\nabla^2 \mathcal R_{\mathrm{soft},L}$
  are uniformly bounded on $B(0,\rho)$ (since it converges
  uniformly). Hence, there exists a constant $K>0$ such that for all
  $L\geq L'$ and all $x,y\in B(0,\rho)$
  \[
    \|\nabla \mathcal R_{\mathrm{soft},L}(x)-\nabla \mathcal R_{\mathrm{soft},L}(y)\|
    \le K\|x-y\|,
    \qquad
    \|\nabla \mathcal R_{\mathrm{soft},\infty}(x)-\nabla \mathcal R_{\mathrm{soft},\infty}(y)\|
    \le K\|x-y\|.
  \]
  Set
  \[
    \zeta(L):=\sup_{x\in B(0,\rho)}\|\nabla \mathcal R_{\mathrm{soft},L}(x)
    -\nabla \mathcal R_{\mathrm{soft},\infty}(x)\|.
  \]
  By uniform convergence $\zeta(L)\to0$ as $L\to\infty$. Consider the difference $e_L(t):=\mu_L(t)-\mu_\infty(t)$. Differentiating and using
  the definitions \eqref{gfl}, \eqref{gfi} we obtain
  \[
    e_L'(t) = -\nabla \mathcal R_{\mathrm{soft},L}(\mu_L(t))
             +\nabla \mathcal R_{\mathrm{soft},\infty}(\mu_\infty(t)).
  \]
  Using the triangle inequality and the Lipschitz bound,
  \begin{align*}
    \frac{\mathrm d}{\mathrm d t}\|e_L(t)\|
    &\le \|e_L'(t)\|
      \\
      &\le \|\nabla \mathcal R_{\mathrm{soft},L}(\mu_L(t))
        -\nabla \mathcal R_{\mathrm{soft},L}(\mu_\infty(t))\|+\|\nabla \mathcal R_{\mathrm{soft},L}(\mu_\infty(t))
        -\nabla \mathcal R_{\mathrm{soft},\infty}(\mu_\infty(t))\|\\
    &\le K\|e_L(t)\| + \zeta(L).
  \end{align*}
  This differential inequality together with $e_L(0)=0$, let us integrate  to obtain
  \[
    \|e_L(t)\| \le \zeta(L) \int_0^t e^{K(t-s)}\,\mathrm ds
    = \zeta(L) \frac{e^{Kt}-1}{K}.
  \]
  Therefore, for every $t\in[0,T]$,
  \[
    \sup_{0\le s\le T}\|\mu_L(s)-\mu_\infty(s)\|
    \le \zeta(L) \frac{e^{KT}-1}{K} \xrightarrow[L\to\infty]{} 0,
  \]
  which proves that $\mu_L\to\mu_\infty$ uniformly on $[0,T]$.
\end{proof}
\subsection{Proof of Proposition \ref{conv_values}}
\begin{proof}
We use the following bound
    $$\Vert \mathcal R_{\mathrm{soft},L}(\mu_L)-\mathcal R_{\mathrm{soft},\infty}(\mu_{\infty})\Vert\leq \Vert \mathcal R_{\mathrm{soft},L}(\mu_L)-\mathcal R_{\mathrm{soft},\infty}(\mu_L)\Vert+\Vert \mathcal R_{\mathrm{soft},\infty}(\mu_L)-\mathcal R_{\mathrm{soft},\infty}(\mu_{\infty})\Vert.$$
    The first term goes to 0 uniformly on $[0,T]$ as $\mu_L$ is bounded by Proposition Lemma \ref{prop:bounded-finite-flow} and Proposition \ref{uniform_c2} for $k=0$. The second term goes to 0 uniformly on $[0,T]$ as $\mathcal R_{\mathrm{soft},\infty}$ is $C^2$, combined with Proposition \ref{conv_traj}.
\end{proof}
\subsection{Proof of Proposition \ref{criticalpointsl}}
\begin{proof}
    We use the characterization of critical points for $\mathcal{R}_{\mathrm{soft},\infty}$ given in Proposition \ref{landscapeinf}, this gives us $\mathrm{crit}(\mathcal{R}_{\mathrm{soft},\infty})=\{\mu^{(1)},\ldots,\mu^{(2d+1)}\}$. Let $f_L=\mathcal{R}_{\mathrm{soft},L}$ and $f=\mathcal{R}_{\mathrm{soft},\infty}$,  Proposition \ref{uniform_c2} shows convergence in $C_{loc}^2$ of $f_L$ towards $f$ as $L\rightarrow\infty$. Then we use Lemma \ref{persistence} with $f_L$ and $f$, this gives us $r_k>0$ such that $B(\mu^{(k)},r_k)$ are pairwise disjoint and that for every $\rho>0$ big enough such that $\cup_{k=1}^{2d+1} B(\mu^{(k)},r_k)\subset B(0,\rho)$. Then for big enough $L$, $$\mathrm{crit}(f_L)\cap B(0,\rho)=\{\mu_L^{(1)},\ldots,\mu_L^{(\Lambda)}\},$$ with $\mu_L^{(k)}\rightarrow\mu^{(k)}$ as $L\rightarrow\infty$ for every $k\in\{1,\ldots,2d+1\}$, and if $\mu^{(k)}$ is a strict saddle point (resp. local minimum) for $f$, then $\mu_L^{(k)}$ is a strict saddle point (resp. local minimum) for $f_L$. 
\end{proof}
\begin{comment}
\subsection{Proof of Proposition \ref{globall}}
\begin{proof}
Let $\mathrm{crit}(\mathcal{R}_{\mathrm{soft},\infty})
=\{\mu^{(1)},\ldots,\mu^{(2d+1)}\}$. 
From the proof of Proposition~\ref{criticalpointsl}, there exist $r_1,\ldots,r_{2d+1}>0$ and $\rho_1>0$ such that 
$\cup_{k=1}^{2d+1} B(\mu^{(k)},r_k)\subset B(0,\rho_1)$. Moreover, by Proposition~\ref{prop:bounded-finite-flow}, there exist $\rho_2>0$ and $L_0\in\mathbb{N}$ such that for all $L\ge L_0$ the gradient flow trajectory $\mu_L(t)$ generated by \eqref{gfl} remains in $B(0,\rho_2)$ for every $t\ge0$. Let $\rho=\max\{\rho_1,\rho_2\}$. Then for $L$ large enough, every critical point reachable by the dynamics \eqref{gfl} lies in $B(0,\rho)$ and must therefore coincide with one of the critical points characterized in Proposition~\ref{criticalpointsl}. Among these, the only stable ones are $\pm\mu_{L,\sigma_1}^\star$. The conclusion then follows as in the proof of Proposition~\ref{globalinf}.
\end{proof}
\end{comment}
\subsection{Proof of Proposition \ref{cr_l}}
\begin{proof}

For $L$ large enough, Proposition~\ref{criticalpointsl} ensures that the critical points of $\mathcal{R}_{\mathrm{soft},L}$ are finite and nondegenerate, with two local minimizers $\pm \mu_{L,\sigma_1}^\star$. Hence, for a generic initialization, the gradient flow converges to one of them.

The exponential decay of the objective follows from the Łojasiewicz inequality, while the convergence rate of the iterates is given by linearization around $\mu_L^\star$, yielding $\tilde{s}_L = \sigma_{\min}(\nabla^2 \mathcal{R}_{\mathrm{soft},L}(\mu_L^\star))>0$. The limit $\tilde{s}_L \to \tilde{s}$ follows from $C^2_{\mathrm{loc}}$ convergence of the risks.
\end{proof}

\subsection{Proof of Proposition \ref{distinf}}
\begin{proof}
The result follows from classical properties of linear transformations of Gaussian vectors, yielding
\[
T_{\infty}^{\mathrm{soft},\mu}(X_1) \sim \mathcal{N}(0,\Gamma(\mu)).
\]
The rank-one structure is immediate. Evaluating at $\mu^\star$ and using $\Sigma u_1=\sigma_1 u_1$ gives $\Gamma(\mu^\star)=\sigma_1 u_1 u_1^\top$, which corresponds to the law of $\langle X,u_1\rangle u_1$.
\end{proof}
\subsection{Proof of Proposition \ref{localcrgf}}
\begin{proof}

The convergence rate is analogous to the proof of Proposition \ref{cr_inf}. From algebraic manipulation we can check that for $a_i, b_i$ defined in \eqref{ab}, we have $$ \frac{(a_1+a_2)(a_3+b_3)}{a_1 + 2 a_2 + b_2}-a_3<2(a_3+b_3),$$
thus
\begin{equation}\label{hats}
    \begin{split}
         \hat{s} &= 2\,\frac{(a_1 + a_2) b_3 - a_3(a_2 + b_2)}{a_1 + 2 a_2 + b_2}\\
         &=\frac{2 \lambda n \theta \xi^2 \left( d n \xi^2 - d \xi^2 + n^{2} \theta + n^{2} \xi^2 + 5 n \theta + 4 n \xi^2 + 2 \theta + 3 \xi^2 \right)}{d \xi^2 + n \theta + n \xi^2 + 4 \theta + 3 \xi^2}\\
         &=\frac{2 \lambda n \theta \xi^2 \left[ \theta (n^2+5n+2) + \xi^2 (n^2 + dn + 4n - d + 3) \right]}
{(n+4)\theta + (d+n+3)\xi^2  }.
    \end{split}
\end{equation}

For $\xi^2$ small enough, this reduces to
\[
\hat{s} \sim 2 \frac{\lambda n(n^2+5n+2) \theta \, \xi^2}{n+4} .
\]

Hence the gradient flow converges locally with an exponential rate arbitrarily close to $\hat{s}$.
\end{proof}
\subsection{Proof of Proposition \ref{uniform_c2_icl}}
\begin{proof}
 Since we have proven in Proposition \ref{uniform_c2} that $\nabla^k\mathcal{R}_{\mathrm{soft},L}^{(\Sigma)}$ converges uniformly (as $L\rightarrow\infty$) on compact sets $\nabla^k\mathcal{R}_{\mathrm{soft},\infty}^{(\Sigma)}$ for $k\in\{0,1,2\}$, i.e., there exists $C(\Sigma)>0$, $$\psi(L,\Sigma)=L^{-\frac{1}{c^2\Vert\Sigma\Vert_2^2+1}}(1+ \ln L)^{\frac{c^2\Vert\Sigma\Vert_2^2}{c^2\Vert\Sigma\Vert_2^2+1}}$$ with $c=20\lambda\rho^2>0$, and $\lim_{L\rightarrow\infty} \psi(L,\Sigma)=0$, such that \[
\sup_{\mu\in B(0,\rho)}
\|
\nabla^k \mathcal R_{\mathrm{soft},L}(\mu)
-
\nabla^k \mathcal R_{\mathrm{soft},\infty}(\mu)
\|_F^2
\le
C(\Sigma) \psi(L,\Sigma).
\] 
 Thus for $k\in\{0,1,2\}$, since $\mathcal{R}_{\mathrm{soft},L}^{(\Sigma)}$ and $\mathcal{R}_{\mathrm{soft},\infty}^{(\Sigma)}$ are $C^2(\mathbb{R}^d)$ functions and $$\mathbb{E}_\Sigma\left[\sup_{\mu\in B(0,\rho)}\nabla^k \mathcal{R}_{\mathrm{soft},L}^{(\Sigma)}(\mu)\right]<\infty, \quad \mathbb{E}_\Sigma\left[\sup_{\mu\in B(0,\rho)}\nabla^k \mathcal{R}_{\mathrm{soft},\infty}^{(\Sigma)}(\mu)\right]<\infty.$$ By the dominated convergence theorem we have that $$\nabla^k\mathcal{R}_L^{\mathrm{ICL}}(\mu)=\mathbb{E}_{\Sigma}[\nabla^k\mathcal{R}_{\mathrm{soft},L}^{(\Sigma)}(\mu)], \quad \nabla^k \mathcal{R}_\infty^{\mathrm{ICL}}(\mu)=\mathbb{E}_\Sigma[\nabla^k\mathcal{R}_{\mathrm{soft},\infty}^{(\Sigma)}(\mu)],$$ and then get \begin{align*}
    \sup_{\mu\in B(0,\rho)}\Vert \nabla^k \mathcal{R}_L^{\mathrm{ICL}}(\mu)- \nabla^k \mathcal{R}_\infty^{\mathrm{ICL}}(\mu)\Vert_F^2&=\sup_{\mu\in B(0,\rho)}\Vert \nabla^k \mathbb{E}_{\Sigma}[\mathcal{R}_{\mathrm{soft},L}^{(\Sigma)}(\mu)]- \nabla^k \mathbb{E}_{\Sigma}[\mathcal{R}_{\mathrm{soft},\infty}^{(\Sigma)}(\mu)]\Vert_F^2\\
    &\leq \sup_{\mu\in B(0,\rho)}\Vert \mathbb{E}_{\Sigma}[\nabla^k \mathcal{R}_{\mathrm{soft},L}^{(\Sigma)}(\mu)- \nabla^k \mathcal{R}_{\mathrm{soft},\infty}^{(\Sigma)}(\mu)]\Vert_F^2\\
    &\leq \mathbb{E}_{\Sigma}\left[\sup_{\mu\in B(0,\rho)}\Vert \nabla^k \mathcal{R}_{\mathrm{soft},L}^{(\Sigma)}(\mu)- \nabla^k \mathcal{R}_{\mathrm{soft},\infty}^{(\Sigma)}(\mu)\Vert_F^2\right]\\
    &\leq \mathbb{E}_{\Sigma\sim W_d(V,n)}[C(\Sigma) \psi(L,\Sigma)].
\end{align*}
Now, we define $\zeta(\Sigma)=C(\Sigma)\sup_{L> 1} \psi(L,\Sigma)$, it is direct to show that $$\sup_{L> 1} \psi(L,\Sigma)\leq 1+c^2\Vert\Sigma\Vert_2^2,$$ and then $$\mathbb{E}_{\Sigma\sim W_d(V,n)}[\zeta(\Sigma)]\leq \mathbb{E}_{\Sigma\sim W_d(V,n)}[C(\Sigma)(1+c^2\Vert\Sigma\Vert_2^2)].$$ Since $C(\Sigma)$ depends polynomially on $\Vert\Sigma\Vert_2$ and that for the Wishart distribution every finite moment is bounded, we get that $\mathbb{E}_{\Sigma\sim W_d(V,n)}[\zeta(\Sigma)]<\infty$, by dominated convergence theorem we get that $\lim_{L\rightarrow\infty}\mathbb{E}_{\Sigma\sim W_d(V,n)}[C(\Sigma) \psi(L,\Sigma)]=\mathbb{E}_{\Sigma\sim W_d(V,n)}[C(\Sigma) \lim_{L\rightarrow\infty}\psi(L,\Sigma)]=0$, and then for $k\in\{0,1,2\}$, \begin{align*}
    \lim_{L\rightarrow\infty}\sup_{\mu\in B(0,\rho)}\Vert \nabla^k \mathcal{R}_L^{\mathrm{ICL}}(\mu)- \nabla^k \mathcal{R}_\infty^{\mathrm{ICL}}(\mu)\Vert_F^2=0\end{align*}

\end{proof}

\subsection{Proof of Proposition \ref{prop:gd_convergence_icl}}
\begin{proof}
    Without loss of generality, by the arguments of Proposition \ref{prop:bounded-finite-flow}, we can assume that the trajectory remains bounded and contained in $B(0,\rho)$.  In particular, all critical points in $B(0,\rho)$ are those identified in Proposition~\ref{criticalpointsicl}, and the only stable ones are the two local minimizers $\pm \mu_{L,\parallel}^\star$. Applying the Stable Manifold Theorem \cite[Theorem III.7]{shub}, we conclude that, for almost every initialization in $B(0,\rho)$, the trajectory converges to one of these two minimizers. The convergence rates come from the same reasoning as for Proposition \ref{cr_l}.
\end{proof}
\section{Technical results}\label{sec:lemmas}
In this section, we present the technical results required for our analysis.
\subsection{Gaussian computations}
We compute specific higher-order moments required for the construction of the objective function.

\begin{proposition}\label{computations}
Let $X_1,X_2 \stackrel{\mathrm{iid}}{\sim} \mathcal N(0,\Sigma)$ in $\mathbb R^d$ and let $\mu\in\mathbb R^d$. We have
\begin{enumerate}
    \item $\E[\Norm{X_1}^2]=\Tr(\Sigma)$.
    \item $\E\left[ (X_1^\top\mu)^2 \right] = \mu^T \Sigma \mu.$
    \item $\E\left[ \Inner{X_1}{X_2} \Inner{X_1}{\mu} \Inner{X_2}{\mu} \right] = \mu^T \Sigma^2 \mu.$
    \item $\E\left[ \Inner{X_1}{X_2} \Inner{X_1}{\mu}^3 \Inner{X_2}{\mu} \right] = 3(\mu^T \Sigma \mu)(\mu^T \Sigma^2 \mu).$
    \item $\E\left[ \Inner{X_1}{\mu}^2 \Norm{X_1}^2 \right] = \Tr(\Sigma)(\mu^T \Sigma \mu) + 2 \mu^T \Sigma^2 \mu.$
    \item $\E\left[ \Inner{X_1}{\mu}^4\Norm{X_1}^2 \right] = 3\Tr(\Sigma)(\mu^T \Sigma \mu)^2 + 12 (\mu^T \Sigma \mu)(\mu^T \Sigma^2 \mu).$
    
\end{enumerate}

\end{proposition}
\begin{proof}
All identities follow from linearity of expectation, independence of $X_1$ and $X_2$, and Isserlis' theorem \citep{isserlis} for centered Gaussian vectors.
\begin{enumerate}
    \item  Since $\|X_1\|^2=\sum_{i=1}^d X_{1,i}^2$ and $\E[X_{1,i}^2]=\Sigma_{ii}$,
\[
\E[\|X_1\|^2]=\sum_{i=1}^d \Sigma_{ii}=\Tr(\Sigma).
\]

\item Expanding $(X_1^\top\mu)^2$ and using $\E[X_{1,i}X_{1,j}]=\Sigma_{ij}$ gives
\[
\E[(X_1^\top\mu)^2]
= \sum_{i,j}\mu_i\mu_j\Sigma_{ij}
= \mu^\top\Sigma\mu.
\]

\item Conditioning on $X_1$ and using independence,
\[
\E[\Inner{X_1}{X_2}\Inner{X_2}{\mu}\mid X_1]
= X_1^\top \E[X_2X_2^\top]\mu
= X_1^\top\Sigma\mu.
\]
Therefore,
\[
\E[\Inner{X_1}{X_2}\Inner{X_1}{\mu}\Inner{X_2}{\mu}]
= \E[(X_1^\top\mu)(X_1^\top\Sigma\mu)]
= \mu^\top\Sigma^2\mu.
\]

\item Conditioning again on $X_1$,
\[
\E[\Inner{X_1}{X_2}\Inner{X_1}{\mu}^3\Inner{X_2}{\mu}]
= \E[(X_1^\top\mu)^3(\mu^\top\Sigma X_1)].
\]
For a centered Gaussian vector $X$ and vectors $\mu_0,\mu_1$, Isserlis' theorem yields
\[
\E[(\mu_0^\top X)^3(\mu_1^\top X)]
= 3(\mu_0^\top\Sigma \mu_0)(\mu_0^\top\Sigma \mu_1).
\]
Applying this with $a=\mu_0$ and $\mu_1=\Sigma\mu$ gives the claim.

\item Expanding
\[
\Inner{X_1}{\mu}^2\Norm{X_1}^2
= \sum_{i,j,k}\mu_i\mu_j X_{1,i}X_{1,j}X_{1,k}^2
\]
and applying Isserlis' theorem to the fourth-order moment yields two types of pairings, leading to
\[
\E[\Inner{X_1}{\mu}^2\Norm{X_1}^2]
= \Tr(\Sigma)(\mu^\top\Sigma\mu) + 2\mu^\top\Sigma^2\mu.
\]

\item Similarly,
\[
\Inner{X_1}{\mu}^4\Norm{X_1}^2
= \sum_{i,j,\ell,m,k}\mu_i\mu_j\mu_\ell\mu_m
X_{1,i}X_{1,j}X_{1,\ell}X_{1,m}X_{1,k}^2.
\]
Applying Isserlis' theorem to the sixth-order moment, pairings where the two copies of $k$ are paired together contribute
$3\Tr(\Sigma)(\mu^\top\Sigma\mu)^2$, while the remaining $12$ pairings contribute
$12(\mu^\top\Sigma\mu)(\mu^\top\Sigma^2\mu)$. Summing both terms gives the result.
\end{enumerate}
\end{proof}

\subsection{Optimization preliminaries}
%In most of this work, we will be able to explicitly characterize the set of critical points of the risk functions we study. The following lemma will therefore be useful to identify the global minimizers.

In this subsection, we gather several optimization results that will be used throughout this work.

\begin{lemma}\label{lem:localtoglobal}
Let $f:\mathbb{R}^d\to\mathbb{R}$ be a $C^1(\mathbb{R}^d)$ coercive function. Assume that the set of critical points of $f$ is completely characterized and every critical point is either a strict saddle or a local minimum. Assume moreover that all local minima attain the same value $m$. Then every local minimum of $f$ is a global minimum, and
\[
\min_{x\in\mathbb{R}^d} f(x)=m.
\]
\end{lemma}

\begin{proof}
Since $f$ is coercive and continuous, it attains its global minimum at some point $x^\star\in\mathbb{R}^d$. Since $f$ is $C^1$, any global minimizer satisfies $\nabla f(x^\star)=0$, hence $x^\star$ is a critical point. By assumption, every critical point is either a strict saddle or a local minimum. A strict saddle cannot be a minimizer, therefore $x^\star$ must be a local minimum. Since all local minima have value $m$, we obtain
\[
f(x^\star)=m.
\]
Thus $m=\min f$, and every local minimum is a global minimum.
\end{proof}

\begin{lemma}\label{nondeg12}
Let $f:\mathbb{R}^d \to \mathbb{R}$ be $C^2$ and let $x^\star$ be a nondegenerate critical point, i.e.,
\[
\nabla f(x^\star)=0
\quad \text{and} \quad
\nabla^2 f(x^\star) \text{ is invertible}.
\]
Then there exist constants $C>0$ and a neighborhood $U$ of $x^\star$ such that
\[
\|\nabla f(x)\|
\ge
C\,|f(x)-f(x^\star)|^{1/2}
\quad \text{for all } x \in U.
\]
\end{lemma}
\begin{proof}
 Since $\nabla^2 f(x^\star)$ is invertible, there exists $c_1>0$ such that
\[
\|\nabla^2 f(x^\star)v\|\ge c_1\|v\|
\quad \text{for all } v\in\mathbb{R}^d.
\]

By Taylor's theorem with integral remainder,
\[
\nabla f(x)
=
\nabla^2 f(x^\star)(x-x^\star)
+
r(x),
\]
where
\[
r(x)
=
\int_0^1 \big(\nabla^2 f(x^\star+t(x-x^\star))-\nabla^2 f(x^\star)\big)(x-x^\star)\,dt.
\]

Since $\nabla^2 f$ is continuous, for any $\varepsilon>0$ there exists $\delta>0$ such that for all $\|x-x^\star\|\le\delta$,
\[
\|\nabla^2 f(x)-\nabla^2 f(x^\star)\|\le \varepsilon.
\]
Hence
\[
\|r(x)\|\le \varepsilon \|x-x^\star\|.
\]

Thus
\[
\|\nabla f(x)\|
\ge
\|\nabla^2 f(x^\star)(x-x^\star)\| - \|r(x)\|
\ge
(c_1-\varepsilon)\|x-x^\star\|.
\]

Choosing $\varepsilon=c_1/2$, we obtain
\[
\|\nabla f(x)\|
\ge
c_2\|x-x^\star\|
\]
for $c_2=\frac{c_1}{2}>0$ and all $\|x-x^\star\|\le\delta$.

Next, by Taylor expansion of $f$,
\[
f(x)-f(x^\star)
=
\frac{1}{2}(x-x^\star)^T \nabla^2 f(x^\star) (x-x^\star)
+
\tilde r(x),
\]
where $|\tilde r(x)|\le \varepsilon \|x-x^\star\|^2$ for $\|x-x^\star\|$ small enough.

Since $\nabla^2 f(x^\star)$ is bounded, there exists $c_3>0$ such that
\[
|(x-x^\star)^T \nabla^2 f(x^\star) (x-x^\star)|
\le c_3 \|x-x^\star\|^2.
\]
Thus
\[
|f(x)-f(x^\star)|
\le
c_4 \|x-x^\star\|^2
\]
for some $c_4>0$.

Combining both estimates yields
\[
\|\nabla f(x)\|
\ge
c_2 \|x-x^\star\|
\ge
\frac{c_2}{\sqrt{c_4}} |f(x)-f(x^\star)|^{1/2}.
\]
\end{proof}

\begin{corollary}\label{cor1}
Let $f:\mathbb{R}^d \to \mathbb{R}$ be real analytic and assume that all its critical points are nondegenerate and finite. Then for each critical point $x^\star$, there exist $C>0$ and a neighborhood $U$ of $x^\star$ such that
\[
\|\nabla f(x)\| \ge C |f(x)-f(x^\star)|^{1/2}
\quad \text{for all } x \in B(x^\star,\varepsilon).
\]
\end{corollary}

\begin{lemma}\label{bestbound}
Let $\alpha,\beta>0$ and $L\ge 1$. Define, for $\Psi\ge 0$,
\[
\varphi_L(\Psi)
\;:=\;
\frac{e^{\alpha \Psi^2}}{L}
\;+\;
(1+\ln L) e^{-\beta \Psi^2}.
\]
Then $\varphi_L$ admits a unique minimizer $\Psi(L)\ge 0$, given by
\[
\Psi^2(L)
=
\frac{1}{\alpha+\beta}
\Bigl(
\ln L + \ln(1+\ln L) + \ln(\beta/\alpha)
\Bigr).
\]
Moreover, the optimal value satisfies
%\[
%\inf_{\Psi\ge 0} \varphi_L(\Psi)
%=
%\Theta\!\left(
%L^{-\frac{\beta}{\alpha+\beta}}
%(1+\ln L)^{\frac{\alpha}{\alpha+\beta}}
%\right),
%\]
%and in particular,
\[
\inf_{\Psi\ge 0} \varphi_L(\Psi)
=
\Theta\!\left(
\psi_{\alpha,\beta}(L)
\right),
\]
where $\psi_{\alpha,\beta}(L)=L^{-\frac{\beta}{\alpha+\beta}}
(1+\ln L)^{\frac{\alpha}{\alpha+\beta}}$.
\end{lemma}

\begin{proof}
Set $x=\Psi^2\ge 0$ and define
\[
f_L(x)=\frac{e^{\alpha x}}{L}+(1+\ln L)e^{-\beta x}.
\]
The function $f_L$ is strictly convex on $\mathbb{R}_+$ %since
%\[
%f_L''(x)
%=
%\frac{\alpha^2 e^{\alpha x}}{L}
%+
%\beta^2(1+\ln L)e^{-\beta x}
%>
%0.
%\]
Therefore, any critical point is the unique global minimizer. Differentiating,
\[
f_L'(x)
=
\frac{\alpha e^{\alpha x}}{L}
-
\beta(1+\ln L)e^{-\beta x}.
\]
Solving $f_L'(x)=0$ yields
\[
e^{(\alpha+\beta)x}
=
\frac{\beta}{\alpha} L(1+\ln L),
\]
which gives the stated expression for $\Psi^2(L)$. Substituting this value into either term of $f_L$,
\[
\frac{e^{\alpha \Psi^2(L)}}{L}
=
L^{-\frac{\beta}{\alpha+\beta}}
(1+\ln L)^{\frac{\alpha}{\alpha+\beta}}
\left(\frac{\beta}{\alpha}\right)^{\frac{\alpha}{\alpha+\beta}},
\]
and the term $(1+\ln L)e^{-\beta\Psi^2(L)}$ has the same order. The claim follows.
\end{proof}

\begin{lemma}[Persistence of nondegenerate critical points]\label{persistence}
Let $f \in C^2(\mathbb{R}^d)$ and let $x^\star$ be a nondegenerate critical point of $f$, i.e.
\[
\nabla f(x^\star)=0,
\qquad 
\nabla^2 f(x^\star)\ \text{is invertible}.
\]
Assume $f_n \to f$ in $C^2_{\mathrm{loc}}(\mathbb{R}^d)$. 
Then there exist $r>0$ and $n_0\in\mathbb{N}$ such that for all $n \ge n_0$:

\begin{enumerate}
\item $f_n$ has a unique critical point $x_n^\star$ in $B(x^\star,r)$,
\item $x_n^\star$ is nondegenerate and moreover $x_n^\star \to x^\star$ as $n \to \infty$,
\item If the set of critical points of $f$ is finite, i.e.,
\[
\mathrm{crit}(f)=\{x^{(1)},\ldots,x^{(\Lambda)}\},
\]
and each $x^{(k)}$, $k\in\{1,\ldots,\Lambda\}$, is nondegenerate, let
$r_k>0$ be such that $B(x^{(k)},r_k)$ are pairwise disjoint. Let $\rho>0$ such that
\[
\bigcup_{k=1}^{\Lambda} B(x^{(k)},r_k)\subset B(0,\rho).
\]
Then for $n$ large enough,
\[
\mathrm{crit}(f_n)\cap B(0,\rho)
=
\{x_n^{(1)},\ldots,x_n^{(\Lambda)}\}.
\]
Moreover, if for some $k\in\{1,\ldots,\Lambda\}$, $x^{(k)}$ is a strict
saddle (resp. a strict local minimum), then $x_n^{(k)}$ is also a strict
saddle (resp. a strict local minimum) for $n$ large enough.
\end{enumerate}
\end{lemma}

\begin{proof}
\begin{enumerate}
    \item 

\textbf{Uniform invertibility of the Hessian.}
Since $\nabla^2 f(x^\star)$ is invertible and $\nabla^2 f$ is continuous, 
there exist $r>0$ and $m>0$ such that
\[
\|\nabla^2 f(x)^{-1}\| \le m
\quad \text{for all } x \in B(x^\star,r).
\]

Because $f_n \to f$ in $C^2_{\mathrm{loc}}$, for $n$ sufficiently large,
\[
\sup_{x\in B(x^\star,r)}
\|\nabla^2 f_n(x) - \nabla^2 f(x)\|
\le \frac{1}{2m}.
\]
Hence for $x \in B(x^\star,r)$, we define $E_n(x)=\nabla^2 f_n(x) - \nabla^2 f(x)$
\[
\nabla^2 f_n(x)
=
\nabla^2 f(x) + E_n(x),
\qquad 
\|E_n(x)\|\le \frac{1}{2m}.
\]
Then, for $n$ sufficiently large
$\nabla^2 f_n(x)$ is invertible and let $A(x)=\nabla^2 f(x)$, then
\[\nabla^2 f_n(x)=A(x)+E_n(x)\]
And we have the identity \[\nabla^2 f_n(x)^{-1}=(I_d+A(x)^{-1}E_n(x))^{-1}A(x)^{-1},\]
where $$\Vert A(x)^{-1}E_n(x)\Vert\leq \Vert A(x)^{-1}\Vert\Vert E_n(x)\Vert\leq m\frac{1}{2m}=\frac{1}{2}.$$
By the Neumann Series Theorem we get that \[\Vert(I_d+A(x)^{-1}E_n(x))^{-1}\Vert\leq \frac{1}{1-\frac{1}{2}}= 2.\]

Thus, for $n$ sufficiently large, we get
\[
\|\nabla^2 f_n(x)^{-1}\|=\|(I_d+A(x)^{-1}E_n(x))^{-1}\|\|A(x)^{-1}\|
\le 2m.
\]
\medskip
\textbf{Definition of the homotopy.}
Define
\[H_n(t,x)=
(1-t)\nabla f(x)+t \nabla f_n(x),\qquad t\in[0,1].\]
Then
\[
D_x H_n(t,x)=(1-t)\nabla^2 f(x)+t \nabla^2 f_n(x),\]
\[
D_t H_n(t,x)
=
\nabla f_n(x) - \nabla f(x).
\]
By the uniform invertibility of the Hessian, $D_x H_n(t,x)$ is invertible on
$[0,1]\times B(x^\star,r)$ for $n$ sufficiently large.\\
\medskip
\textbf{Technical condition.}
Let
\[
M := \sup_{(t,x)\in[0,1]\times B(x^\star,r)}\|(D_x H_n(t,x))^{-1}\|.
\]
From the uniform invertibility of $\nabla^2 f$ and $\nabla^2 f_n$, $M<\infty$ uniformly for $n$ large.

Since $f_n \to f$ in $C^1_{\mathrm{loc}}$,
\[
\sup_{x\in B(x^\star,r)}\|\nabla f_n(x)-\nabla f(x)\|\longrightarrow 0.
\]
Therefore, for $n$ sufficiently large,
\[
\sup_{(t,x)\in[0,1]\times B(x^\star,r)}
\|(D_x H_n(t,x))^{-1}D_t H_n(t,x)\|\le M\sup_{x\in B(x^\star,r)}\|\nabla f_n(x)-\nabla f(x)\|< r.
\]
\medskip

\textbf{Conclusion.} Since
\[
H_n(0,x^\star)=\nabla f(x^\star)=0,
\]
Then by \cite[Theorem 4.2.1]{implicit} yields a continuous curve $x_n(t)$ with
\[
H_n(t,x_n(t))=0,
\qquad 
x_n(0)=x^\star,
\qquad 
x_n(t)\in B(x^\star,r).
\]
Define $x_n^\star := x_n(1)$. Then
\[
\nabla f_n(x_n^\star)=0,
\qquad 
x_n^\star \in B(x^\star,r).
\]
Uniqueness follows from the local invertibility of $D_x H_n$.\\
\item 
Let $\Gamma:=\frac{1}{\Vert [\nabla^2 f(x^\star)]^{-1}\Vert}$, by continuity of $\nabla^2 f$, shrinking $r$ if necessary, we have for every $x\in B(x^\star, r)$, $$\Vert \nabla^2 f(x)-\nabla^2 f(x^\star)\Vert\leq \frac{\Gamma}{2},$$ and since $\Vert\nabla^2 f(x^\star)v\Vert\geq \Gamma \Vert v\Vert$, using Taylor we get $$\nabla f(x)=\nabla^2 f(x^\star)(x-x^\star)+\int_{0}^1 (\nabla^2 f(x^\star+t(x-x^\star))-\nabla^2 f(x^\star))(x-x^\star)dt.$$

We take norms, and after bounding, we obtain
\[
\|x-x^\star\|
\leq
\frac{2}{\Gamma}\|\nabla f(x)\|.
\]
Applying this estimate with \(x=x_n^\star\), we would need to show that $\Vert\nabla f(x_n^\star)\Vert\to 0$ in order to show that $x_n^\star\to x^\star$. Since \(\nabla f_n(x_n^\star)=0\),
\[
\nabla f(x_n^\star)
=
\nabla f(x_n^\star)-\nabla f_n(x_n^\star),
\]
and from \(f_n\to f\) in \(C^1_{\mathrm{loc}}(\mathbb R^d)\), we conclude that $x_n^\star\to x^\star$ as $n\to\infty$.

Besides,
\[
\begin{aligned}
\|\nabla^2 f_n(x_n^\star)-\nabla^2 f(x^\star)\|
&\leq
\|\nabla^2 f_n(x_n^\star)-\nabla^2 f(x_n^\star)\| +
\|\nabla^2 f(x_n^\star)-\nabla^2 f(x^\star)\|.
\end{aligned}
\]
Since \(f_n\to f\) in \(C^2_{\mathrm{loc}}(\mathbb R^d)\) and \(x_n^\star\to x^\star\),
\[
\nabla^2 f_n(x_n^\star)\to \nabla^2 f(x^\star).
\]
As \(\nabla^2 f(x^\star)\) is invertible, \(x_n^\star\) is nondegenerate for \(n\) sufficiently large.
\item Assume now that 
\[
\mathrm{crit}(f)=\{x^{(1)},\ldots,x^{(\Lambda)}\},
\qquad 
\nabla^2 f(x^{(k)}) \ \text{is invertible for all } k.
\]

Since the set is finite, define
\[
\delta:=\frac12\min_{k\neq \ell}\|x^{(k)}-x^{(\ell)}\|>0.
\]

Applying items (1)--(2) to each $x^{(k)}$, there exist $r_k\in(0,\delta)$ and 
$n_k\in\mathbb{N}$ such that for all $n\ge n_k$:

\begin{enumerate}
\item $f_n$ has a unique critical point $x_n^{(k)}$ in $B(x^{(k)},r_k)$,
\item $x_n^{(k)}\to x^{(k)}$,
\item $x_n^{(k)}$ is nondegenerate.
\end{enumerate}

Let 
\[
n_0:=\max_{k\in\{1,\ldots,\Lambda\}} n_k.
\]

For $n\ge n_0$ the balls $B(x^{(k)},r_k)$ are disjoint and
\[
\{x_n^{(1)},\ldots,x_n^{(\Lambda)}\}
\subset \mathrm{crit}(f_n).
\]

It remains to show that there are no other critical points in $B(0,\rho)$.

Let
\[
K:=\bigcup_{k=1}^{\Lambda}B(x^{(k)},r_k)\subset B(0,\rho),
\]
then \begin{equation}\label{inc1}
\{x_n^{(1)},\ldots,x_n^{(\Lambda)}\}
\subseteq \mathrm{crit}(f_n)\cap B(0,\rho).    
\end{equation}

Since $f$ has no critical point in $K^c\cap B(0,\rho)$ and $\nabla f$ is
continuous, the compactness of $K^c\cap B(0,\rho)$ implies
\[
\eta_\rho
:=
\inf_{x\in K^c\cap B(0,\rho)}
\|\nabla f(x)\|
>0.
\]

Because $f_n\to f$ in $C^1_{\mathrm{loc}}$, we have
\[
\sup_{x\in B(0,\rho)}
\|\nabla f_n(x)-\nabla f(x)\|\to0 .
\]

Hence for $n$ large enough and every $x\in K^c\cap B(0,\rho)$,
\[
\|\nabla f_n(x)\|
\ge
\|\nabla f(x)\|
-
\|\nabla f_n(x)-\nabla f(x)\|
\ge
\frac{\eta_\rho}{2}
>0 .
\]

Therefore $f_n$ has no critical point in $K^c\cap B(0,\rho)$.
Let $x\in \mathrm{crit}(f_n)\cap B(0,\rho)$, if $x\in K^c$ we get a contradiction with the previous founding. Then $x\in K$, and there exists $k\in \{1,...,\Lambda\}$ such that $x\in B(x^{(k)},r_k)$, since the only critical point of $f_n$ inside this ball is $x_n^{(k)}$,  we conclude that $x=x_n^{(k)}$ and
\[
\mathrm{crit}(f_n)\cap B(0,\rho)
\subseteq
\{x_n^{(1)},\ldots,x_n^{(\Lambda)}\},
\]
then by \eqref{inc1} we conclude the equality of sets.

Finally, since $x_n^{(k)}\to x^{(k)}$ and $f_n\to f$ in $C^2_{\mathrm{loc}}$,
\[
\nabla^2 f_n(x_n^{(k)}) \to \nabla^2 f(x^{(k)}).
\]
Because $\nabla^2 f(x^{(k)})$ is invertible, its eigenvalues are bounded away
from zero. By continuity of the spectrum, the inertia of
$\nabla^2 f_n(x_n^{(k)})$ coincides with that of $\nabla^2 f(x^{(k)})$ for
$n$ large enough. Hence if $x^{(k)}$ is a strict local minimum (resp. a
strict saddle), then $x_n^{(k)}$ is also a strict local minimum (resp. a
strict saddle) for $n$ large enough.

\end{enumerate}
\end{proof}
\subsection{Almost sure and \texorpdfstring{$L^2-$}{L2} convergence of encodings}\label{concentration}

In the following subsection, we present the lemmas needed to establish almost sure and $L^2$ convergence of $T_{\mathrm{soft},L}^\mu$ to $T_{\mathrm{soft},\infty}^\mu$. 

\begin{lemma}\label{lemma_as}
    Consider $X_1,\ldots,X_L$ i.i.d $\mathcal{N}(0,\Sigma)$, and $$T_L(X_1)=\frac{\sum_{k=1}^L X_k \mathrm{exp}(\lambda \langle X_1,\mu\rangle\langle X_k,\mu\rangle)}{\sum_{k=1}^L \mathrm{exp}(\lambda \langle X_1,\mu\rangle\langle X_k,\mu\rangle)}.$$
    And $T_{\infty}(X_1)=\lambda\Sigma\mu\mu^\top X_1$, we have that $T_L\to T_{\infty}$ a.s., and similarly for its Jacobian $D_\mu T_L(X_1)\to D_\mu T_{\infty}(X_1)$ a.s., and its Hessian $D_\mu^2 T_L(X_1)\to D_\mu^2 T_{\infty}(X_1)$ a.s..
\end{lemma}

\begin{proof}
Consider $X\sim \mathcal{N}(0,\Sigma)$ and let us fix $X_1=z$, and define $\eta_k(z)=\eta_k(\mu,z)=\mathrm{exp}(\lambda \langle z,\mu\rangle\langle X_k,\mu\rangle)$, then $$T_L(z)=\frac{\sum_{k=1}^L \eta_k(z)X_k}{\sum_{k=1}^L \eta_k(z)}=\frac{\frac{1}{L}\eta_1(z)X_1+\frac{L-1}{L}\frac{1}{L-1}\sum_{k=2}^L \eta_k(z)X_k}{\frac{1}{L}\eta_1(z)+\frac{L-1}{L}\frac{1}{L-1}\sum_{k=2}^L \eta_k(z)}.$$
By the strong law of large numbers, we have that $$\lim_{L\rightarrow\infty}T_L(z)=\frac{\mathbb{E}[X\mathrm{exp}(\lambda\langle X,\mu\rangle\langle z,\mu\rangle)]}{\mathbb{E}[\mathrm{exp}(\lambda\langle X,\mu\rangle\langle z,\mu\rangle)]}=\lambda\Sigma\mu\mu^\top z, \text{ a.s..}$$  Therefore $$\mathbb{P}\left(\lim_{L\rightarrow\infty} T_L(X_1)=\lambda\Sigma\mu\mu^\top X_1\Big|X_1\right)=1.$$
Taking expectation w.r.t. $X_1$, we get that $\mathbb{P}\left(\lim_{L\rightarrow\infty} T_L(X_1)=\lambda\Sigma\mu\mu^\top X_1\right)=1$ or $\lim_{L\rightarrow\infty} T_L(X_1)=\lambda\Sigma\mu\mu^\top X_1$ a.s..

\paragraph{First order.} Let $$N_L(\mu,z)=\frac{1}{L-1}\sum_{k=2}^L \eta_k(\mu,z)X_k,\quad  S_L(\mu,z)=\frac{1}{L-1}\sum_{k=2}^L \eta_k(\mu,z),$$ then $T_L(z)=\frac{N_L(\mu,z)}{S_L(\mu,z)}+R_{1,L}(z)$ and $$D_{\mu} T_L(z)=\frac{1}{S_L(\mu,z)}D_{\mu}(N_L(\mu,z))-\frac{1}{S_L(\mu,z)^2}N_L(\mu,z)(D_{\mu} S_L(\mu,z))^\top+D_{\mu}R_{1,L}(z).$$
Using the strong law of large numbers, we get (since $\lim_{L\rightarrow\infty}D_{\mu}R_{1,L}(z)=0$ a.s.) \begin{align*}
    \lim_{L\rightarrow\infty} D_{\mu} T_L(z)&=\frac{\mathbb{E}[X(\nabla_{\mu}\mathrm{exp}(\lambda\langle X,\mu\rangle\langle z,\mu\rangle)) ^\top]}{\mathbb{E}[\mathrm{exp}(\lambda\langle X,\mu\rangle\langle z,\mu\rangle)]]}-\frac{\mathbb{E}[X\mathrm{exp}(\lambda\langle X,\mu\rangle\langle z,\mu\rangle)](\mathbb{E}[\nabla_{\mu}\mathrm{exp}(\lambda\langle X,\mu\rangle\langle z,\mu\rangle)])^\top}{\mathbb{E}[\mathrm{exp}(\lambda\langle X,\mu\rangle\langle z,\mu\rangle)]]^2}\\
    &=\lambda(\Sigma\mu z^\top+(\mu^\top z)\Sigma)=D_{\mu} T_{\infty}(z). 
\end{align*}
And we conclude as before that $\lim_{L\rightarrow\infty} D_{\mu} T_L(X_1)=D_{\mu} T_{\infty}(X_1)$ a.s..

\paragraph{Second-order.}
Condition on $X_1=z$ and recall that
\[
N_L(\mu,z)=\frac{1}{L-1}\sum_{k=2}^L \eta_k(\mu,z)X_k,\qquad
S_L(\mu,z)=\frac{1}{L-1}\sum_{k=2}^L \eta_k(\mu,z),
\]
so that
\[
T_L(z)=\frac{N_L(\mu,z)}{S_L(\mu,z)}+R_{1,L}(z).
\]
Differentiating twice w.r.t. $\mu$ and using the quotient rule,
\[
D_\mu^2 T_L(z)
= D_\mu^2\!\left(\frac{N_L(z)}{S_L(z)}\right)+D_\mu^2 R_{1,L}(z),
\]
where 
\[
\begin{aligned}
D_\mu^2\!\left(\frac{N_L}{S_L}\right)[h]
&= \frac{1}{S_L} D_\mu^2 N_L[h]
- \frac{1}{S_L^2} \Big(
(D_\mu N_L)\langle D_\mu S_L, h\rangle
+ (D_\mu S_L)\langle D_\mu N_L, h\rangle
\Big)\\
&\quad - \frac{1}{S_L^2} N_L D_\mu^2 S_L[h]
+ \frac{2}{S_L^3} N_L \langle D_\mu S_L, h\rangle D_\mu S_L .
\end{aligned}
\]
 By the strong law of large numbers applied componentwise to 
\[
\big\{X_k\eta_k(\mu,z),\;D_\mu[X_k\eta_k(\mu,z)],\;D_\mu^2[X_k\eta_k(\mu,z)]\big\}_{k\ge2},
\]
and to
\[
\big\{\eta_k(\mu,z),\;D_\mu\eta_k(\mu,z),\;D_\mu^2\eta_k(\mu,z)\big\}_{k\ge2},
\]
we obtain almost surely the convergence of $N_L$, $D_\mu N_L$, $D_\mu^2 N_L$ and of $S_L$, $D_\mu S_L$, $D_\mu^2 S_L$ toward their respective expectations. We verify
\[
\lim_{L\to\infty}D_\mu^2 T_L(z)=D_\mu^2 T_{\infty}(z),
\]
where, using the closed form
\[
T_{\infty}(z)=\lambda \Sigma\mu\mu^\top z,
\]
the Hessian satisfies,
\[
D_\mu^2 T_{\infty}(z)[h]
=\lambda\big(\Sigma h z^\top+\langle h,z\rangle \Sigma\big).
\]
Finally, as before we conclude
\[
\lim_{L\to\infty}D_\mu^2 T_L(X_1)=D_\mu^2 T_{\infty}(X_1)\quad\text{a.s.}
\]

\end{proof}

\begin{lemma}\label{lemma_concentration}
 Consider $X,X_1,\ldots,X_L$ i.i.d $\mathcal{N}(0,\Sigma)$, and $$T_L^{\mu}(X_1)=\frac{\sum_{k=1}^L X_k \mathrm{exp}(\lambda \langle X_1,\mu\rangle\langle X_k,\mu\rangle)}{\sum_{k=1}^L \mathrm{exp}(\lambda \langle X_1,\mu\rangle\langle X_k,\mu\rangle)},\quad T_{\infty}^{\mu}(X_1)=\lambda\Sigma\mu\mu^\top X_1,$$
      and we define $g:\mathbb{N}\times\mathbb{R}_+ \times\mathbb{R}^d\times\mathbb{R}^{d\times d}\times\mathbb{N}\rightarrow\mathbb{R}$ by $$g(L,\lambda,\mu,\Sigma,P)=L^{-\frac{1}{P^2\lambda^2(\mu^\top\Sigma\mu)^2+1}}(1+ \ln L)^{\frac{P^2\lambda^2(\mu^\top\Sigma\mu)^2}{P^2\lambda^2(\mu^\top\Sigma\mu)^2+1}},$$ then we have that $$\mathbb{E}[\Vert T_L^{\mu}(X_1)- T_{\infty}^{\mu}(X_1)\Vert^2]= \mathcal{O}\!\left( g(L,\lambda,\mu,\Sigma,12)\right).$$ Also,  $$\mathbb{E}[\Vert D_\mu T_L^{\mu}(X_1)- D_\mu T_{\infty}^{\mu}(X_1)\Vert_F^2]= \mathcal{O}\!\left( g(L,\lambda,\mu,\Sigma,16)\right),$$
And  $$\mathbb{E}[\Vert D_\mu^2 T_L^{\mu}(X_1)- D_\mu^2 T_{\infty}^{\mu}(X_1)\Vert_F^2]= \mathcal{O}\!\left( g(L,\lambda,\mu,\Sigma,20)\right).$$
     
We note $\lim_{L\rightarrow\infty} \sup_{\mu\in B(0,\rho)}g(L,\lambda,\mu,\Sigma,P)=0$, since the dependency of $\mu$ inside the $\mathcal{O}-$term is polynomial, $T_L^{\mu}(X_1)$ and $DT_L^{\mu}(X_1)$ converge in $L^2$ uniformly over compact sets of $\mu$ to $T_{\infty}^{\mu}(X_1)$ and $DT_{\infty}^{\mu}(X_1)$, respectively, as $L \to \infty$.

\end{lemma}
\begin{proof}
Throughout the proof, $C_{\mathrm{params}}>0$ denotes a constant depending only on the indicated parameters. We condition on $X_1=z$.

\paragraph{Notation.}
Define
\[
\eta_k(z)=\exp(\lambda\langle z,\mu\rangle\langle X_k,\mu\rangle),\quad
N_L(z)=\frac{1}{L-1}\sum_{k=2}^L \eta_k(z)X_k,\quad
S_L(z)=\frac{1}{L-1}\sum_{k=2}^L \eta_k(z),
\]
and
\[
N=\mathbb{E}[N_L(z)],\quad S=\mathbb{E}[S_L(z)],\quad
S_\theta=S+\theta(S_L-S),\quad N_\theta=N+\theta(N_L-N).
\]
Let
\[
\xi = \lambda^2\langle z,\mu\rangle^2 \mu^\top\Sigma\mu.
\]

\paragraph{General structure.}
All terms appearing below are finite sums of quantities of the form
\[
T_k = \alpha_k\,\mathbb{E}\!\left[\prod_i |Z_{i,k}|^{p_{i,k}}\right],
\]
where $Z_{i,k}$ belongs to
\[
N_L-N,\; S_L-S,\; D_\mu N_L-D_\mu N,\; D_\mu S_L-D_\mu S,\; S_\theta^{-1},\; N_\theta,\; D_\mu N_\theta,\; D_\mu S_\theta.
\]
Using Lemmas~\ref{boundsnd} and \ref{generalized_holder},
\[
T_k
\le
C_{\lambda,\Sigma}
L^{-\frac{1}{2}\sum_i p_{i,k}}
(1+\|z\|^{m_k})
\exp\!\left(\frac{1}{2}\Big(\sum_i p_{i,k}\Big)^2 \xi\right).
\]
We define
\[
P := \max_k \sum_i p_{i,k}.
\]

%------------------------------------------------------------
\paragraph{0-th order bound.}

We write
\[
T_L(z)=\frac{N_L(z)}{S_L(z)} + R_{1,L}(z),
\]
where
\[
R_{1,L}(z)=\frac{\eta_1(z)}{L}
\frac{zS_L(z)-N_L(z)}{S_L(z)\left(S_L(z)+\frac{\eta_1(z)}{L}\right)}.
\]

Using Lemma~\ref{boundsnd} with $p=2$,
\[
\mathbb{E}[\|R_{1,L}(z)\|^2]
\le
\frac{C_{\lambda,\Sigma}}{L^2}
(1+\|z\|^2)
\exp\!\left(\frac{2^2}{2}\xi\right).
\]

\medskip

\noindent
Consider the map
\[
F_1:\mathbb{R}^d \times \mathbb{R} \to \mathbb{R}^d,
\qquad
F_1(x,y)=\frac{x}{y}.
\]
Its differentials are
\[
D F_1(x,y)[h,k]
=
\frac{h}{y} - \frac{x k}{y^2},
\]
\[
D^2 F_1(x,y)[(h,k),(h,k)]
=
\frac{2x}{y^3}k^2 - \frac{2}{y^2}hk.
\]

Applying Taylor's formula at $(N,S)$, there exists $\theta\in(0,1)$ such that
\[
\frac{N_L}{S_L}-\frac{N}{S}
=
D F_1(N,S)[N_L-N, S_L-S]
+
R_{2,L},
\]
where
\[
R_{2,L}
=
\frac{1}{2}
D^2 F_1(N_\theta,S_\theta)
[(N_L-N,S_L-S),(N_L-N,S_L-S)].
\]

Thus
\[
R_{2,L}
=
\frac{(N+\theta(N_L-N))(S_L-S)^2}{(S+\theta(S_L-S))^3}
-
\frac{(N_L-N)(S_L-S)}{(S+\theta(S_L-S))^2},
\]
and
\[
\|R_{2,L}\|^2
\le
2\frac{\|N_{\theta}\|^2 (S_L-S)^4}{S_\theta^6}
+
2\frac{\|N_L-N\|^2 (S_L-S)^2}{S_\theta^4}.
\]

Applying Lemmas~\ref{boundsnd} and \ref{generalized_holder}, there exists $m_0>0$ such that
\[
\mathbb{E}[\|R_{2,L}(z)\|^2]
\le
\frac{C_{\lambda,\Sigma}}{L^2}
(1+\|z\|^{m_0})
\exp\!\left(\frac{12^2}{2}\xi\right).
\]

\paragraph{Origin of $P=12$.}
The highest order product is
\[
\|N_\theta\|^2 (S_L-S)^4 S_\theta^{-6},
\quad\Rightarrow\quad
P = 2+4+6 = 12.
\]

Moreover,
\[
\mathbb{E}\left[\left\|\frac{N_L-N}{S}-\frac{N(S_L-S)}{S^2}\right\|^2\right]
=
\frac{C_{\lambda,\Sigma}}{L}
(1+\|z\|^{m_0})
\exp\!\left(\frac{2^2}{2}\xi\right).
\]

Thus
\[
\mathbb{E}[\|T_L-T_\infty\|^2 \mid X_1=z]
\le
\frac{C_{\lambda,\Sigma}}{L}
(1+\|z\|^{m_0})
\exp\!\left(\frac{12^2}{2}\xi\right).
\]

Applying Lemma~\ref{bestbound},
\[
\mathbb{E}[\|T_L-T_\infty\|^2]
=
\mathcal O\!\big(g(L,\lambda,\mu,\Sigma,12)\big).
\]

%------------------------------------------------------------
\paragraph{1-st order bound.}

We write
\[
D_\mu T_L - D_\mu T_\infty
=
\left(\frac{D_\mu N_L}{S_L}-\frac{D_\mu N}{S}\right)
-
\left(\frac{N_L(D_\mu S_L)^\top}{S_L^2}-\frac{N(D_\mu S)^\top}{S^2}\right)
+
\widetilde R_{1,L}(z),
\]
where $\widetilde R_{1,L}(z)$ gathers all terms arising from differentiating the contribution of $\eta_1(z)$ in the numerator and denominator of $T_L$. .

\medskip

\noindent
Consider the map
\[
F_2:\mathbb{R}^{d\times d}\times\mathbb{R}\to\mathbb{R}^{d\times d},
\qquad
F_2(A,y)=\frac{A}{y}.
\]
Applying Taylor's formula at $(D_\mu N,S)$ yields an expansion of
\[
\frac{D_\mu N_L}{S_L}-\frac{D_\mu N}{S}.
\]

Similarly, define
\[
F_3:\mathbb{R}^d \times \mathbb{R}^d \times \mathbb{R} \to \mathbb{R}^{d\times d},
\qquad
F_3(x,y,w)=\frac{xy^\top}{w^2},
\]
and apply Taylor's formula at $(N,D_\mu S,S)$.

All resulting terms are finite sums of products handled above. Using Lemmas~\ref{boundsnd} and \ref{generalized_holder}, there exists $m_1>0$ such that
\[
\mathbb{E}[\|D_\mu T_L-D_\mu T_\infty\|^2 \mid X_1=z]
\le
\frac{C_{\lambda,\Sigma}}{L}
(1+\|z\|^{m_1})
\exp\!\left(\frac{16^2}{2}\xi\right).
\]

\paragraph{Origin of $P=16$.}
The highest order term is
\[
\|N_\theta\|^2 \|D_\mu S_\theta\|^2 (S_L-S)^4 S_\theta^{-8},
\quad\Rightarrow\quad
P = 2+2+4+8 = 16.
\]

Applying Lemma~\ref{bestbound},
\[
\mathbb{E}[\|D_\mu T_L-D_\mu T_\infty\|^2]
=
\mathcal O\!\big(g(L,\lambda,\mu,\Sigma,16)\big).
\]

%------------------------------------------------------------
\paragraph{2-nd order bound.}

We use the identity
\[
D_\mu^2\!\left(\frac{N_L}{S_L}\right)
=
\frac{D_\mu^2 N_L}{S_L}
-
\frac{2(D_\mu N_L)(D_\mu S_L)^\top}{S_L^2}
-
\frac{N_L D_\mu^2 S_L}{S_L^2}
+
\frac{2N_L(D_\mu S_L)(D_\mu S_L)^\top}{S_L^3},
\]
and the analogous expression for $N/S$.

Each difference is expanded using Taylor formulas for maps of the form
\[
(A,y)\mapsto \frac{A}{y},\qquad
(x,y,w)\mapsto \frac{xy^\top}{w^2},\qquad
(x,y,z,w)\mapsto \frac{xyz^\top}{w^3}.
\]

Using Lemmas~\ref{boundsnd} and \ref{generalized_holder}, there exists $m_2>0$ such that
\[
\mathbb{E}[\|D_\mu^2 T_L-D_\mu^2 T_\infty\|^2 \mid X_1=z]
\le
\frac{C_{\lambda,\Sigma}}{L}
(1+\|z\|^{m_2})
\exp\!\left(\frac{20^2}{2}\xi\right).
\]

\paragraph{Origin of $P=20$.}
The highest order term is
\[
\|N_\theta\|^2
\|D_\mu S_\theta\|^2
\|D_\mu S_\theta\|^2
(S_L-S)^4
S_\theta^{-10},
\quad\Rightarrow\quad
P = 2+2+2+4+10 = 20.
\]

Applying Lemma~\ref{bestbound},
\[
\mathbb{E}[\|D_\mu^2 T_L-D_\mu^2 T_\infty\|_F^2]
=
\mathcal O\!\big(g(L,\lambda,\mu,\Sigma,20)\big).
\]

\end{proof}
   
\begin{lemma}\label{bound}
    Let $f:\mathbb{R}^d\rightarrow\mathbb{R}^{m_1\times m_2}$ and $X_1,\ldots, X_L$ be i.i.d. random variables  $\mathcal{N}(0,\Sigma)$ , let $A_L=\frac{1}{L}\sum_{k=1}^L f(X_k)$, then for $p\geq 2$, we have $$\mathbb{E}[\Vert A_L-\mathbb{E}[A_L]\Vert_F^p]\leq C_p L^{-\frac{p}{2}}\mathbb{E}[\Vert f(X_1)\Vert_{F}^p],$$
    where $C_p$ is a constant that only depends on $p$, and $\Vert\cdot\Vert_F$ is the Frobenius norm on $\mathbb{R}^{m_1\times m_2}$.
\end{lemma}
\begin{proof}
    Let $Y_k=f(X_k)-\mathbb{E}[f(X_1)]$. By construction, $\mathbb{E}[Y_k]=0$ and $Y_k$ are i.i.d., we have $$\mathbb{E}[\Vert A_L-\mathbb{E}[A_L]\Vert^p]=\frac{1}{L^p}\mathbb{E}\left[\Big\Vert\sum_{k=1}^L Y_k\Big\Vert^p\right],$$

    Rosenthal's inequality states that there exists a constant $R_p$  depending only on $p$ such that $$\mathbb{E}\left[\Big\Vert\sum_{k=1}^L Y_k\Big\Vert^p\right]\leq R_pL^{\frac{p}{2}}\mathbb{E}[\Vert Y_1\Vert^p],$$
    Besides by Jensen's inequality, $$\mathbb{E}[\Vert Y_1\Vert^p]\leq 2^p \mathbb{E}[\Vert f(X_1)\Vert^p].$$
    We conclude by taking $C_p=2^p R_p$.

\end{proof}
\begin{lemma}\label{boundsnd}
    Let $\lambda>0, z,\mu\in \mathbb{R}^d$, $\theta\in [0,1]$, $X,X_1,\ldots, X_L$ i.i.d. $\mathcal{N}(0,\Sigma)$, \begin{align*}
                \eta_k(z)&=\mathrm{exp}(\lambda\langle z,\mu\rangle\langle X_k,\mu\rangle),\\
                N_L(z)&=\frac{1}{L-1}\sum_{k=2}^L\eta_k(z)X_k,\\
                N&= \mathbb{E}[N_L],\\
                S_L(z)&=\frac{1}{L-1}\sum_{k=2}^L\eta_k(z),\\
                S&= \mathbb{E}[S_L],\\
                S_{\theta}&=S+\theta (S_L-S),\\
                N_{\theta}&=N+\theta (N_L-N),\\
            \end{align*}
            Then for $p>0$, there exists $C_{p,\Sigma},C_p>0$ (depending only on the constants in the subscripts) such that letting $\xi= \lambda^2\langle z,\mu\rangle^2\mu^\top\Sigma\mu$ we have:  \begin{align*}
                \mathbb{E}[S_{\theta}^{-p}]&\leq 2\mathrm{exp}\left(\frac{p^2}{2}\xi\right),\\
                \mathbb{E}[\Vert N_L-N\Vert^p]&\leq C_{p,\Sigma}L^{-\frac{p}{2}}(1+\lambda^p\Vert\mu\Vert^{2p}\Vert z\Vert^p)\mathrm{exp}\left(\frac{p^2}{2}\xi\right),\\
                \mathbb{E}[|S_L-S|^p]&\leq C_pL^{-\frac{p}{2}}\mathrm{exp}\left(\frac{p^2}{2}\xi\right),\\
                \mathbb{E}[\Vert D_{\mu} N_L-D_{\mu}N\Vert^p]&\leq C_{p,\Sigma}L^{-\frac{p}{2}}  \lambda^p\Vert\mu\Vert^p\Vert z\Vert^p(1+\lambda^{2p}\Vert\mu\Vert^{4p}\Vert z\Vert^{2p})\mathrm{exp}\left(\frac{p^2}{2}\xi\right),\\
                \mathbb{E}[\Vert D_{\mu} S_L-D_{\mu} S\Vert^p]&\leq C_{p,\Sigma}L^{-\frac{p}{2}} \lambda^p \Vert\mu\Vert^p\Vert z\Vert^p(1+\lambda^p\Vert\mu\Vert^{2p}\Vert z\Vert^p)\mathrm{exp}\left(\frac{p^2}{2}\xi\right).
            \end{align*}
\end{lemma}
\begin{proof}
            
            By Jensen's inequality we have that, for $p>0$: $S_{\theta}^{-p}\leq S^{-p}+S_L^{-p}$ and $$S_L^{-p}\leq \frac{1}{L-1}\sum_{k=2}^L \mathrm{exp}(-p\lambda\langle z,\mu\rangle\langle X_k,\mu\rangle),$$ thus $$\mathbb{E}[S_{\theta}^{-p}]\leq S^{-p}+\mathbb{E}[\mathrm{exp}(-p\lambda\langle z,\mu\rangle\langle X,\mu\rangle)]\leq 2\mathrm{exp}\left(\frac{1}{2}p^2 \lambda^2\langle z,\mu\rangle^2\mu^\top\Sigma\mu\right).$$

Using Lemma \ref{bound} with $f_1(x)=x\mathrm{exp}(\lambda\langle z,\mu\rangle\langle x,\mu\rangle)$ and $f_2(x)=\mathrm{exp}(\lambda\langle z,\mu\rangle\langle x,\mu\rangle)$ and Lemma \ref{pqbound}, we get \begin{align*}
    \mathbb{E}[\Vert N_L-N\Vert^p]&\leq C_pL^{-\frac{p}{2}}\mathbb{E}[\Vert f_1(X)\Vert^p]\\
    %&\leq C_{p}'L^{-\frac{p}{2}}(\Tr(\Sigma)^{\frac{p}{2}}+\Vert\Sigma\mu\Vert^p\lambda^p|\langle z,\mu\rangle|^p)\mathrm{exp}\left(\frac{1}{2}p^2\lambda^2\langle z,\mu\rangle^2\mu^\top\Sigma\mu\right)\\
    &\leq C_{p,\Sigma}L^{-\frac{p}{2}}(1+\Vert\mu\Vert^p\lambda^p|\langle z,\mu\rangle|^p)\mathrm{exp}\left(\frac{1}{2}p^2\lambda^2\langle z,\mu\rangle^2\mu^\top\Sigma\mu\right)\\
    &\leq C_{p,\Sigma}L^{-\frac{p}{2}}(1+\lambda^p\Vert\mu\Vert^{2p}\Vert z\Vert^p)\mathrm{exp}\left(\frac{1}{2}p^2\lambda^2\langle z,\mu\rangle^2\mu^\top\Sigma\mu\right),
\end{align*} and $$\mathbb{E}[|S_L-S|^p]\leq C_pL^{-\frac{p}{2}}\mathbb{E}[\Vert f_2(X)\Vert^p]=C_pL^{-\frac{p}{2}}\mathrm{exp}\left(\frac{1}{2}p^2\lambda^2\langle z,\mu\rangle^2\mu^\top\Sigma\mu\right).$$

Besides, with $$f_3(x)=\lambda\mathrm{exp}(\lambda\langle z,\mu\rangle\langle x,\mu\rangle)x(\langle z,\mu\rangle x+\langle x,\mu\rangle z)^\top,$$ and $$f_4(x)=\lambda\mathrm{exp}(\lambda\langle z,\mu\rangle\langle x,\mu\rangle)(\langle z,\mu\rangle x+\langle x,\mu\rangle z),$$
we derive the bounds \begin{align*}
    \mathbb{E}[\Vert D_{\mu} N_L-D_{\mu}N\Vert^p]&\leq C_pL^{-\frac{p}{2}}\mathbb{E}[\Vert f_3(X)\Vert_F^p]\\
    &\leq C_pL^{-\frac{p}{2}} C_{p,\Sigma} \lambda^p\Vert\mu\Vert^p\Vert z\Vert^p(1+\lambda^{2p}\langle z,\mu\rangle^{2p}\Vert\mu\Vert^{2p})\mathrm{exp}\left(\frac{1}{2}p^2\lambda^2\langle z,\mu\rangle^2\mu^\top\Sigma\mu\right)\\
    &\leq C_{p,\Sigma}L^{-\frac{p}{2}}  \lambda^p\Vert\mu\Vert^p\Vert z\Vert^p(1+\lambda^{2p}\Vert\mu\Vert^{4p}\Vert z\Vert^{2p})\mathrm{exp}\left(\frac{1}{2}p^2\lambda^2\langle z,\mu\rangle^2\mu^\top\Sigma\mu\right),\\
\end{align*}

and

\begin{align*}
    \mathbb{E}[\Vert D_{\mu} S_L-D_{\mu} S\Vert^p]&  \leq C_pL^{-\frac{p}{2}}\mathbb{E}[\Vert f_4(X)\Vert^p]\\
    &\leq C_pL^{-\frac{p}{2}} C_{p,\Sigma}\lambda^p \Vert\mu\Vert^p\Vert z\Vert^p(1+\lambda^p\Vert\mu\Vert^{2p}\Vert z\Vert^p)\mathrm{exp}\left(\frac{1}{2}p^2\lambda^2\langle z,\mu\rangle^2\mu^\top\Sigma\mu\right)\\
   &\leq C_{p,\Sigma}L^{-\frac{p}{2}} \lambda^p \Vert\mu\Vert^p\Vert z\Vert^p(1+\lambda^p\Vert\mu\Vert^{2p}\Vert z\Vert^p)\mathrm{exp}\left(\frac{1}{2}p^2\lambda^2\langle z,\mu\rangle^2\mu^\top\Sigma\mu\right).\\
\end{align*}
Furthermore, using Lemma \ref{bound} with $$f_5(x)=x\mathrm{exp}(\lambda\langle z,\mu\rangle\langle x,\mu\rangle)[\lambda^2(\langle z,h\rangle\langle x,\mu\rangle+\langle z,\mu\rangle\langle x,h\rangle)^2+2\lambda\langle z,h\rangle\langle x,h\rangle]^\top,$$
and
$$f_6(x)=\mathrm{exp}(\lambda\langle z,\mu\rangle\langle x,\mu\rangle)[\lambda^2(\langle z,h\rangle\langle x,\mu\rangle+\langle z,\mu\rangle\langle x,h\rangle)^2+2\lambda\langle z,h\rangle\langle x,h\rangle],$$ and Lemma \ref{pqbound}, we get
\begin{align*}
    \mathbb{E}[\Vert D_{\mu}^2 N_L[h]- D_{\mu}^2 N[h]\Vert_F^p]&\leq C_pL^{-\frac{p}{2}}\mathbb{E}[\Vert f_5(X_1)\Vert_F^p]\\
    &\leq  C_{p,\Sigma}\lambda^{2p}\Vert z\Vert^p\Vert h\Vert^p(1+\lambda^{2p}\langle z,\mu\rangle^{2p}\Vert \mu\Vert^{2p})\mathrm{exp}\left(\frac{p^2}{2}\lambda^2\langle z,\mu\rangle^2\mu^\top\Sigma\mu\right).
\end{align*}
and
\begin{align*}
    \mathbb{E}[\Vert D_{\mu}^2 S_L[h]- D_{\mu}^2 S[h]\Vert^p]&\leq  C_pL^{-\frac{p}{2}}\mathbb{E}[\Vert f_6(X_1)\Vert^p]\\
    &  \leq C_{p,\Sigma}\lambda^p\Vert z\Vert^p\Vert h\Vert^p(1+\lambda^p\langle z,\mu\rangle^p\Vert \mu\Vert^p)\mathrm{exp}\left(\frac{p^2}{2}\lambda^2\langle z,\mu\rangle^2\mu^\top\Sigma\mu\right).
\end{align*}
\end{proof}

\begin{lemma}\label{pqbound}
    Consider $\mu\in\mathbb{R}^d, \Sigma\in \mathbb{R}^{d\times d}$ symmetric and positive definite, $p,q>0$ and $X\sim \mathcal{N}(0,\Sigma)$, then $$\mathbb{E}[\Vert X\Vert^p \mathrm{exp}(q\langle X,\mu\rangle)]\leq C_{p,\Sigma}(1+q^p\Vert \mu\Vert^p)\mathrm{exp}\left(\frac{q^2}{2}\mu^\top\Sigma\mu\right).$$
\end{lemma}

\begin{proof}
    We have that $$\mathbb{E}[\Vert X\Vert^p \mathrm{exp}(q\langle X,\mu\rangle)]=\mathrm{exp}\left(\frac{q^2}{2}\mu^\top\Sigma\mu\right)\mathbb{E}[\Vert X+q\Sigma\mu\Vert^p].$$
    And we bound $$\Vert X+q\Sigma\mu\Vert^p\leq 2^{p-1}(\Vert X\Vert^p+\Vert\Sigma\Vert_{op}^pq^p\Vert\mu\Vert^p).$$
    So there exists $C_{p,\Sigma}\eqdef 2^{p-1}\max\{\mathbb{E}[\Vert X\Vert^p],\Vert\Sigma\Vert_{op}^p\}>0$ such that $$\mathbb{E}[\Vert X+q\Sigma\mu\Vert^p]\leq C_{p,\Sigma} (1+q^p\Vert\mu\Vert^p).$$

    And we conclude.
    
\end{proof}
\begin{comment}
\begin{lemma}\label{convexp}
Let $\alpha>0$ and $X \sim \mathcal{N}(0, \Sigma)$, and let $\mu \in \mathbb{R}^d$. Then
\[
\mathbb{E}\big[ \exp(\alpha \langle X, \mu \rangle^2)\big] < +\infty
\quad \text{if and only if} \quad 
\alpha < \frac{1}{2 \mu^\top \Sigma \mu}.
\]
\end{lemma}

\begin{proof}
Set $Y := \langle X, \mu \rangle \sim \mathcal{N}(0, \sigma^2)$ with $\sigma^2 := \mu^\top \Sigma \mu$. Then we can write
\[
\mathbb{E}\big[   e^{\alpha \langle X, \mu \rangle^2}\big] = \mathbb{E}[ e^{\alpha Y^2}] = \frac{1}{\sqrt{2 \pi \sigma^2}} \int_{-\infty}^{\infty}  \exp\Big(\alpha y^2 - \frac{y^2}{2 \sigma^2}\Big) dy.
\]

Combine terms in the exponential:
\[
\alpha y^2 - \frac{y^2}{2\sigma^2} = - \frac{1 - 2 \alpha \sigma^2}{2 \sigma^2}   y^2.
\]

Hence the integral becomes
\[
\mathbb{E}[ e^{\alpha Y^2}] = \frac{1}{\sqrt{2 \pi \sigma^2}} \int_{-\infty}^{\infty}  \exp\Big(- \frac{1 - 2 \alpha \sigma^2}{2\sigma^2} y^2 \Big) dy.
\]

- If $1 - 2 \alpha \sigma^2 > 0$ (i.e., $\alpha < \frac{1}{2 \sigma^2}$), the function is integrable. Therefore the expectation is finite.  

- If $1 - 2 \alpha \sigma^2 \le 0$ (i.e., $\alpha \ge \frac{1}{2 \sigma^2}$), the exponent is positive and the integral diverges, giving $+\infty$.

Substituting back $\sigma^2 = \mu^\top \Sigma \mu$, we conclude
\[
\mathbb{E}\big[ \exp(\alpha \langle X, \mu \rangle^2)\big] < +\infty
\quad \iff \quad 
\alpha < \frac{1}{2 \mu^\top \Sigma \mu}.
\]
\end{proof}
\end{comment}
\begin{lemma}\label{generalized_holder}
    Let $Z_1,\ldots,Z_n$ be non-negative random variables such that for each $i\in \{1,\ldots,n\}$ there exists $C_{i,p}$ that grows at most exponentially in $p$ such that,
$$\mathbb{E}[ |Z_i|^p]\leq C_{i,p}\exp \left(\frac{p^2}{2}\xi\right),$$
for some $\xi>0$. Let $p_1,\ldots,p_n>0$. Then there exists $C$ that grows at most exponentially in $(p_1,\ldots,p_n)$ such that $$\mathbb{E}\left[\prod_{i=1}^n|Z_i|^{p_i}\right]\leq C\mathrm{exp}\left(\frac{1}{2}\left(\sum_{i=1}^n p_i\right)^2\xi\right)$$
\end{lemma}
\begin{proof}
    We apply the generalized Hölder's inequality to get that for every $q_1,\ldots,q_n >0$ such that $\sum_{i=1}^n \frac{1}{q_i}=1$, \begin{align*}
        \mathbb{E}\left[\prod_{i=1}^n|Z_i|^{p_i}\right]&\leq \prod_{i=1}^n\mathbb{E}[|Z_i|^{p_iq_i}]^{\frac{1}{q_i}}\\
        &\leq \prod_{i=1}^n \left[C_{i,p_iq_i}\mathrm{exp}\left(\frac{p_i^2q_i^2}{2}\xi\right)\right]^{\frac{1}{q_i}}\\
        &= \left(\prod_{i=1}^n (C_{i,p_iq_i})^{\frac{1}{q_i}}\right)\mathrm{exp}\left(\frac{1}{2}\sum_{i=1}^n\left(p_i^2q_i\right)\xi\right).
        \end{align*}
        Choosing $q_i=\frac{\sum_{j=1}^n p_j}{p_i}$, which minimizes $\sum_{i=1}^n p_i^2q_i$ given $\sum_{i=1}^n \frac{1}{q_i}=1$, we conclude since $C_{i,\sum_{j=1}^n p_j}^{\frac{p_i}{\sum_{j=1}^n p_j}}$ grows exponentially in $p_i$ and $C=\prod_{i=1}^n C_{i,\sum_{j=1}^n p_j}^{\frac{p_i}{\sum_{j=1}^n p_j}}$ grows exponentially in $(p_1,\ldots,p_n)$. 
\end{proof}

\subsection{ICL technical propositions}
\begin{lemma}\label{lem:wishart}
    Let $\Sigma\sim W_d(V,n)$, and let  $\mu\in\mathbb{R}^d$ be deterministic. Then: \begin{enumerate}
        \item $\mathbb{E}[\operatorname{tr}\Sigma]=n\operatorname{tr}V.$
        \item $\mathbb{E}[\mu^{\top}\Sigma^{2}\mu]
= n(n+1)\,\mu^{\top}V^{2}\mu + n\,\operatorname{tr}(V)\,\mu^{\top}V\mu.$
\item $\mathbb{E}[(\mu^{\top}\Sigma^{2}\mu)(\mu^{\top}\Sigma^{2}\mu)]=n(n+2)[(n+3)(\mu^{\top}V\mu)(\mu^{\top}V^{2}\mu)+(\mu^{\top}V\mu)^2\operatorname{tr}(V)]$.
    \end{enumerate} 
\end{lemma}
\begin{proof}
We use the standard Gaussian representation of the Wishart distribution:
\[
\Sigma=\sum_{r=1}^n x_r x_r^\top,
\qquad
x_r \overset{iid}{\sim}\mathcal N(0,V).
\]

\smallskip
\begin{enumerate}
    \item 
Since $\operatorname{tr}(x_r x_r^\top)=\|x_r\|^2$ and
$\mathbb E[\|x_r\|^2]=\operatorname{tr}V$, linearity of expectation gives
\[
\mathbb E[\operatorname{tr}\Sigma]
=\sum_{r=1}^n \mathbb E[\|x_r\|^2]
=n\,\operatorname{tr}V.
\]

\smallskip
\item
Write
\[
\mu^\top\Sigma^2\mu
=\sum_{i,j=1}^n (\mu^\top \alpha_i)(\alpha_i^\top x_j)(\mu^\top x_j).
\]
Splitting the sum into the cases $i=j$ and $i\neq j$:

\emph{(i) Diagonal terms.} For $i=j$,
\[
\mathbb E\big[(\mu^\top x)^2(x^\top x)\big]
=\mu^\top V\mu\,\operatorname{tr}V + 2\,\mu^\top V^2\mu,
\]
by Isserlis' formula.

\emph{(ii) Off-diagonal terms.} For $i\neq j$, independence yields
\[
\mathbb E[(\mu^\top \alpha_i)(\alpha_i^\top x_j)(\mu^\top x_j)]
=\mu^\top V^2\mu.
\]

Counting terms,
\[
\mathbb E[\mu^\top\Sigma^2\mu]
=n\big(\mu^\top V\mu\,\operatorname{tr}V+2\mu^\top V^2\mu\big)
+n(n-1)\mu^\top V^2\mu,
\]
which simplifies to
\[
\mathbb E[\mu^\top\Sigma^2\mu]
= n(n+1)\mu^\top V^2\mu
+ n\,\operatorname{tr}V\,\mu^\top V\mu.
\]

\smallskip
\item 
Let $s_r=\mu^\top x_r$. Then
\[
\mu^\top\Sigma\mu=\sum_k s_k^2,
\qquad
\mu^\top\Sigma^2\mu=\sum_{i,j}s_i(\alpha_i^\top x_j)s_j.
\]
Hence
\[
\mathbb E[(\mu^\top\Sigma\mu)(\mu^\top\Sigma^2\mu)]
=\sum_{i,j,k}\mathbb E\big[s_k^2 s_i(\alpha_i^\top x_j)s_j\big].
\]

The expectation depends on coincidences among the indices $(i,j,k)$.
Using Isserlis' theorem and independence, one obtains:
\begin{itemize}
\item $i,j,k$ all distinct: contribution $(\mu^\top V\mu)(\mu^\top V^2\mu)$.
\item $i=j\neq k$: contribution $(\mu^\top V\mu)
\big(\mu^\top V\mu\,\operatorname{tr}V+2\mu^\top V^2\mu\big)$.
\item $i=k\neq j$ or $a=j\neq i$: contribution $3(\mu^\top V\mu)(\mu^\top V^2\mu)$.
\item $i=j=k$: contribution
$4(\mu^\top V\mu)(\mu^\top V^2\mu)
+(\mu^\top V\mu)^2\operatorname{tr}V$.
\end{itemize}

Summing all contributions with their combinatorial counts yields
\[
\mathbb E[(\mu^\top\Sigma\mu)(\mu^\top\Sigma^2\mu)]
=
n(n+2)\!\left[(n+3)(\mu^\top V\mu)(\mu^\top V^2\mu)
+(\mu^\top V\mu)^2\operatorname{tr}V\right].
\]

This concludes the proof.
\end{enumerate}
\end{proof}

\begin{lemma}\label{lemma:icl_expansion}
Let $\Sigma \sim W_d(V,n)$, then $\mathcal{R}_{\infty}^{\mathrm{ICL}}(\mu)=\mathbb{E}_{\Sigma\sim W_d(V,n)}[\mathcal{R}_{\mathrm{soft},\infty}^{(\Sigma)}(\mu)]$, and \begin{align*}
        \mathcal{R}_{\infty}^{\mathrm{ICL}}(\mu)&=n\operatorname{tr}(V)-2\lambda n[ (n+1)\,\mu^{\top}V^{2}\mu + \,\operatorname{tr}(V)\,\mu^{\top}V\mu]\\
        &\quad +\lambda^2 n(n+2)[(n+3)(\mu^{\top}V\mu)(\mu^{\top}V^{2}\mu)+(\mu^{\top}V\mu)^2\operatorname{tr}(V)].
    \end{align*}
    In particular, if $V=\xi^2I_d+\theta vv^t$ for $\Vert v\Vert=1, \xi>0,\theta>0$, and we let $\alpha=\langle \mu,v\rangle$ and $r=\Vert\mu\Vert$, we have that there exists $\tilde{\mathcal{R}}^{\mathrm{ICL}}:\mathbb{R}_+\times\mathbb{R}_+\rightarrow\mathbb{R}$ such that $$\mathcal{R}_{\infty}^{\mathrm{ICL}}(\mu)=\tilde{\mathcal{R}}^{\mathrm{ICL}}(r^2,\alpha^2),$$ where 
    \begin{align*}
\tilde{\mathcal{R}}^{\mathrm{ICL}}(r^2,\alpha^2)&= n\big(d\xi^2 +\theta\big) 
-2\lambda n\Big[
(n+1)\Big(\xi^4 r^2+(2\xi^2\theta+\theta^2)\alpha^2\Big)
+\big(d\xi^2+\theta\big)\Big(\xi^2 r^2+\theta\alpha^2\Big)
\Big] \\
&\quad
+\lambda^2 n(n+2)\Bigg\{
(n+3)\Big(\xi^2 r^2+\theta\alpha^2\Big)
\Big(\xi^4 r^2+(2\xi^2\theta+\theta^2)\alpha^2\Big) \\
&\qquad\qquad
+\big(d\xi^2+\theta\big)\Big(\xi^2 r^2+\theta\alpha^2\Big)^2
\Bigg\}.
    \end{align*}
Furthermore, for this particular $V$, the gradient satisfies
\[
\nabla \mathcal{R}_{\infty}^{\mathrm{ICL}}(\mu)
=
2\big(A(r,\alpha)\mu + B(r,\alpha)\alpha v\big),
\]
where
\[
A(r,\alpha)=a_1 r^2 + a_2 \alpha^2 - a_3,
\qquad
B(r,\alpha)=b_1 r^2 + b_2 \alpha^2 - b_3,
\]
for constants $a_i,b_i>0$ defined as \begin{align}\label{ab}
\begin{cases}
    a_1&=2\lambda^2n(n+2)\xi^4[(n+d+3)\xi^2+\theta],\\
    a_2 &=\lambda^2n(n+2)\xi^2\theta[(3n+2d+9)\xi^2+(n+5)\theta],\\
    a_3&=2\lambda n \xi^2[(n+d+1)\xi^2+\theta],\\
    b_1&=a_2,\\
    b_2&=2\lambda^2 n(n+2)\theta^2[(2n+d+6)\xi^2+(n+4)\theta],\\
    b_3&=2\lambda n\theta[(2n+d+2)\xi^2+(n+2)\theta].
\end{cases}
\end{align}
\end{lemma}
\begin{proof}
The expression of $\mathcal{R}_{\infty}^{\mathrm{ICL}}(\mu)$ follows directly from \eqref{rsoftinfty} together with Lemma~\ref{lem:wishart}.

In the case where $V=\xi^2 I_d+\theta vv^\top$, expanding the quadratic forms $\mu^\top V \mu$ and $\mu^\top V^2 \mu$ in terms of $r^2=\|\mu\|^2$ and $\alpha=\langle \mu,v\rangle$ yields the representation
\[
\mathcal{R}_{\infty}^{\mathrm{ICL}}(\mu)
=
\tilde{\mathcal{R}}^{\mathrm{ICL}}(r^2,\alpha^2).
\]

Differentiating this expression with respect to $\mu$, using $\nabla r^2 = 2\mu$ and $\nabla \alpha^2 = 2\alpha v$, gives the stated gradient form
\[
\nabla \mathcal{R}_{\infty}^{\mathrm{ICL}}(\mu)
=
2\big(A(r,\alpha)\mu + B(r,\alpha)\alpha v\big),
\]
where $A=\nabla_{r^2}\tilde{\mathcal{R}}^{\mathrm{ICL}}$ and $B=\nabla_{\alpha^2}\tilde{\mathcal{R}}^{\mathrm{ICL}}$ are polynomials in $(r^2,\alpha^2)$ whose coefficients are obtained by explicit identification. The expressions of $a_i,b_i$ follow from direct computation.
\end{proof}

\begin{lemma}\label{conditions}
    Let $\xi^2,\theta,\lambda>0$ and $n\geq d\geq 1$, and $a_i,b_i$ defined as in \eqref{ab}, then $a_1b_3>a_2a_3$ and $(a_1+a_2)b_3>(a_2+b_2)a_3$.
\end{lemma}
\begin{proof}
Developing the terms, we have that \begin{align*}
     a_1b_3-a_2a_3&= 2\lambda^3n^2[c_1\theta\xi^8+c_2\theta^2\xi^6+c_3\theta^3\xi^4],\\
         (a_1+a_2)b_3-(a_2+b_2)a_3&= 2\lambda^3n^2(n+2)[c_4\theta\xi^8+c_5\theta^2\xi^6+c_6\theta^3\xi^4+c_7\theta^4\xi^2],
 \end{align*} 
 with \begin{align*}
     c_1&=(n+2)(d(n-1)+n^2+4n+3),\\
     c_2&=(n+2)(d(n-1)+n^2+5n+2),\\
     c_3&=(n-1)(n+2),\\
     c_4&=d(n-1)+n^2+4n+3,\\
     c_5&=(2d(n-1)+3n^2+13+8),\\
     c_6&=d(n-1)+3n^2+14n+7,\\
     c_7&=n^2+5n+2.
 \end{align*}
 Since $n\geq d\geq1$, we have that every constant $c_1,\ldots,c_7$ is positive, concluding the lemma.
\end{proof}

\begin{proposition}[Families of stationary points]\label{families}
Under the notation of Lemma \ref{lemma:icl_expansion} and the parametrization $\mu = \alpha v + w$ with $w \perp v$ and $r = \|\mu\|$, all stationary points of $\mathcal{R}_{\infty}^{\mathrm{ICL}}$ belong to one of the following families:

\begin{enumerate}
\item \(\textbf{Trivial: }\mu=0.\)

\item \(\textbf{Orthogonal: }\alpha=0,\; w\neq 0\), with \(A(r,0)=0.\)

\item \(\textbf{Aligned: }w=0,\; \alpha\neq 0\), so that \(\mu=\alpha v\), with
\[
A(r,\alpha)+B(r,\alpha)=0.
\]

\item \(\textbf{Off-axis: }\alpha\neq 0,\; w\neq 0\), with
\[
A(r,\alpha)=0
\quad \text{and} \quad
B(r,\alpha)=0.
\]
\end{enumerate}
\end{proposition}
\begin{proof}
    A stationary point satisfies
\[
A(r,\alpha)\mu + B(r,\alpha)\alpha v = 0.
\]
Writing $\mu=\alpha v + w$ with $w\perp v$ and projecting onto $\mathrm{span}(v)$ and its orthogonal complement yields
\[
A(r,\alpha)w = 0,
\qquad
\alpha\big(A(r,\alpha)+B(r,\alpha)\big)=0.
\]
The conclusion follows by considering whether $\alpha=0$ or not and whether $w=0$ or not.
\end{proof}

\begin{proposition}[Characterization of families]\label{criticalglobal}

The function $\mathcal{R}_{\infty}^{\mathrm{ICL}}:\mathbb{R}^d \to \mathbb{R}$ satisfies:

\begin{enumerate}
    \item $\mu = 0$ is a local maximum: the Hessian has only negative eigenvalues along all nonzero directions.
    \item Any admissible orthogonal point is a strict saddle: the Hessian has exactly one negative eigenvalue along $v$, one positive eigenvalue along the vector's own direction and $d-2$ zero eigenvalues.
    \item Any admissible aligned point is a strict local minimum. Moreover, the only admissible aligned points are $\mu^\star = \pm \alpha^\star v$, for $\alpha^\star>0$ defined in \eqref{astar}.
    \item There are no admissible off-axis solutions.
\end{enumerate}

Since $\mathcal{R}_{\infty}^{\mathrm{ICL}}$ is coercive, these two points are the global minimizers of the function. Finally, for almost every initialization $\mu_0 \in \mathbb{R}^d$, the gradient flow of $\mathcal{R}_{\infty}^{\mathrm{ICL}}$ converges to one of the two global minimizers $\pm \alpha^\star v$.
\end{proposition}

\begin{proof}
We recall by Lemma \ref{lemma:icl_expansion} that $\nabla \mathcal{R}_{\infty}^{\mathrm{ICL}}(\mu)=2(A(r,\alpha)\mu+B(r,\alpha)\alpha v)$, where $$A(r,\alpha)=a_1r^2+a_2\alpha^2-a_3, \quad B(r,\alpha)=b_1r^2+b_2\alpha^2-b_3,$$
for $a_i,b_i, i=\{1,2,3\}$ defined as \begin{align}
\begin{cases}
    a_1&=2\lambda^2n(n+2)\xi^4[(n+d+3)\xi^2+\theta],\\
    a_2 &=\lambda^2n(n+2)\xi^2\theta[(3n+2d+9)\xi^2+(n+5)\theta],\\
    a_3&=2\lambda n \xi^2[(n+d+1)\xi^2+\theta],\\
    b_1&=a_2,\\
    b_2&=2\lambda^2 n(n+2)\theta^2[(2n+d+6)\xi^2+(n+4)\theta],\\
    b_3&=2\lambda n\theta[(2n+d+2)\xi^2+(n+2)\theta].
\end{cases}
\end{align} By Lemma \ref{conditions}, we have that 
\begin{align}
    \label{c1} a_1 b_3 &> a_2 a_3, \\
    \label{c2} (a_1+a_2) b_3 &> (a_2+b_2) a_3,
\end{align} holds for any $\lambda,\xi^2, \theta>0$ and $n\geq d\geq 1$.\\

\noindent Differentiating the gradient we obtain the following Hessian:
\begin{align*}
\nabla^2 \mathcal{R}_{\infty}^{\mathrm{ICL}}(\mu)&=2AI_d + 2Bvv^\top+4(\partial_{r^2}A)\mu\mu^\top+4\alpha^2(\partial_{\alpha^2}B) v v^\top+4\alpha(\partial_{\alpha^2}A)(\mu v^\top+ v\mu^\top).\\
&=2AI_d + 2Bvv^\top+4a_1\mu\mu^\top+4\alpha^2b_2v v^\top+4\alpha a_2(\mu v^\top+ v\mu^\top).
\end{align*}
We check that:
\begin{itemize}
\item Origin: We have that $\nabla^2 \mathcal{R}_{\infty}^{\mathrm{ICL}}(0)=2A(0,0)I_d+2B(0,0)vv^\top=-2a_3-2b_3vv^\top$, since $a_3,b_3>0$, all eigenvalues are strictly negative.
\item Orthogonal: Here $r^2=\frac{a_3}{a_1}$ and the Hessian simplifies to $\nabla^2 \mathcal{R}_{\infty}^{\mathrm{ICL}}(\mu)=4a_1\mu\mu^\top+2B(r,0)vv^\top$. In direction $\mu$ the associated eigenvalue is $4a_1r^2>0$ and in direction $v$, is $2B(r,0)=2(b_1\frac{a_3}{a_1}-b_3)=2(a_2\frac{a_3}{a_1}-b_3)$, which is negative due to \eqref{c1}. In any other direction $w\in \{\mu,v\}^{\perp}$, we will have null eigenvalues, in particular $\mathrm{dim}(\{\mu,v\}^{\perp})=d-2$. 
\item Aligned: Here $\mu=\alpha v$, so $r^2=\alpha^2=\frac{a_3+b_3}{a_1+a_2+b_1+b_2}=\frac{a_3+b_3}{a_1+2a_2+b_2}$, if \begin{equation}\label{astar}
    \alpha^\star =\sqrt{\frac{a_3+b_3}{a_1+2a_2+b_2}},
\end{equation}
then $\mu=\pm \alpha^\star v$ and the Hessian becomes $\nabla^2 \mathcal{R}_{\infty}^{\mathrm{ICL}}(\mu)=2A(I-vv^\top)+4(a_3+b_3)vv^\top$, in direction $v$ the associated eigenvalue is $4(a_3+b_3)>0$. In perpendicular directions to $v$, the associated eigenvalue is $$2A(r,\alpha)=2[(a_1+a_2)\frac{a_3+b_3}{a_1+2a_2+b_2}-a_3],$$ which is positive due to \eqref{c2}. 
\item Off-axis: This solution exists only if $a_1r^2+a_2\alpha^2-a_3=0$ and $a_2r^2+b_2\alpha^2-b_3=0$, and $r^2>\alpha^2>0$, we solve the system of equations $$\begin{pmatrix}a_1 & a_2\\a_2 &b_2 \end{pmatrix}\begin{pmatrix}
    r^2 \\\alpha^2
\end{pmatrix}=\begin{pmatrix}
    a_3 \\ b_3
\end{pmatrix}.$$ Let $\Delta=a_1b_2-a_2^2$, if:
\begin{itemize}
    \item $\Delta=0$, the matrix is singular and due to \eqref{c1}, we have that the system has no solution.
    \item $\Delta<0$, the solution $\alpha^2=\frac{a_1b_3-a_2a_3}{\Delta}$ is negative due to \eqref{c1}, and we have a contradiction with the fact that $\alpha^2\geq 0$.
    \item $\Delta>0$, solving the system yields $r^2-\alpha^2=\frac{(a_2+b_2)a_3-(a_1+a_2)b_3}{\Delta}$, and by \eqref{c2}, then $r^2-\alpha^2<0$, which is a contradiction.
\end{itemize}  
Thus, under our two conditions \eqref{c1}-\eqref{c2}, no off-axis solution exists.
\end{itemize}
\end{proof}

\begin{proposition}\label{criticalpointsicl}
Let $\Sigma \sim W_d(V,n)$ with $V=\xi^2I_d+\theta vv^t$ and $\Vert v\Vert$=1. Then, for $L$ large enough, we can characterize the landscape of critical points of $\mathcal{R}_{L}^{\mathrm{ICL}}$ within a compact region. More precisely, for sufficiently large $\rho>0$, there exists $L_0\in\mathbb{N}$, such that for $L\geq L_0$, the set of critical points of $\mathcal{R}_{L}^{\mathrm{ICL}}$ that lies in $B(0,\rho)$ is contained in the union of the following sets:
$$\mathrm{crit}(\mathcal{R}_{L}^{\mathrm{ICL}})\cap B(0,\rho)\subseteq \{\mu_{L,0}^\star\}\cup\{\pm\mu_{L,\parallel}^\star\}\cup \{U_j: j=1,\ldots,2(d-1)\},$$  
where:
\begin{enumerate}
\item \label{item1}
The point $\mu_{L,0}^\star$ satisfies $\mu_{L,0}^\star \to 0$ as $L \to \infty$. This point is a strict local maximum.

\item \label{item2}
The points $\pm\mu_{L,\parallel}^\star$ satisfy $\pm\mu_{L,\parallel}^\star \to \pm\alpha^\star v$ as $L \to \infty$, with $\alpha^\star$ according to \eqref{astar}. These points are local minimizers.

\item \label{item3}
For each orthogonal critical point $\mu_\perp^{(j)},\ j=1,\dots,2(d-1)$ of $\mathcal{R}_\infty^{\mathrm{ICL}}$, there exists a neighborhood $U_j$ such that every critical point of $\mathcal{R}_L^{\mathrm{ICL}}$ in $U_j$ is a strict saddle.

\end{enumerate}
\end{proposition}
\begin{proof}
The first two items \ref{item1} and \ref{item2} follow directly from Proposition \ref{persistence}. It remains to prove \ref{item3}. Let \[ \mathrm{crit}(\mathcal{R}_\infty^{\mathrm{ICL}}) = \{0,\pm \alpha^\star v,\mu_\perp^{(1)},\ldots,\mu_\perp^{(2(d-1))}\}, \] where $\mu_\perp^{(j)}$ denote the orthogonal critical points of the limiting objective. The points $0$ and $\pm\alpha^\star v$ are nondegenerate, while each $\mu_\perp^{(j)}$ is degenerate as it has $(d-2)$ null eigenvalues. Choose $r_0,r_\parallel>0$ such that the balls \[ B(0,r_0),\qquad B(\pm \alpha^\star v,r_\parallel) \] contain no other critical point of $\mathcal{R}_\infty^{\mathrm{ICL}}$. By Proposition \ref{persistence}, for $L$ sufficiently large there exist unique critical points \[ \mu_{L,0}^\star \in B(0,r_0), \qquad \mu_{L,\parallel}^\star \in B(\alpha^\star v,r_\parallel), \qquad -\mu_{L,\parallel}^\star \in B(-\alpha^\star v,r_\parallel). \] Next, for each orthogonal critical point $\mu_\perp^{(j)}$, choose $r_j>0$ such that $B(\mu_\perp^{(j)},r_j)$ contains no other critical point of $\mathcal{R}_\infty^{\mathrm{ICL}}$. Let $\rho>0$ big enough such that \[ K := B(0,r_0) \cup B(\alpha^\star v,r_\parallel) \cup B(-\alpha^\star v,r_\parallel) \cup \bigcup_{j=1}^{2(d-1)}B(\mu_\perp^{(j)},r_j)\subset B(0,\rho). \] Since $\nabla \mathcal{R}_\infty^{\mathrm{ICL}}$ is continuous and has no zero on $K^c$, for every $\rho>0$ there exists \[ \eta_\rho := \inf_{x\in K^c\cap B(0,\rho)} \|\nabla \mathcal{R}_\infty^{\mathrm{ICL}}(x)\| >0 . \] Because $\mathcal{R}_L^{\mathrm{ICL}}\to \mathcal{R}_\infty^{\mathrm{ICL}}$ in $C^1_{\mathrm{loc}}$, we have \[ \sup_{x\in B(0,\rho)} \|\nabla \mathcal{R}_L^{\mathrm{ICL}}(x)-\nabla \mathcal{R}_\infty^{\mathrm{ICL}}(x)\| \to0 . \] Therefore $\mathcal{R}_L^{\mathrm{ICL}}$ has no critical point in $K^c\cap B(0,\rho)$, and then $\mathrm{crit}(\mathcal{R}_L^{\mathrm{ICL}})\cap B(0,\rho)\subseteq K$.

Finally, fix one orthogonal critical point $\mu_\perp^{(j)}$ and consider the ball
$B(\mu_\perp^{(j)},r_j)$. By Proposition~\ref{criticalglobal}, the Hessian
$\nabla^2\mathcal{R}_\infty^{\mathrm{ICL}}(\mu_\perp^{(j)})$ has one positive
eigenvalue $\lambda_+>0$ and one negative eigenvalue $\lambda_-<0$.
Let
\[
\gamma := \min\{\lambda_+,-\lambda_-\} >0 .
\]

Since $\nabla^2\mathcal{R}_\infty^{\mathrm{ICL}}$ is continuous,
there exists $r_j>0$ (possibly smaller than before) such that for every
$\mu\in B(\mu_\perp^{(j)},r_j)$ the matrix
$\nabla^2\mathcal{R}_\infty^{\mathrm{ICL}}(\mu)$ has an eigenvalue at least
$\gamma/2$ and another at most $-\gamma/2$.

Because $\mathcal{R}_L^{\mathrm{ICL}}\to
\mathcal{R}_\infty^{\mathrm{ICL}}$ in $C^2_{\mathrm{loc}}$, we have
\[
\sup_{\mu\in B(\mu_\perp^{(j)},r_j)}
\left\|
\nabla^2\mathcal{R}_L^{\mathrm{ICL}}(\mu)
-
\nabla^2\mathcal{R}_\infty^{\mathrm{ICL}}(\mu)
\right\|
\to 0 .
\]

Hence for $L$ sufficiently large and every $\mu\in B(\mu_\perp^{(j)},r_j)$,
\[
\left\|
\nabla^2\mathcal{R}_L^{\mathrm{ICL}}(\mu)
-
\nabla^2\mathcal{R}_\infty^{\mathrm{ICL}}(\mu)
\right\|
\le \frac{\gamma}{4}.
\]

By continuity of the eigenvalues (Weyl's inequality), the Hessian
$\nabla^2\mathcal{R}_L^{\mathrm{ICL}}(\mu)$ has one eigenvalue
at least $\gamma/4$ and another at most $-\gamma/4$ for all
$\mu\in B(\mu_\perp^{(j)},r_j)$.

Let $\mu_L$ be a critical point of $\mathcal{R}_L^{\mathrm{ICL}}$
in $B(\mu_\perp^{(j)},r_j)$. Then
$\nabla\mathcal{R}_L^{\mathrm{ICL}}(\mu_L)=0$, and the Hessian at this
point has both a positive and a negative eigenvalue. Therefore
$\mu_L$ is a strict saddle.
\end{proof}

\section{Numerical experiments}\label{sec:numerical}

In this section, we present numerical experiments that illustrate and empirically validate the theoretical results developed throughout the paper. In particular, we study the convergence behavior of the different models toward the principal eigenvector of the underlying covariance structure, as well as the effect of key parameters such as the prompt length and the ambient dimension.

\subsection{Experimental Setup}

We consider a covariance matrix of the form
\[
\Sigma = A A^\top + 0.1 I_d,
\]
where $A \in \mathbb{R}^{d \times d}$ has i.i.d. standard Gaussian entries. The target direction is given by the principal eigenvector $u_1$ associated with the largest eigenvalue of $\Sigma$.

All experiments are conducted using stochastic gradient descent (SGD) with constant step size. More precisely, let $\mathcal{R}(\mu)$ denote the objective function of interest (either the empirical risk or the population risk, depending on the setting). The iterates $\{\mu_k\}_{k\geq 0}$ are defined by
\[
\mu_{k+1}
=
\mu_k - \gamma \, g_k,
\]
where $\gamma > 0$ is a constant learning rate and $g_k$ is a stochastic gradient estimator of $\nabla \mathcal{R}(\mu_k)$.

In the finite-prompt setting, $g_k$ is computed from a random batch of samples. For instance, in the softmax model, we draw $X_1^{(k)}, \dots, X_L^{(k)} \sim \mathcal{N}(0,\Sigma)$ and define
\[
g_k = \nabla_\mu \widehat{\mathcal{R}}_{L}(\mu_k; X_1^{(k)}, \dots, X_L^{(k)}),
\]
where $\widehat{\mathcal{R}}_{L}$ is the empirical risk associated with the sampled prompt. More precisely, at each iteration we draw a batch of size $B$, consisting of independent prompts $(X_{b,1}^{(k)}, \dots, X_{b,L}^{(k)})_{b=1}^B$ with $X_{b,j}^{(k)} \sim \mathcal{N}(0,\Sigma)$, and define
\[
\widehat{\mathcal{R}}_{L}(\mu)
=
\frac{1}{B}
\sum_{b=1}^B
\left\|
X_{b,1}^{(k)} - T_L^\mu(X_{b,1}^{(k)}, \dots, X_{b,L}^{(k)})
\right\|^2,
\]
so that
\[
g_k = \nabla_\mu \widehat{\mathcal{R}}_{L}(\mu_k).
\]

In the infinite-prompt setting, we treat \(\Sigma\) as known (or equivalently \(V = \xi^2 I_d + vv^\top\) in the spiked Wishart model), and we perform deterministic gradient descent, i.e., the gradient is computed directly from the population objective:
\[
g_k = \nabla \mathcal{R}(\mu_k).
\]

The initialization $\mu_0$ is sampled uniformly from $\mathbb{S}^{d-1}$. Unless otherwise specified, we use the following parameters:
\begin{itemize}
    \item Learning rate: $\gamma = 10^{-4}$,
    %\item Number of iterations: $T = 3000$,
    \item Batch size: $B=256$,
    %\item Softmax scaling: $\lambda = 0.1$,
    \item Number of independent runs: $10$.
\end{itemize}

At each iteration $k$, we assess performance via the alignment between the normalized iterate and the target direction. In the standard setting, this is given by
\[
\left|\left\langle \frac{\mu_k}{\|\mu_k\|}, u_1 \right\rangle\right|,
\]
where $u_1$ denotes the principal eigenvector of $\Sigma$. \\

In the spiked Wishart setting, where $\Sigma \sim W_d(\xi^2 I_d + v v^\top)$, we instead measure alignment with the spike direction:
\[
\left|\left\langle \frac{\mu_k}{\|\mu_k\|}, v \right\rangle\right|.
\]
\begin{remark}
Gradient computations in the numerical experiments were carried out using JAX \citep{jax}.
\end{remark}

\subsection{Softmax Attention: Finite and Infinite Prompt}

We first study the softmax attention model in both finite and infinite prompt regimes.

In the finite prompt setting, at each iteration we sample $X_1, \dots, X_L \sim \mathcal{N}(0,\Sigma)$ and perform stochastic gradient updates using the empirical risk of \eqref{rsoftl}. In the infinite prompt setting, we instead optimize on the closed from of the population risk \eqref{rsoftinfty}, which corresponds to the limit as $L \to \infty$.

Figure~\ref{fig:softmax_finite} shows the convergence behavior in the finite prompt case, with $L=100$. Figure~\ref{fig:softmax_infinite} shows the corresponding infinite prompt dynamics. Both figures superposed are shown in Figure~\ref{comparison}.

We observe that in both regimes the iterates converge toward the principal eigenvector. Moreover, the infinite prompt setting exhibits smoother and more stable convergence, as it removes sampling noise. These observations are consistent with the theoretical analysis and illustrate how the finite prompt model approximates the infinite prompt limit.
\begin{comment}
\begin{figure}[htbp]
    \centering
    \includegraphics[width=0.7\linewidth]{plot_softmax_finite_L.pdf}
    \caption{Softmax attention with finite prompt length.}
    \label{fig:softmax_finite}
\end{figure}

\begin{figure}[htbp]
    \centering
    \includegraphics[width=0.7\linewidth]{plot_softmax_infinite.pdf}
    \caption{Softmax attention in the infinite prompt regime.}
    \label{fig:softmax_infinite}
\end{figure}
\end{comment}
\begin{figure}[htbp]
    \centering
    
    \begin{subfigure}[t]{0.48\linewidth}
        \centering
        \includegraphics[width=\linewidth]{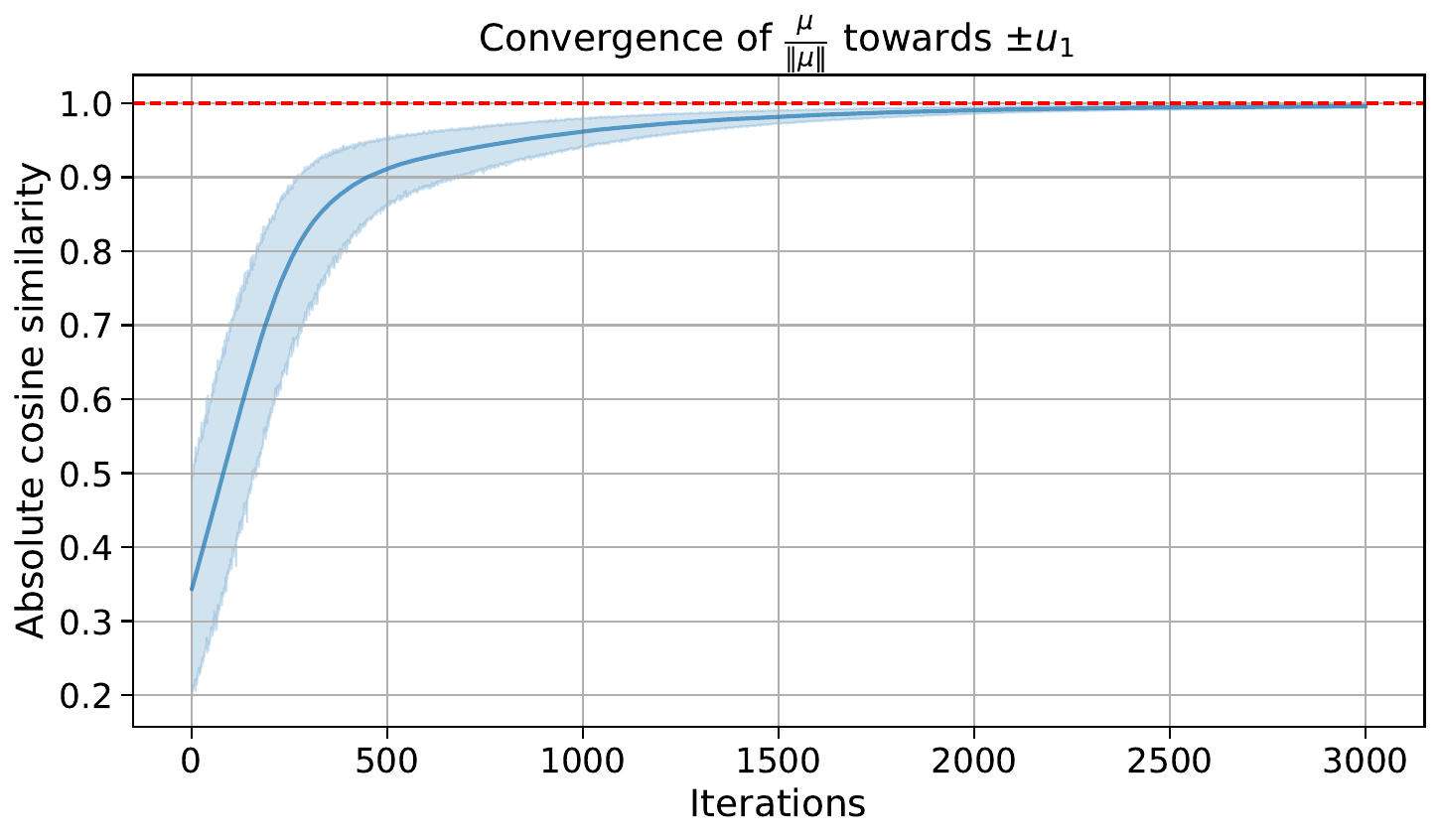}
        \caption{SGD on finite prompt risk, $L=100$.}
        \label{fig:softmax_finite}
    \end{subfigure}
    \hfill
    \begin{subfigure}[t]{0.48\linewidth}
        \centering
        \includegraphics[width=\linewidth]{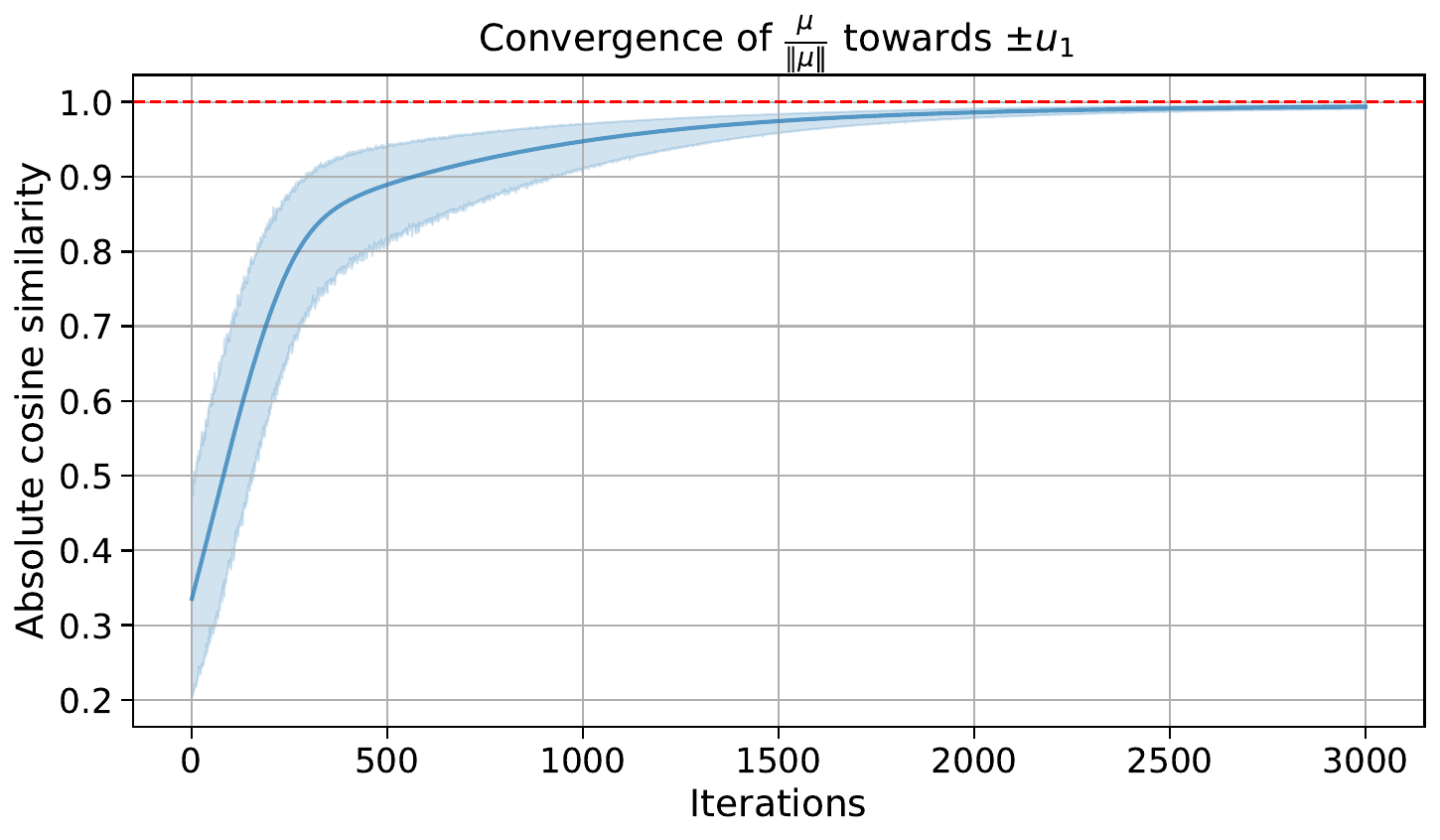}
        \caption{GD on infinite prompt risk.}
        \label{fig:softmax_infinite}
    \end{subfigure}
    
    \caption{Softmax attention: finite vs. infinite prompt regimes.}
    \label{fig:softmax_comparison}
\end{figure}

\subsection{Linear Attention}

We next consider the linear attention model introduced in Section~\ref{sec:app:linear_version}. This model replaces the softmax weighting with a linear aggregation rule, leading to a simpler objective.

As illustrated in Figure~\ref{fig:linear_attention_combined}, we consider the setting \(d=5\), \(L=6\), and \(\lambda=0.01\). In Figure~\ref{linear1}, we run SGD on the empirical risk associated with~\eqref{rlin}, while in Figure~\ref{linear2}, we perform gradient descent on its analytic counterpart~\eqref{rlinl}. 

In both cases, the iterates converge toward the principal eigenvector of \(\Sigma\). This demonstrates that the recovery of the leading principal component is not specific to the softmax mechanism, but rather reflects a more general phenomenon driven by the structure of the aggregation underlying attention.

\begin{figure}[htbp]
    \centering
    \begin{subfigure}{0.48\linewidth}
        \centering
        \includegraphics[width=\linewidth]{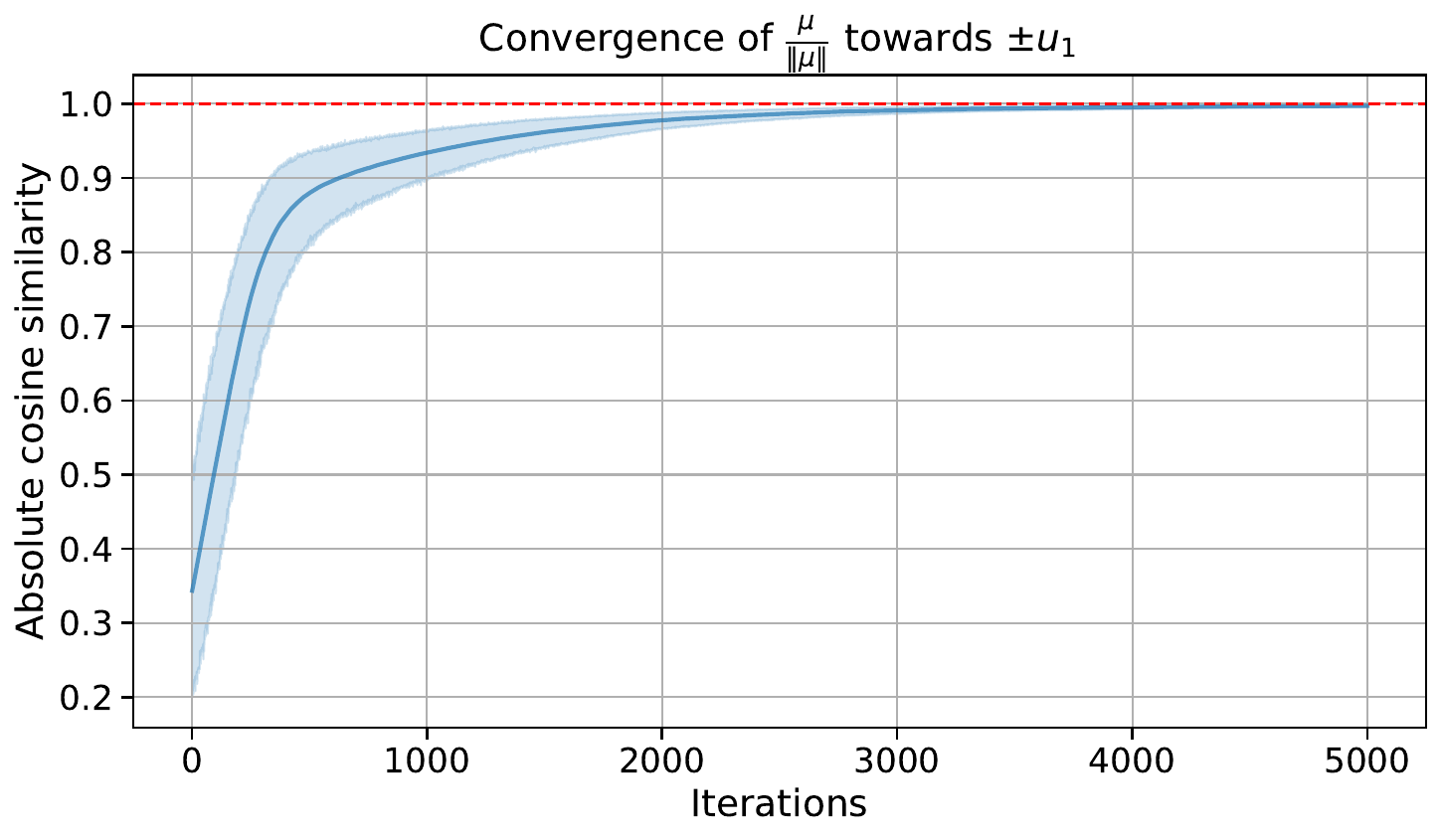}
        \caption{SGD on empirical risk of linear attention.}
        \label{linear1}
    \end{subfigure}
    \hfill
    \begin{subfigure}{0.48\linewidth}
        \centering
        \includegraphics[width=\linewidth]{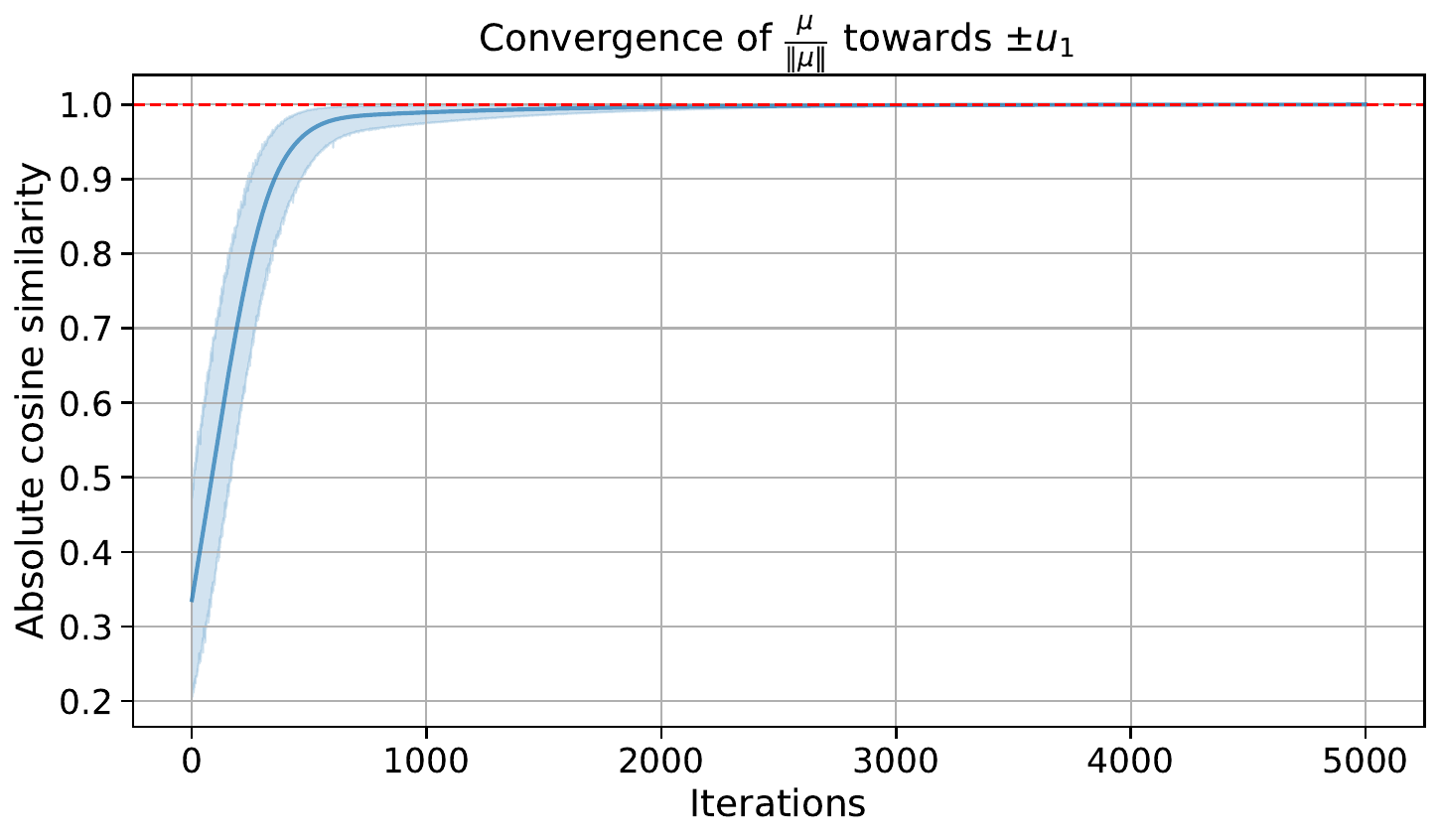}
        \caption{GD on analytic risk of linear attention.}
        \label{linear2}
    \end{subfigure}
    \caption{Convergence of linear attention toward the principal eigenvector under empirical and analytic risks.}
    \label{fig:linear_attention_combined}
\end{figure}

\subsection{Scaling with Prompt Length}

We now investigate the role of the prompt length $L$ in both the softmax and linear attention settings. In both cases, we consider prompt lengths $L$ ranging from $3$ to $50$ using $20$ evenly spaced values. We fix the dimension to $d=5$, and run the optimization for $T=5000$ iterations. For each value of $L$, we perform $10$ independent runs and measure the final alignment.

The only difference between the two settings lies in the choice of the parameter $\lambda$: we use $\lambda=0.1$ for the softmax model and $\lambda=0.001$ for the linear attention model.

As shown in Figure~\ref{fig:softmax_vs_L}, performance improves with $L$ in both settings, illustrating the transition from the finite-prompt regime to the population regime. In particular, the alignment increases and stabilizes as $L$ grows, providing empirical evidence that the finite-prompt model converges toward its infinite-prompt counterpart. Moreover, the consistency between the softmax and linear cases suggests that this behavior is not specific to the softmax mechanism, but rather reflects a more general phenomenon tied to the structure of the aggregation.

\begin{figure}[htbp]
    \centering
    \begin{subfigure}[t]{0.48\linewidth}
        \centering
        \includegraphics[width=\linewidth]{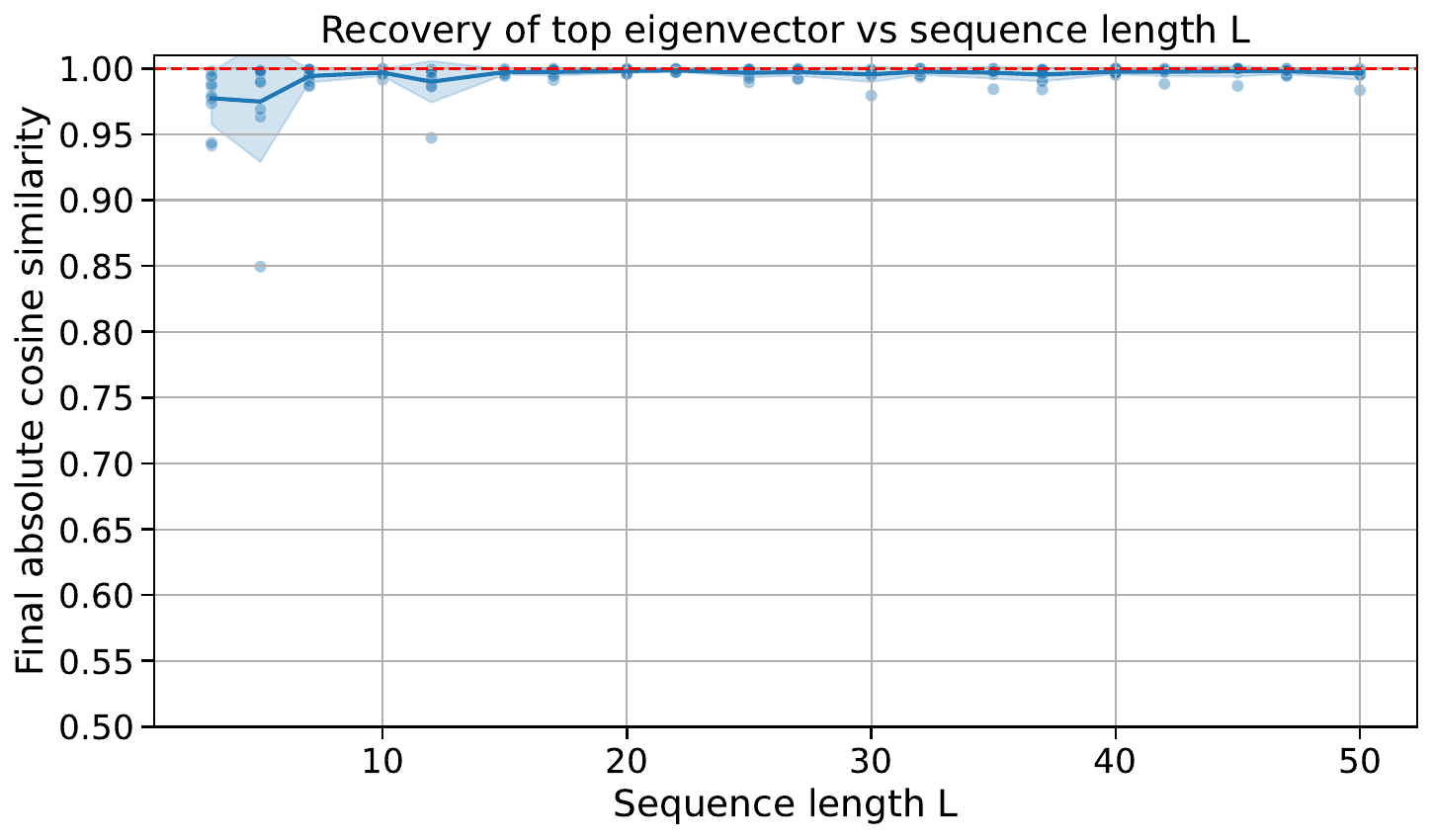}
        \caption{Softmax attention}
    \end{subfigure}
    \hfill
    \begin{subfigure}[t]{0.48\linewidth}
        \centering
        \includegraphics[width=\linewidth]{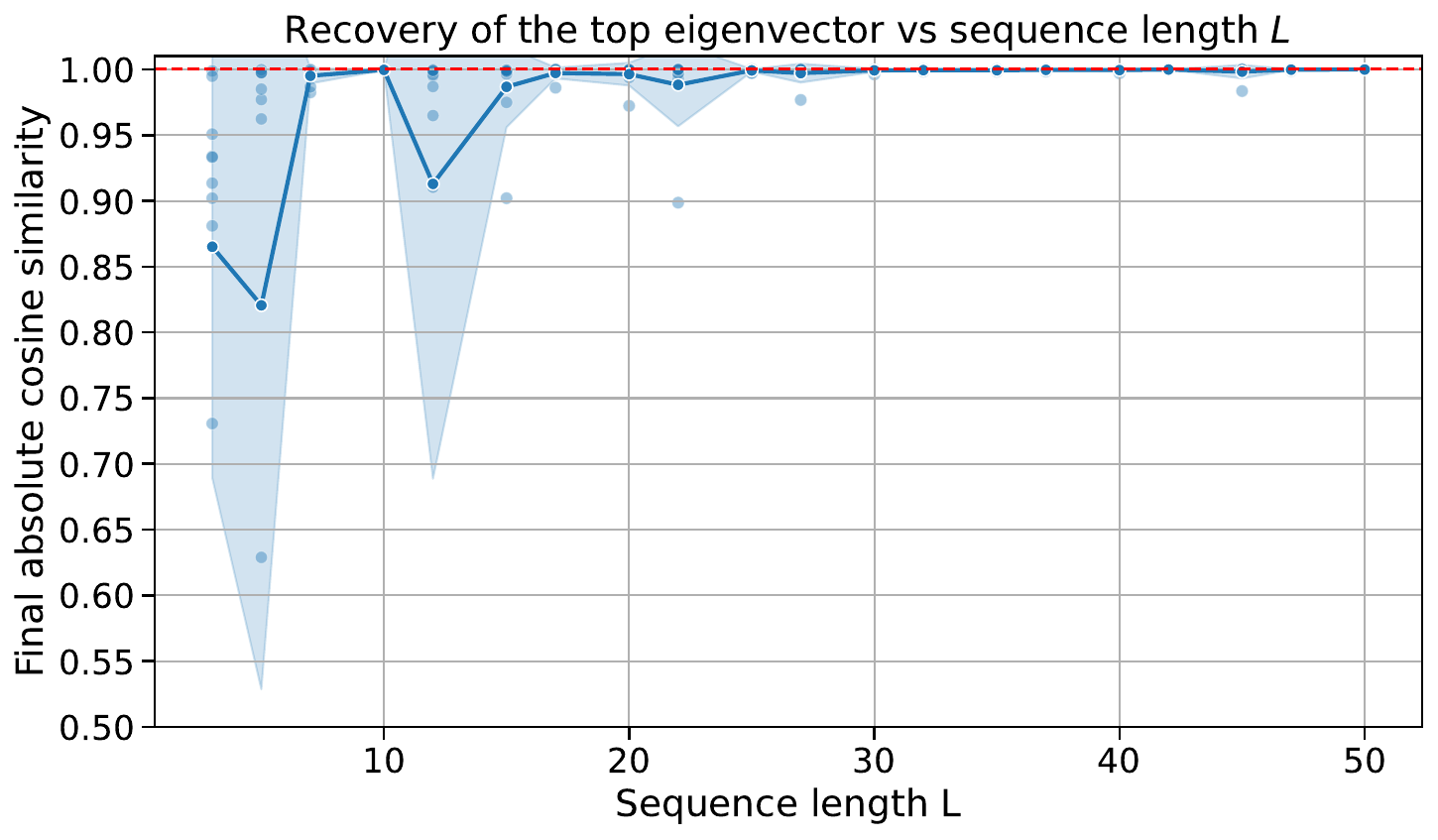}
        \caption{Linear attention}
    \end{subfigure}
    \caption{Final alignment as a function of the prompt length $L$ for both softmax and linear attention models.}
    \label{fig:softmax_vs_L}
\end{figure}

\subsection{Scaling with Dimension}

We also study the effect of the ambient dimension \(d\) across both the softmax and linear attention models. For dimensions \(d\) ranging from 3 to 100 in increments of 5, we generate a new covariance matrix \(\Sigma\) for each dimension and evaluate the final alignment after \(T=5000\) iterations, averaging over 10 independent runs in each case.

In all experiments, we scale the hyperparameters with the dimension. In the softmax model, we set the learning rate \(\gamma = 0.5/d^2\) and \(\lambda = 0.1/d\), while in the linear attention model we use \(\gamma = 1/d^2\) and \(\lambda = 0.01/d\). We also fix the context length to \(L = d\) in the linear case.

For the softmax model, we analyze the infinite-prompt regime, which admits an explicit closed-form expression depending on \(\Sigma\) (see \eqref{rsoftinfty}). For the linear attention model, we consider its finite-prompt formulation, which also admits an explicit closed-form expression (for fixed \(L\)) depending on \(\Sigma\) (see \eqref{rlinl}).

Figure~\ref{fig:dimension_scaling} reports the resulting performance as a function of \(d\). In both models, we observe that increasing the dimension makes recovery more challenging, reflecting the growing difficulty of estimating the principal component in higher-dimensional spaces. Despite this degradation, both methods consistently retain a strong alignment with the leading eigenvector, highlighting the robustness of the underlying mechanism.
\begin{comment}
\begin{figure}[htbp]
    \centering
    \includegraphics[width=0.7\linewidth]{plot_alignment_vs_dimension.pdf}
    \includegraphics[width=0.7\linewidth]{plot_linear_vs_dimension.pdf}
    \caption{Final alignment as a function of the dimension $d$ for softmax (top) and linear attention (bottom).}
    \label{fig:dimension_scaling}
\end{figure}
\end{comment}
\begin{figure}[htbp]
    \centering
    
    \begin{subfigure}[t]{0.48\linewidth}
        \centering
        \includegraphics[width=\linewidth]{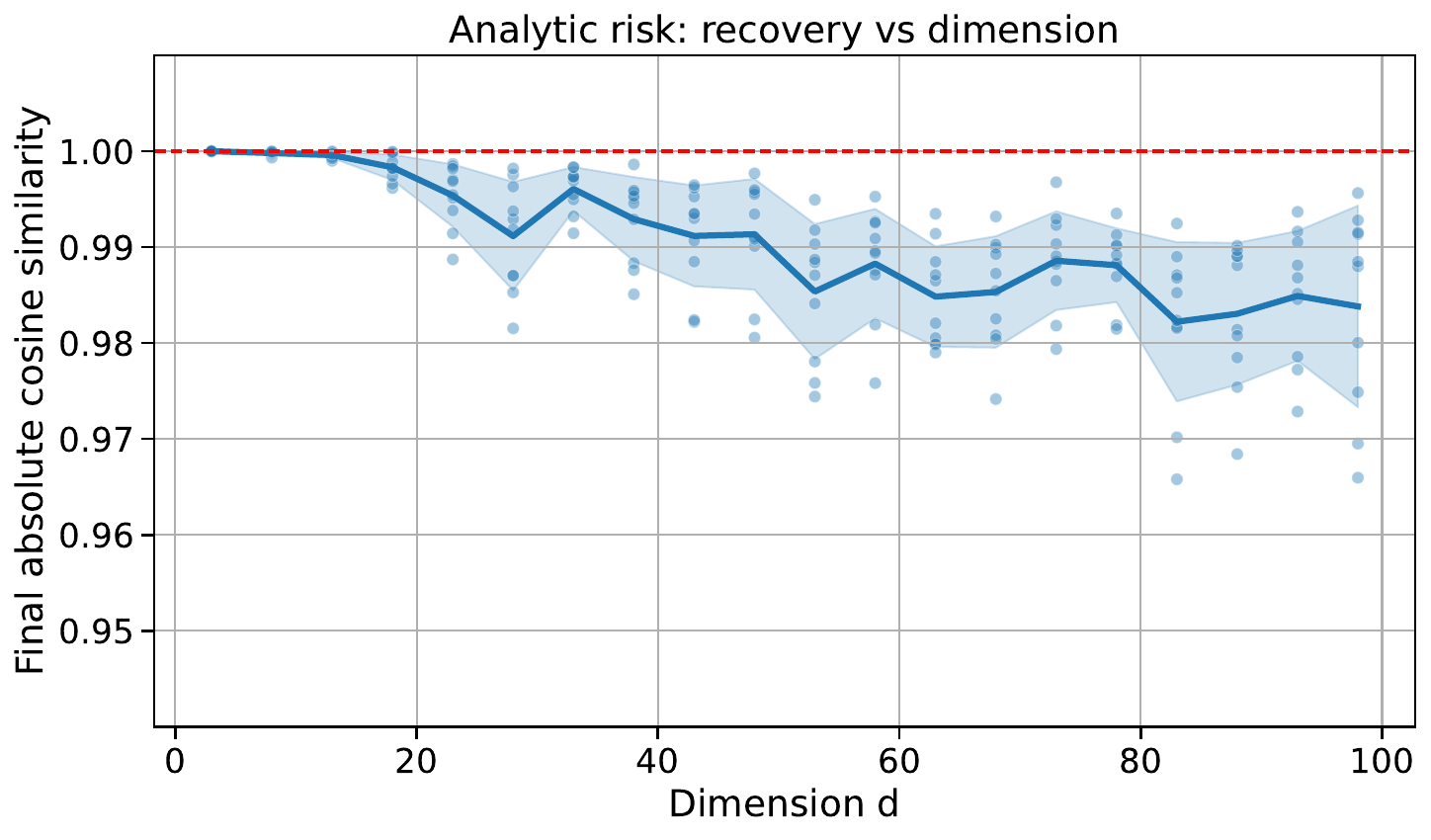}
        \caption{Softmax attention}
        \label{fig:dim_softmax}
    \end{subfigure}
    \hfill
    \begin{subfigure}[t]{0.48\linewidth}
        \centering
        \includegraphics[width=\linewidth]{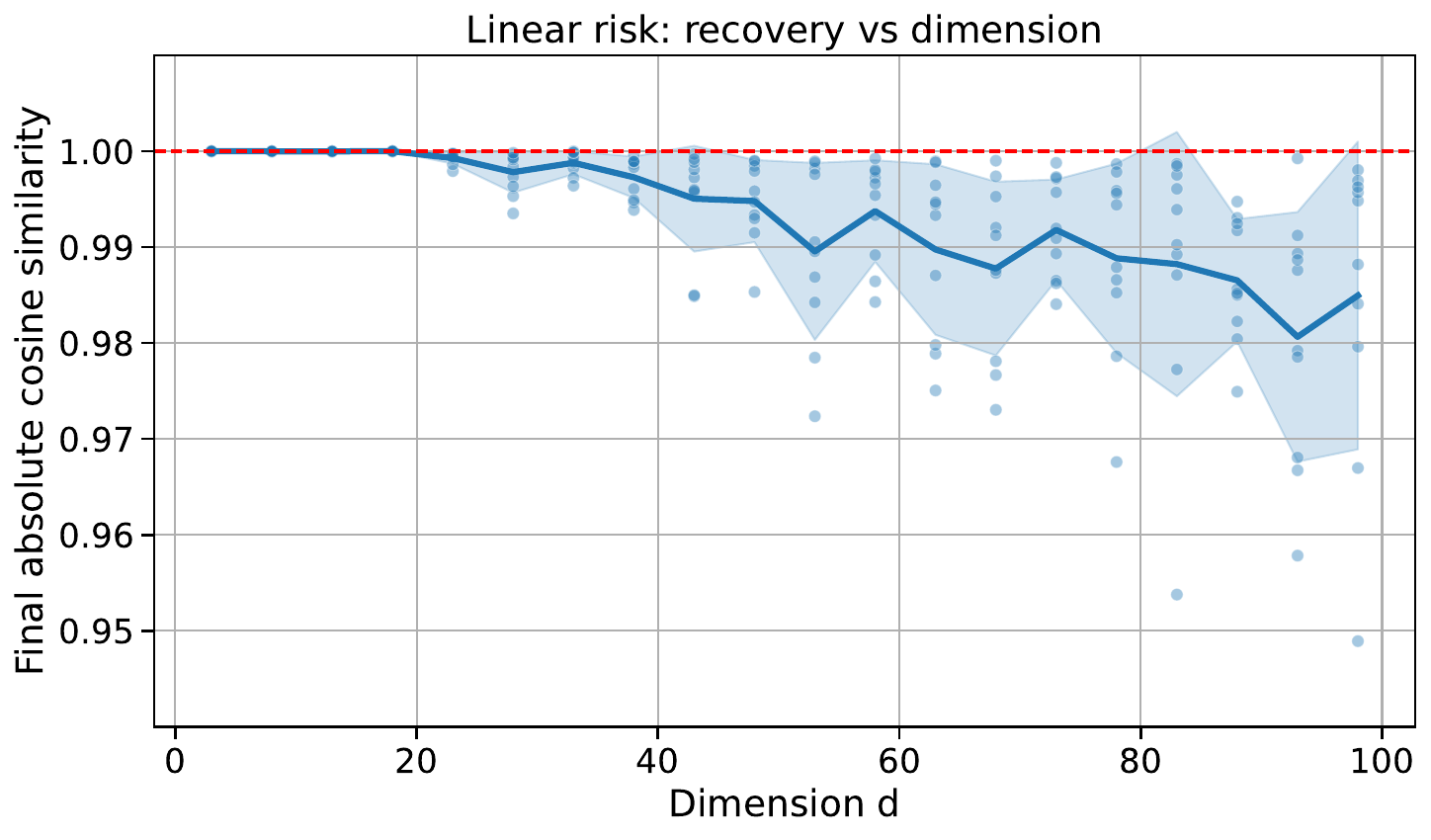}
        \caption{Linear attention}
        \label{fig:dim_linear}
    \end{subfigure}
    
    \caption{Final alignment as a function of the dimension $d$.}
    \label{fig:dimension_scaling}
\end{figure}

\subsection{In-Context Learning: Finite and Infinite Prompt}

Finally, we consider the in-context learning (ICL) setting, where the covariance matrix is itself random and follows a spiked Wishart distribution:
\[
\Sigma \sim W_d\big(\xi^2 I_d + \theta v v^\top, n\big),
\]
with $d=5, \xi = 1$, $\theta = 2$, and $n = 10$.

 In the finite-prompt regime, we approximate the risk \eqref{icl_L} using Monte Carlo sampling with $100$ samples of $\Sigma$ and $100$ samples of data per covariance matrix. In the infinite-prompt regime, we optimize the population risk directly using Lemma \ref{lemma:icl_expansion}.

Figures~\ref{fig:icl_finite} and \ref{fig:icl_infinite} show convergence toward the spike direction $v$. Both figures superposed are shown in Figure~\ref{comparison_ICL}. %The infinite case exhibits smoother convergence, while the finite case reflects additional variability due to sampling.

\begin{figure}[htbp]
    \centering
    
    \begin{subfigure}[t]{0.48\linewidth}
        \centering
        \includegraphics[width=\linewidth]{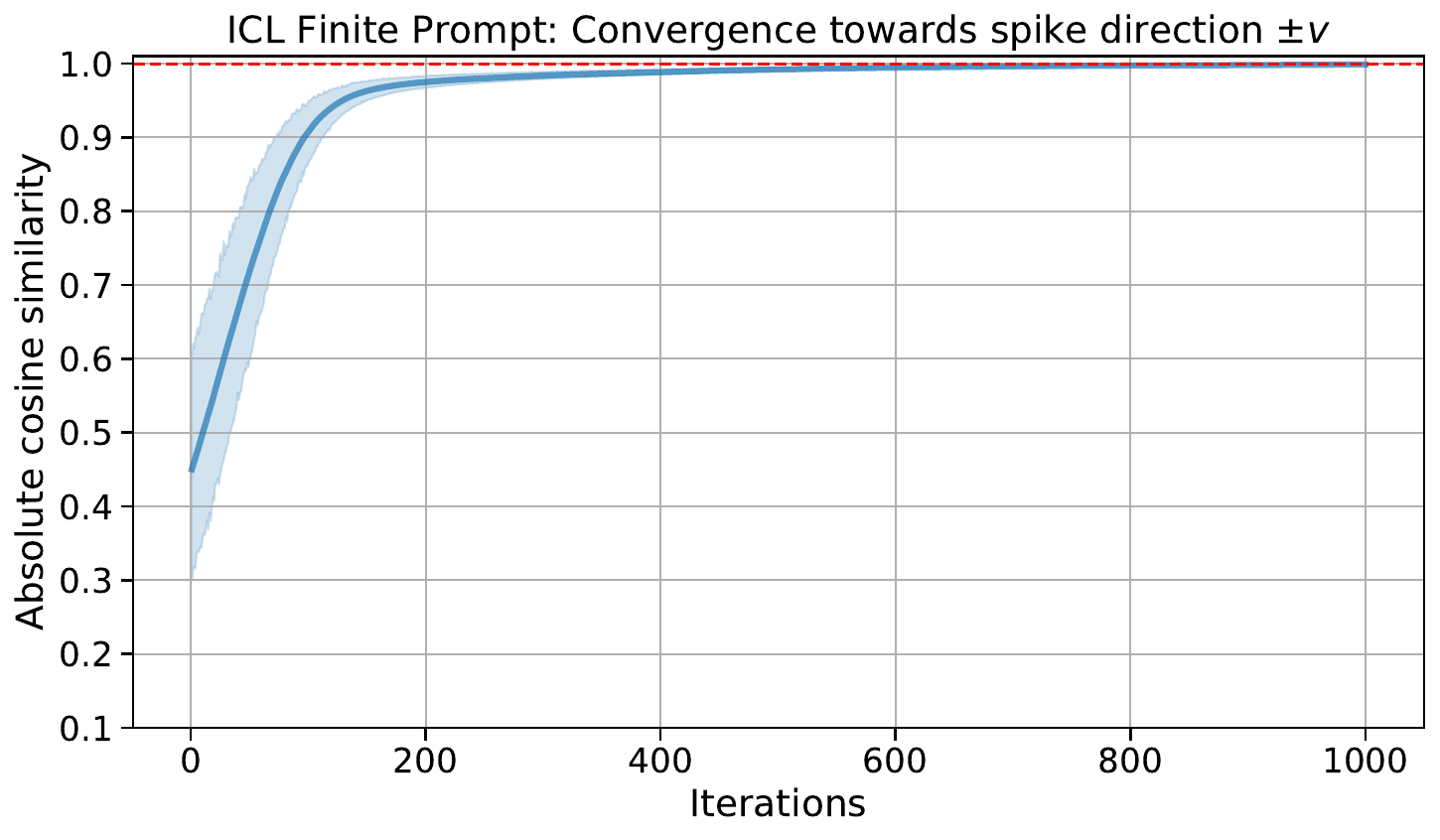}
        \caption{SGD on finite prompt ICL risk, $L=100$}
        \label{fig:icl_finite}    
    \end{subfigure}
    \hfill
    \begin{subfigure}[t]{0.48\linewidth}
        \centering
        \includegraphics[width=\linewidth]{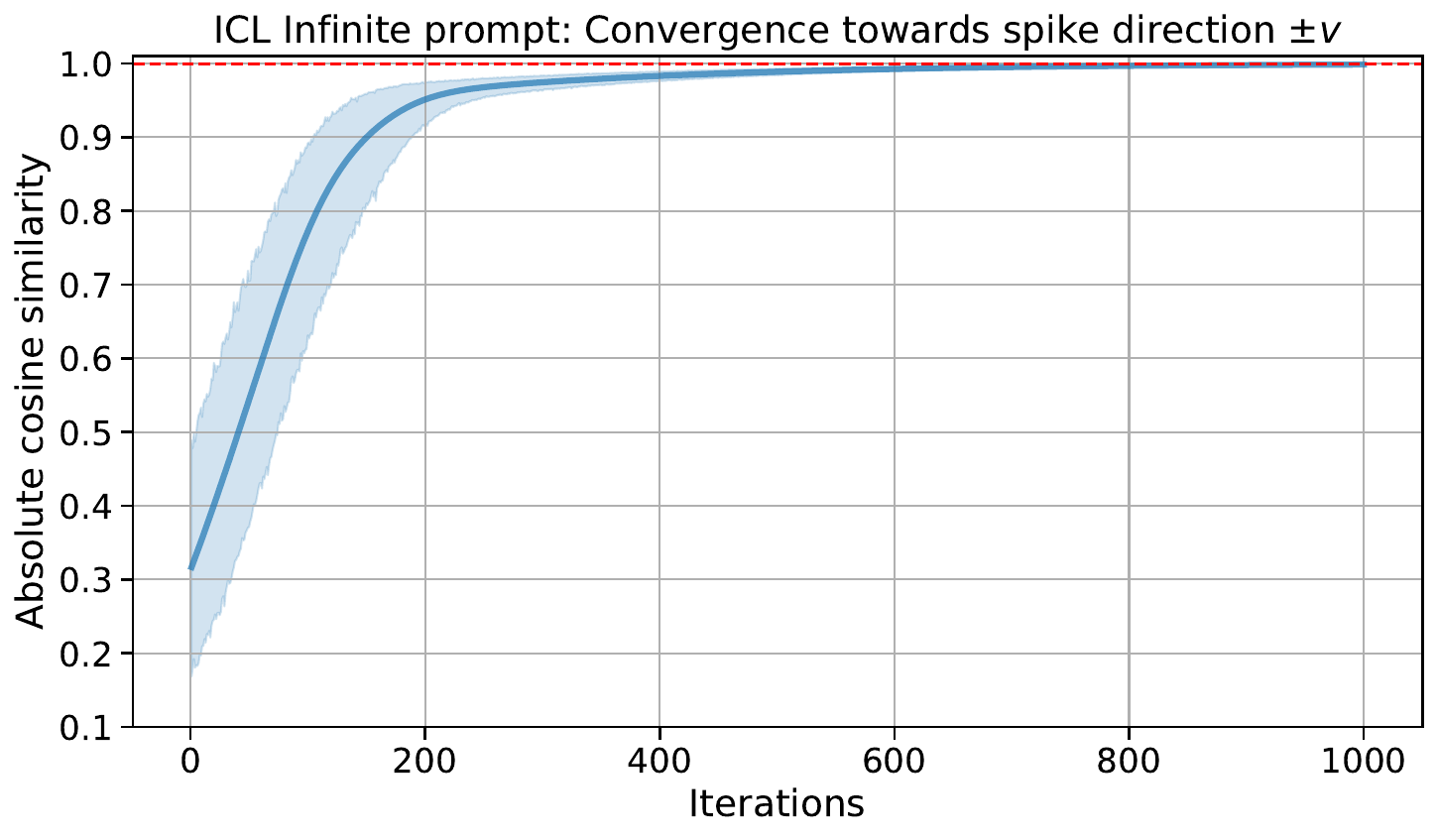}
        \caption{GD on infinite prompt ICL risk}
        \label{fig:icl_infinite}
    \end{subfigure}
    
    \caption{In-context learning: finite vs. infinite prompt regimes.}
    \label{fig:icl_comparison}
\end{figure}

\subsection{In-context learning: Scaling with Prompt Length}
We now study the effect of the prompt length $L$ in the in-context learning (ICL) setting. We consider values of $L$ ranging from 3 to 100 in increments of 5. For each value of $L$, we perform $10$ independent runs and report the final alignment after $T=3000$ iterations.

In the finite-prompt regime, for each iteration we sample covariance matrices from the spiked Wishart distribution and generate data accordingly, while in the infinite-prompt regime we directly optimize the corresponding population risk.

Figure~\ref{icl1} shows the final alignment as a function of $L$. We observe that, similarly to the standard softmax setting, performance improves as the prompt length increases. This reflects the fact that larger prompts provide a better approximation of the population objective, reducing the variability induced by sampling both the data and the covariance matrices.

These results further support the theoretical prediction that the finite-prompt ICL model converges toward its infinite-prompt counterpart as $L$ grows.

\subsection{Scaling with Dimension}
In this experiment, we investigate how the ambient dimension \(d\) affects the performance of the infinite-prompt in-context learning (ICL) model. We consider dimensions \(d\) ranging from 3 to 100 in increments of 5. For each value of \(d\), we fix \(n = d\), sample a covariance matrix \(\Sigma\), and evaluate the alignment of the learned direction with the spike direction \(v\) after \(T = 2000\) iterations, where the hyperparameters are scaled with the dimension, with learning rate \(\gamma = 0.5/d^2\) and \(\lambda = 0.1/d\).

Figure~\ref{icl2} reports the resulting alignment as a function of \(d\). As the dimension increases, we observe a gradual degradation in performance, reflecting the increased difficulty of extracting the principal component in higher-dimensional settings. Nevertheless, the model maintains a significant alignment with the leading eigenvector across all dimensions, illustrating the robustness of the ICL mechanism in the infinite-prompt limit.

\begin{figure}[htbp]
    \centering
    \begin{subfigure}[t]{0.48\linewidth}
        \centering
        \includegraphics[width=\linewidth]{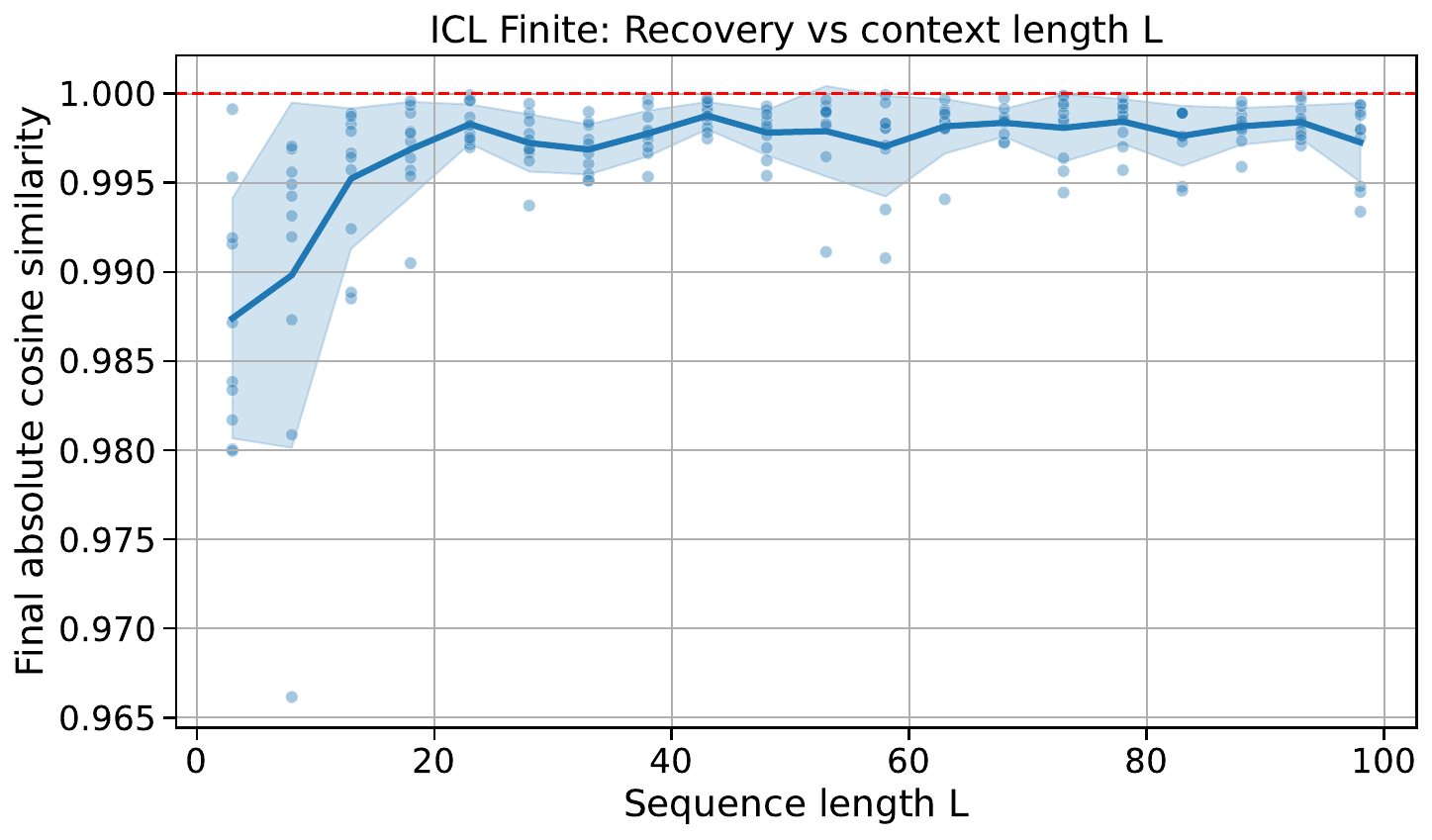}
        \caption{Final alignment as a function of the prompt length $L$.}
        \label{icl1}
    \end{subfigure}
    \hfill
    \begin{subfigure}[t]{0.48\linewidth}
        \centering
        \includegraphics[width=\linewidth]{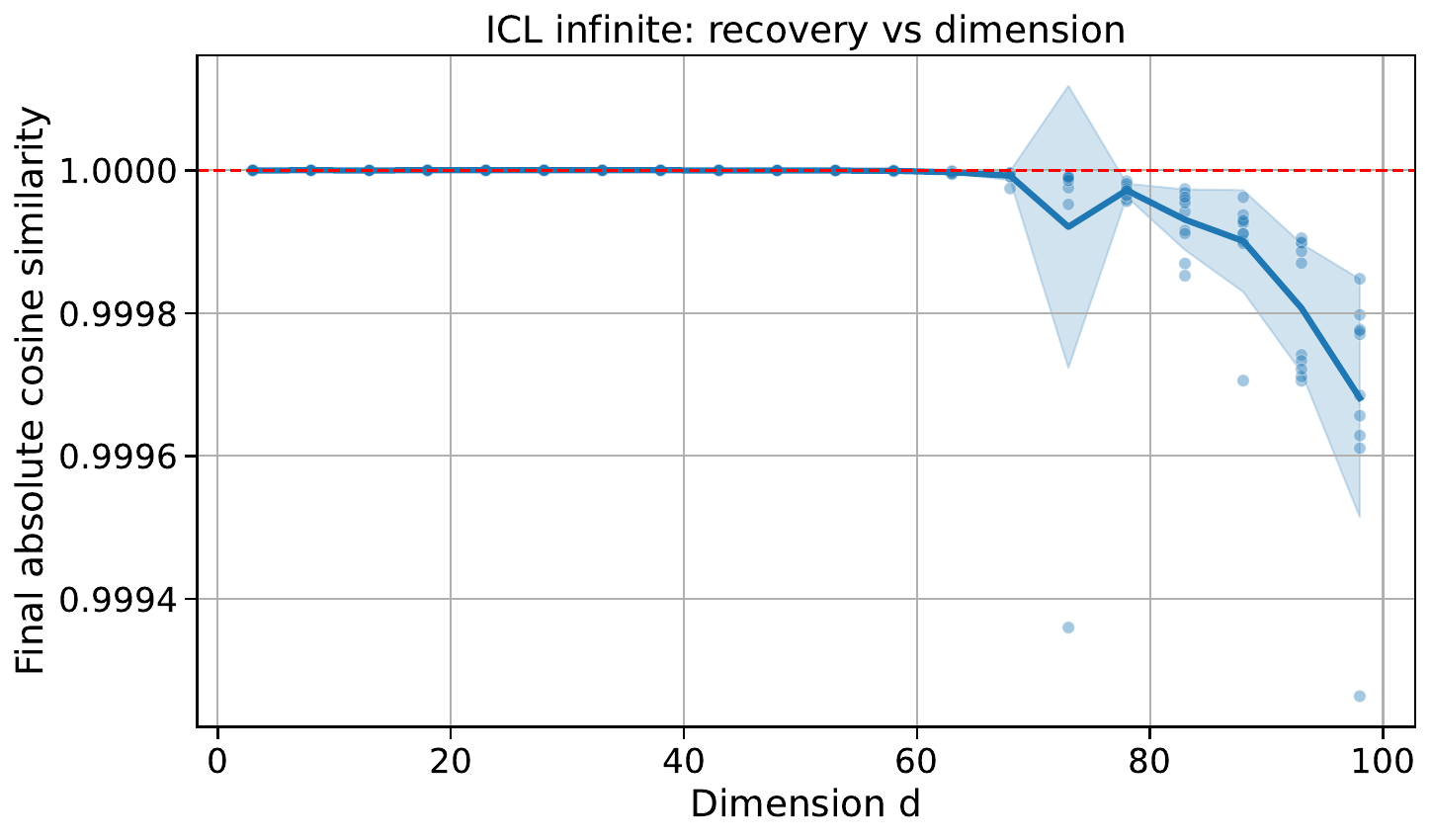}
        \caption{Final alignment as a function of the dimension $d$.}
        \label{icl2}
    \end{subfigure}
    \caption{ICL performance as a function of prompt length and dimension.}
    \label{fig:icl_combined}
\end{figure}

\subsection{Discussion on the numerical experiments}

Overall, the numerical experiments strongly support our theoretical findings. Across all models, we observe consistent recovery of the principal component or spike direction. The experiments highlight the role of finite-prompt effects and illustrate the convergence of finite models toward their corresponding population limits.
In this section, we present numerical experiments that illustrate and empirically validate the theoretical results developed throughout the paper. In particular, we study the convergence behavior of the different models toward the principal eigenvector of the underlying covariance structure, as well as the effect of key parameters such as the prompt length and the ambient dimension. The experiments run in a few minutes on a standard laptop, except for Figure \ref{icl1}, which may take up to an hour due to Monte Carlo sampling of the annealed expectations.

%%%%%%%%%%%%%%%%%%%%%%%%%%%%%%%%%%%%%%%%%%%%%%%%%%%%%%%%%%%%

\end{document}